\documentclass[11pt]{article}
\usepackage[utf8]{inputenc}

\usepackage{enumitem}
\usepackage{t1enc}
\usepackage{dsfont}
\usepackage{amsmath,amsthm,amsfonts,amssymb,bbm,color}

\usepackage[hmargin=2.5cm,vmargin=2.5cm]{geometry}
\usepackage{setspace}
\usepackage{hyperref}

\usepackage{mathtools}
\usepackage{ stmaryrd }

\usepackage{graphicx, psfrag}
\graphicspath{{figures/}}

\usepackage{booktabs}
\usepackage{longtable}

\newtheorem{thm}{Theorem}
\newtheorem{corollary}[thm]{Corollary}
\newtheorem{lemma}[thm]{Lemma}
\newtheorem{prop}[thm]{Proposition}

\newcommand{\bracing}[1]{\ensuremath{\left\{ #1 \right\}}}
\newcommand{\ceil}[1]{\ensuremath{\left\lceil #1 \right\rceil}}

\newcommand{\bbracket}[1]{\ensuremath{\llbracket #1 \rrbracket}}

\newcommand{\real}{\mathbb{R}}
\newcommand{\N}{\mathbb{N}}
\newcommand{\Nzero}{\mathbb{N}_0}
\newcommand{\F}{\mathcal{F}}

\newcommand{\prob}{\mathbb{P}}
\newcommand{\expect}{\mathbb{E}}
\newcommand{\probT}{\mathbb{P}_{T_n}}
\newcommand{\expectT}{\mathbb{E}_{T_n}}

\newcommand{\eps}{\varepsilon}

\newcommand{\X}{\mathcal{X}}
\newcommand{\Xone}{\mathcal{X}_1}
\newcommand{\XN}{\mathcal{X}_N}
\newcommand{\Y}{\mathcal{Y}}

\newcommand{\Aa}{\mathcal{A}}
\newcommand{\Bb}{\mathcal{B}}
\newcommand{\Cc}{\mathcal{C}}
\newcommand{\Dd}{\mathcal{D}}
\newcommand{\Ee}{\mathcal{E}}
\newcommand{\Gg}{\mathcal{G}}
\newcommand{\Mm}{\mathcal{M}}
\newcommand{\Nn}{\mathcal{N}}
\newcommand{\Pp}{\mathcal{P}}
\newcommand{\Sbf}{\mathbf{S}}
\newcommand{\U}{\mathcal{U}}

\def\R{\mathbb{R}} 

\def\1{\mathds{1}}
\def\to{\rightarrow}

\newcommand{\Nonetwo}{[N]\times\bracing{1,2}}

\newcommand{\Teps}{T^{\eps_N}}
\newcommand{\Tneps}{T_n^{\eps_N}}

\numberwithin{equation}{section}
\numberwithin{thm}{section}
\newcommand\numberthis{\addtocounter{equation}{1}\tag{\theequation}}

\title{Genealogy and spatial distribution of the $N$-particle branching random walk with polynomial tails}
\author{Sarah Penington\thanks{Department of Mathematical Sciences, University of Bath, UK} \and
Matthew I.~Roberts\footnotemark[1] \and Zs\'ofia Talyig\'as\footnotemark[1]}
\date{\today}

\begin{document}
\maketitle

\begin{abstract}
The $N$-particle branching random walk is a discrete time branching particle system with selection. We have $N$ particles located on the real line at all times. At every time step each particle is replaced by two offspring, and each offspring particle makes a jump of non-negative size from its parent's location, independently from the other jumps, according to a given jump distribution. Then only the $N$ rightmost particles survive; the other particles are removed from the system to keep the population size constant. Inspired by work of J.~B\'erard and P.~Maillard, we examine the long term behaviour of this particle system in the case where the jump distribution has regularly varying tails and the number of particles is large. We prove that at a typical large time the genealogy of the population is given by a star-shaped coalescent, and that almost the whole population is near the leftmost particle on the relevant space scale.
\end{abstract}
	
\section{Introduction}
\subsection{The $N$-BRW model}\label{sect: NBRW}
We investigate a particle system called $N$-particle branching random walk ($N$-BRW). In this discrete time stochastic process, at each time step, we have $N$ particles located on the real line. We say that the particles at the $n$th time step or at time $n$ belong to the $n$th generation. The locations of the particles change at every time step according to the following rules. Every particle has two offspring. The offspring particles have random independent displacements from their parents' locations, according to some prescribed displacement distribution supported on the non-negative real numbers. Then from the $2N$ offspring particles, only the $N$ particles with the rightmost positions survive to form the next generation. That is, at each time step we have a \textit{branching step} in which the $2N$ offspring particles move, and we have a \textit{selection step}, in which $N$ out of the $2N$ offspring are killed. Ties are decided arbitrarily. We describe the process more formally in Section~\ref{sect: NBRWdef}.

We will use the notation~$[N]:=\bracing{1,\dots,N}$ and $\Nzero := \N\cup\bracing{0}$ throughout. A pair $(i,n)$ with $i\in[N]$ and $n\in\N_0$ will represent the $i$th particle from the left in generation $n$. We also refer to the rightmost particle $(N,n)$ as the \textit{leader} at time $n$. Furthermore, we will denote the locations of the $N$ particles in the $n$th generation by the ordered set of $N$ real numbers
\begin{equation}\label{eq: def_Xn}
\X(n) = \left\{\Xone(n) \leq \dots \leq \XN(n)\right\},
\end{equation}
where $\X_i(n)$ is the location of particle $(i,n)$. We sometimes call $\X(n)$ the \textit{particle cloud}. 

The long term behaviour of the $N$-BRW heavily depends on the tail of the displacement distribution. Motivated by the work of B\'erard and Maillard \cite{Berard2014}, we investigate the $N$-BRW in the case where the displacement distribution is regularly varying, and $N$ is large.

We say that a function $f$ is regularly varying with index $\alpha\in\real$ if for all $y>0$,
\begin{equation}\label{eq: regvar}
\frac{f(xy)}{f(x)} \to y^\alpha  \text{ as } x\to\infty.
\end{equation}
Let $X$ be a random variable and let the function $h$ be defined by
\begin{equation}\label{eq: poly_tail}
\prob(X>x) = \frac{1}{h(x)} \text{ for }x\geq 0.
\end{equation}
We assume throughout that $\prob(X \geq 0) = 1$, that $h$ is regularly varying with index $\alpha > 0$, and that the displacement distribution of the $N$-BRW is given by \eqref{eq: poly_tail}. These are the same assumptions under which the results of \cite{Berard2014} were proved. The reader may wish to think of the particular regularly varying function given by $h(x)=x^\alpha$ for $x\geq 1$ and $h(x)=1$ for $x\in[0,1)$. We do not expect significant change in the behaviour of the $N$-BRW if jumps of negative size are allowed, but we do not prove this; we use the assumption that the jumps are non-negative several times in our argument. 

\subsection{Time and space scales} \label{sect: scalingintro}

Before explaining our main result, we describe the time and space scales we will be working with. We define
\begin{equation}\label{eq: LN}
\ell_N := \ceil{\log_2 N},
\end{equation}
for $N\geq 2$; this is the time scale we will be using throughout. To avoid trivial cases we always assume that $N\geq 2$. The time scale $\ell_N$ is the time it takes for the descendants of one particle to take over the whole population, if none are killed in selection steps.

For the space scale we choose
\begin{equation}\label{eq: aN}
a_N := h^{-1}(2N \ell_N),
\end{equation}
where $h$ is as in \eqref{eq: poly_tail}, and $h^{-1}$ denotes the generalised inverse of $h$ defined by 
\begin{equation}\label{eq: hinv}
h^{-1}(x) := \inf\bracing{y\geq 0:\: h(y) > x}.
\end{equation}
It is worth thinking of the particular case $h(x) = x^\alpha$ for $x\geq 1$, for which we have $a_N = (2N\ell_N)^{1/ \alpha}$ and $h(a_N) = 2N\ell_N$.

With the choice of $a_N$ in \eqref{eq: aN}, for any positive constant $c$, the expected number of jumps which are larger than~$ca_N$ in a time interval of length $\ell_N$ is of constant order, as $N$ goes to infinity. The heuristic picture in \cite{Berard2014} says that jumps of order $a_N$ govern the speed, the spatial distribution, and the genealogy of the population for $N$ large. Besides the main result of \cite{Berard2014} on the asymptotic speed of the particle cloud, it is conjectured that at a typical time the majority of the population is close to the leftmost particle, and that the genealogy of the population is given by a star-shaped coalescent. In this paper we prove these conjectures.

\subsection{The main result (in words)}\label{sect: intro_result}
Stating our main result precisely involves introducing some more notation and defining some rather intricate events. We will do this in Section~\ref{sec:mainresult}. In this section we instead aim to explain the main message of the theorem. When we say `with high probability', we mean with probability converging to 1 as $N\to\infty$.

\vspace{3mm}

\textit{For all $\eta>0$, $M\in\N$ and $t>4\ell_N$, the $N$-BRW has the following properties with high probability:}
\begin{itemize}
	\item \textbf{Spatial distribution:}
	\textit{At time $t$ there are $N-o(N)$ particles within distance $\eta a_N$ of the leftmost particle, i.e.~in the interval $[\Xone(t),\Xone(t)+\eta a_N]$.}
	\item \textbf{Genealogy:} \textit{The genealogy of the population on an $\ell_N$ time scale is asymptotically given by a star-shaped coalescent, and the time to coalescence is between $\ell_N$ and $2\ell_N$.}
	
	\textit{That is, there exists a time~$T\in[t-2\ell_N,t-\ell_N]$ such that with high probability, if we choose $M$ particles uniformly at random at time $t$, then every one of these particles descends from the rightmost particle at time $T$. Furthermore, with high probability no two particles in the sample of size $M$ have a common ancestor after time $T+\eps_N\ell_N$, where $\eps_N$ is any sequence satisfying $\eps_N\to 0$ and $\eps_N\ell_N\to\infty$, as $N\to\infty$.} 
\end{itemize}

The star-shaped genealogy might seem counter-intuitive because every particle has only two descendants. Indeed, if we take a sample of $M>2$ particles at time $t$, and look at the lineages of these particles, they certainly cannot coalesce in one time step. Our result says that all coalescences of the lineages of the sample occur within $o(\ell_N)$ time. Therefore, looking on an $\ell_N$ time scale the coalescence appears instantaneous.

\subsection{Heuristic picture}\label{sect: heur_pic}
We construct our heuristic picture based on the tribe heuristics for the $N$-BRW with regularly varying tails described in \cite{Berard2014}. The tribe heuristics say that at a typical large time there are $N-o(N)$ particles close to the leftmost particle if we look on the $a_N$ space scale. We call this set of particles the \textit{big tribe}. Furthermore, there are \textit{small tribes} of size $o(N)$ to the right of the big tribe. The number of such small tribes is $O(1)$. While the position of the big tribe moves very little on the $a_N$ space scale, the number of particles in the small tribes doubles at each time step. As a result, the big tribe eventually dies out, and one of the small tribes grows to become the new big tribe and takes over the population.

To escape the big tribe and create a new tribe that takes over the population, a particle must make a big jump of order $a_N$. As we explained in Section~\ref{sect: scalingintro}, jumps of this size occur on an $\ell_N$ time scale, and $\ell_N$ is the time needed for a new tribe to grow to a big tribe of size $N$.

Take $t>4\ell_N$. Building on the tribe heuristics, we describe the following picture. Assume that a particle becomes the leader with a big jump of order $a_N$. We claim that this particle will have of order $N$ surviving descendants $\ell_N$ time after the big jump. Moreover, the particle that makes the last such jump before time $t_1:=t-\ell_N$ will be the common ancestor of the majority of the population at time $t$. We denote the generation of this ancestor particle by $T$, and assume that $T\in [t_2,t_1]$. In Figure~\ref{fig: winningTribe} we illustrate how a new tribe is formed at time $T$, and how it grows to a big tribe by time $t$. 
We will prove the main result described in Section~\ref{sect: intro_result} by showing that the picture in Figure~\ref{fig: winningTribe} develops with high probability.

\begin{figure}[h]
\begin{center}
\def\svgwidth{\columnwidth}
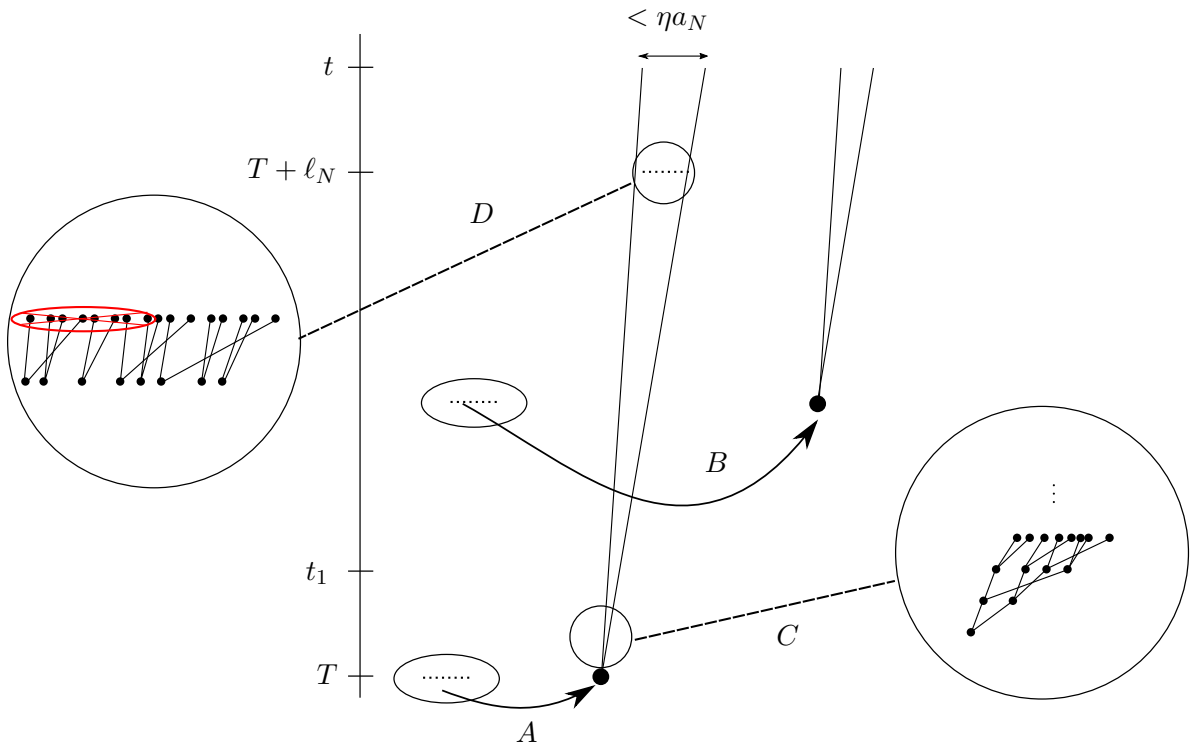
\caption{A particle that makes a big jump of order $a_N$ at time $T$ is the common ancestor of almost the whole population at time $t$. {\small The vertical axis represents time, and the particles' locations are depicted horizontally, increasing from left to right. The black dots represent particles. Horizontal dotted lines in an ellipse or circle show where the majority of the population (the big tribe) is. The arrows represent jumps from the big tribe. We use circles to zoom in on the population. The particles circled in red are killed in the selection step. The events labelled $A$ to $D$ are described in the main text.}}
\label{fig: winningTribe}
\end{center}	
\end{figure}

We introduce the notation
\begin{equation}\label{eq: ti}
t_i := t - i\ell_N,
\end{equation}
for $t,i\in\N$. The message of Figure~\ref{fig: winningTribe}, which we will prove later, is that the following occurs with high probability.

\vspace{3mm}

\noindent
\textbf{A:} At time $T\in [t_2,t_1]$, particle $(N,T)$ has taken a big jump of order $a_N$ and escaped the big tribe. It now leads by a large distance, and its descendants will be the leaders at least until time $t_1$.

There are two main reasons for this. First, we define $T$ as the last time before time $t_1$ when a big jump of order $a_N$ creates a new leader, so particles with big jumps in the time interval $[T,t_1]$ cannot become leaders. Second, particles with smaller jumps not descending from particle $(N,T)$ are unlikely to catch up with the leading tribe, because paths with small jumps move very little on the $a_N$ space scale. This is an important property of random walks with regularly varying tails, which we will state and prove in Lemma~\ref{lemma: durrett} and apply in Corollary~\ref{cor: path}.

\vspace{3mm}

\noindent
\textbf{B:} After time $t_1$, there might be particles which do not descend from particle $(N,T)$, but which, by making a big jump of order $a_N$, move beyond the tribe of particle $(N,T)$. However, these particles have substantially less than $\ell_N$ time to produce descendants by time $t$, and so each of them can only have $o(N)$ descendants at time $t$. Particles which do not descend from $(N,T)$ are unlikely to move beyond the tribe of particle $(N,T)$ without making a big jump.

There will only be $O(1)$ big jumps of order $a_N$ between times $t_1$ and $t$, because jumps of order $a_N$ happen with frequency of order $1/ \ell_N$. Therefore, until time $t$, the total number of particles to the right of the tribe of particle $(N,T)$ is at most $o(N)$.

\vspace{3mm}

\noindent
\textbf{C:} The tribe of particle $(N,T)$ doubles in size at each step up to (almost) time $T+\ell_N$. Selection does not affect these particles significantly, because the number of particles to the right of this tribe is at most $o(N)$ before time $T+\ell_N$, as we explained in part B.

\vspace{3mm}

\noindent
\textbf{D:} At time $T+\ell_N$ there are $N$ particles to the right of the position of particle $(N,T)$. This is an elementary property of the $N$-BRW, following from the non-negativity of the jump sizes. The $N$ particles are mainly in the tribe of particle $(N,T)$, and there may be $o(N)$ particles ahead of the tribe. From this point on, the $N$ leftmost offspring particles in the tribe of particle $(N,T)$ do not survive.

\vspace{3mm}

\noindent
Then, between times $T+\ell_N$ and $t$, the number of particles in the tribe of particle $(N,T)$ will remain $N-o(N)$, where the $o(N)$ part doubles at each time step but does not reach order $N$ by time $t$. Therefore, almost every particle at time $t$ descends from particle $(N,T)$. 

Furthermore, as the number of descendants of particle $(N,T)$ only reaches order $N$ at (roughly) time $T+\ell_N$, the descendants of particle $(N,T)$ are unlikely to make big jumps of order $a_N$ before time $T+\ell_N$. We will prove this property (and many others) in Lemma~\ref{lemma: probC}. Only $O(1)$ descendants of particle $(N,T)$ make big jumps of order $a_N$ between times $T$ and $t$, and these big jumps are likely to happen after time $T+\ell_N$, and so significantly after time $t_1$. Therefore, most time-$t$ descendants of particle $(N,T)$ will not have an ancestor which made a big jump between times $T$ and $t$, thus they will not move far from their ancestor's position $\XN(T)$ on the $a_N$ space scale. 

In order to prove our statements in Section~\ref{sect: intro_result} we also need to show that there is at least one particle which becomes the new leader with a jump of order $a_N$ during the time interval $[t_2,t_1]$. The existence of such a particle will imply that indeed there exists $T\in[t_2,t_1]$ as in Figure~\ref{fig: winningTribe}. We give a heuristic argument for this in Section~\ref{sect: proof_heuristics}, where we also explain the idea for proving that if we take a sample of $M$ particles at time $t$ then the coalescence of the ancestral lineages of these particles happens within a time window of width $o(\ell_N)$.

\subsection{Optimality of our main result}\label{sect: intro_extra}

In order to show that our main result is more or less optimal, we will prove two additional results.

\vspace{3mm}

\noindent
\textbf{Spatial distribution:} Our main theorem says that \emph{most} particles in the population are likely to be within distance $\eta a_N$ of the leftmost at time $t$, for arbitrarily small $\eta>0$ when $N$ and $t$ are large. We will show that this is not true of \emph{all} particles: the distance between the leftmost and rightmost particles is typically of order $a_N$, and is arbitrarily large on the $a_N$ space scale with positive probability. Therefore our result that most particles are close to the leftmost particle on the $a_N$ space scale gives meaningful information on the shape of the particle cloud at a typical time. We state this formally in Proposition~\ref{prop: diameter} and then prove it in Section~\ref{sect: significance}.


\vspace{3mm}

\noindent
\textbf{Genealogy:} Our main theorem says that the generation $T$ of the most recent common ancestor of a sample from the population at time $t$ is between times $t_2$ and $t_1$ with high probability. We will prove that this is the strongest possible result in the sense that for any subinterval of $[t_2,t_1]$ with length of order $\ell_N$ there is a positive probability that $T$ is in that subinterval. This will be the main message of Proposition~\ref{prop: A2prime}, which we prove in Section~\ref{sect: significance}.

We also mention here that the precise statement of our main result, Theorem~\ref{thm}, implies that the distribution of the rescaled time to coalescence, $(t-T)/\ell_N$, has no atom at $1$ or $2$ in the limit $N\to \infty$.


\subsection{Related work}\label{sect: related_work}
The $N$-BRW shows dramatically different behaviours with different jump distributions; this includes the speed at which the particle cloud moves to the right, the spatial distribution within the population, and the genealogy. Below we discuss existing results and conjectures on these properties of the $N$-BRW. We start by summarising the results of B\'erard and Maillard, who studied the speed of the particle cloud when the displacement distribution is heavy-tailed.

\subsubsection*{Heavy-tailed displacement distribution}
B\'erard and Maillard \cite{Berard2014} introduced the stairs process, the record process of a shifted space-time Poisson point process. They proved that it describes the scaling limit of the pair of trajectories of the leftmost and rightmost particles' positions $(\Xone(n),\XN(n))_{n\in\Nzero}$ when the jump distribution has polynomial tails. The correct scaling is to speed up time by $\log_2 N$ and to shrink the space scale by $a_N$. Using the relation between the $N$-BRW and the stairs process they prove their main result: the speed of the particle cloud grows as $a_N/\log_2 N$ in $N$, and the propagation is linear or superlinear (but at most polynomial) in time. The propagation is linear if the jump distribution has finite expectation, and superlinear otherwise; the asymptotics follow from the behaviour of the stairs process. This behaviour is different from that of the classical branching random walk without selection, where the propagation is exponentially fast in time in a heavy-tailed setting \cite{Durrett1983MaximaOB}. 

The tribe heuristics in \cite{Berard2014} predict---but do not prove---that the majority of the population is located close to the leftmost particle, that the genealogy should be star-shaped, and that the relevant time scale for coalescence of ancestral lineages is $\ell_N$. We will prove the above properties in Theorem~\ref{thm}, and therefore the present paper and \cite{Berard2014} together provide a comprehensive picture of the $N$-BRW with regularly varying tails, including the behaviour of the speed, spatial distribution and genealogy.

\subsubsection*{Light-tailed displacement distribution}
Particle systems with selection have been studied with light-tailed displacement distribution in the physics literature as a microscopic stochastic model for front propagation. First Brunet and Derrida \cite{Brunet1997, Brunet2000}, and later Brunet, Derrida, Mueller and Munier \cite{Brunet2007, Brunet2006} made predictions on the behaviour of particle systems with branching and selection.

\textbf{Speed:} For the $N$-BRW, B\'erard and Gou\'er\'e \cite{Berard2010} proved the existence of the asymptotic speed of the particle cloud as time goes to infinity, which in fact applies for any jump distribution with finite expectation. They also proved that the asymptotic speed converges to a finite limiting speed as the number of particles $N$ goes to infinity, with a surprisingly slow rate $(\log N)^{-2}$, which was predicted by Brunet and Derrida \cite{Brunet1997, Brunet2000}. The limiting speed is the same as the speed of the rightmost particle in a classical branching random walk without selection with exponentially decaying tails \cite{hammersley1974, kingman1975, Biggins1976}.

\textbf{Spatial distribution:} The spatial distribution in the light-tailed case is also predicted in \cite{Brunet1997, Brunet2000}. The authors argue that the fraction of particles to the right of a given position at a given time should evolve according to an analogue of the FKPP equation. The FKPP equation is a reaction-diffusion equation admitting travelling wave solutions. 
Rigorous results on the relation between particle systems with selection and free boundary problems with travelling wave solutions have been proved in \cite{durrett2011} and \cite{berestycki2019, demasi2017}.

\textbf{Genealogy:} On the genealogy of the $N$-BRW with light-tailed displacement distribution, the papers \cite{Brunet2007,Brunet2006} arrived at the following conjecture (see also~\cite{maillard2016}). If we pick two particles at random in a generation, then the number of generations we need to go back to find a common ancestor of the two particles is of order $(\log N)^3$. Furthermore, if we take a uniform sample of $k$ particles in a generation and trace back their ancestral lines, the coalescence of their lineages is described by the Bolthausen-Sznitman coalescent, if time is scaled by $(\log N)^3$. This property has been shown for a continuous time model, a branching Brownian motion (BBM) with absorption \cite{berestycki2013}, where particles are killed when hitting a deterministic moving boundary. For the $N$-BRW and its continuous time analogue, the $N$-BBM, no rigorous proof has yet been given. 

\subsubsection*{Displacement distribution with stretched exponential tail}
As we have seen, the behaviour of the $N$-BRW is significantly different in the light-tailed and heavy-tailed cases. It is then a natural question to ask what happens in an intermediate regime, where the jump distribution has stretched exponential tails. Random walks and branching random walks with stretched exponential tails have been investigated in the literature \cite{denisov2008, gantert2000}, but questions about the $N$-BRW with such a jump distribution, such as asymptotic speed, spatial distribution, and genealogy, remain open. In the future we intend to investigate the $N$-BRW in the stretched exponential case.

\subsection{Organisation of the paper}
In Section~\ref{sec:mainresult} we state Theorem~\ref{thm} and Propositions~\ref{prop: A2prime} and \ref{prop: diameter}, our main results, which we have explained in Sections~\ref{sect: intro_result} and \ref{sect: intro_extra}. Furthermore, we give a heuristic argument for the proof of Theorem~\ref{thm}, introduce the notation we will be using throughout, and carry out the first step towards proving Theorem~\ref{thm} in Lemma~\ref{lemma: rewrite}. As a result, the proof of Theorem~\ref{thm} will be reduced to proving Propositions~\ref{prop: A1A3} and \ref{prop: A4}. We prove the former in Sections~\ref{sect: deterministic} and \ref{sect: probabilities} and the latter in Section~\ref{sect: star-shaped}.

In Section~\ref{sect: deterministic} we give a deterministic argument for the existence of a common ancestor between times $t_1$ and $t_2$ of almost the whole population at time $t$. The argument will also imply that almost every particle in the population at time $t$ is near the leftmost particle. Then in Section~\ref{sect: probabilities} we check that the events of the deterministic argument occur with high probability. A key step in the proof is to see that paths cannot move a distance of order $a_N$ in $\ell_N$ time without making at least one jump of order $a_N$. We prove a large deviation result to show this, taking ideas from \cite{Durrett1983MaximaOB} and \cite{gantert2000}. The other important tool, which we will use to estimate probabilities, is Potter's bound for regularly varying functions. 

In Section~\ref{sect: star-shaped} we prove that the genealogy is star-shaped. We will use concentration results from~\cite{McDiarmid1998} to see that a single particle at time $T+\eps_N\ell_N$ cannot have more than of order $N^{1-\eps_N}$ surviving descendants at time $t$, which will be enough to conclude the result.

In Section~\ref{sect: significance} we prove Propositions~\ref{prop: A2prime} and \ref{prop: diameter} using some of our ideas from the deterministic argument in Section~\ref{sect: deterministic}.

Section~\ref{sect: glossary} is a glossary of notation, where we collect the notation most frequently used in this paper with a brief explanation, and with a reference to the section or equation where the notation is defined. In Section~\ref{sect: glossary} we also list the most important intermediate steps of the proof of our main result. 
	
\section{Genealogy and spatial distribution result} \label{sec:mainresult}
\subsection{Formal definition of the $N$-BRW}\label{sect: NBRWdef}
Let $X_{i,b,n}$, $i\in[N]$, $b\in\bracing{1,2}$, $n\in\Nzero$ be i.i.d.~random variables with common law given by \eqref{eq: poly_tail}. Each $X_{i,b,n}$ stands for the jump size of the $b$th offspring of particle $(i,n)$. Let $\X(0) = \{\Xone(0) \le \ldots \le \XN(0)\}$ be any ordered set of $N$ real numbers, which represents the initial locations of the $N$ particles. Now we describe inductively how $\X(0)$ and the random variables $X_{i,b,n}$, $i\in[N]$, $b\in\bracing{1,2}$, $n\in\Nzero$ determine the $N$-BRW, that is, the sequence of locations of the $N$ particles, $(\X(n))_{n\in\Nzero}$.

We start with the initial configuration of particles $\X(0)$. Once~$\X(n)$ has been determined for some~${n\in\Nzero}$, then~$\X(n+1)$ is defined as follows. Each particle has two offspring, each of which performs a jump from the location of its parent. The $2N$ independent jumps at time $n$ are then given by the i.i.d.~random variables $X_{i,b,n}$, $i\in[N]$, $b\in\bracing{1,2}$ as above. After the jumps, only the $N$ rightmost offspring particles survive; that is,
$\X(n+1) = \left\{\Xone(n+1) \leq \dots \leq \XN(n+1)\right\}$ is given by the $N$ largest numbers from the collection~${(\X_i(n) + X_{i,b,n})_{i\in[N],b\in\bracing{1,2}}}$. Ties are decided arbitrarily.

Note that since the jumps are non-negative, the sequences $\X_i(n)$ are non-decreasing in $n$ for all $i\in[N]$. Indeed, at time $n$ there are at least $N-i+1$ particles to the right of or at position $\X_i(n)$, and so there are at least~$\min(N,2(N-i+1))$ particles to the right of or at $\X_i(n)$ at time $n+1$, so we must have $\X_i(n+1)\geq\X_i(n)$. We refer to this property as \emph{monotonicity} throughout.

\subsection{Statement of our main result}\label{sect: statement}
We explained the message of our main result in Section~\ref{sect: intro_result}. In this section we provide the precise statement in Theorem~\ref{thm}. First we introduce the setup for the theorem. 

For~$n,k\in\Nzero$ and $i\in[N]$ we will denote the index of the time-$n$ ancestor of the particle $(i,n+k)$ by
$$\zeta_{i,n+k}(n),$$
i.e.~particle $(\zeta_{i,n+k}(n),n)$ is the ancestor of $(i,n+k)$.
Recall that the relevant space scale for our process is $a_N$, defined in \eqref{eq: aN}. For $r\ge 0$ and $n\in \Nzero$, let $L_{r,N}(n)$ denote the number of particles which are within distance $r a_N$ of the leftmost particle at time $n$:
\begin{equation}\label{eq: LepsN}
L_{r,N}(n) := \max\bracing{i\in[N]:\: \X_i(n) \leq \Xone(n) + r a_N}.
\end{equation}
Define a sequence $(\eps_N)_{N\in \N}$ such that $\eps_N\ell_N$ is an integer for all $N\geq 1$, and which satisfies
\begin{equation}\label{eq: epsN1}
\eps_N\ell_N \to \infty \text{ and } \eps_N\to 0 \text{ as }N\to\infty.
\end{equation}
We introduce two events which describe the spatial distribution and the genealogy of the population at a given time $t$. Our main result, Theorem~\ref{thm}, says that these two events occur with high probability. We define the events for all $N\geq 2$ and $t>4\ell_N$. For $\eta > 0$ and~$\gamma\in(0,1)$, the first event says that at least~$N - N^{1-\gamma}$ particles (i.e.~almost the whole population if $N$ is large) are within distance $\eta a_N$ of the leftmost particle at time $t$. We let
\begin{equation}\label{eq: A1}
\Aa_1 = \Aa_1(t,N,\eta,\gamma):= \bracing{L_{\eta,N}(t)\geq N - N^{1-\gamma}}.
\end{equation}

Recall the notation $t_i$ from \eqref{eq: ti}. We illustrate the second event in Figure~\ref{fig: star-shaped}. We sample $M\in\N$ particles uniformly at random from the population at time $t$. Let $\Pp = (\Pp_1,\dots,\Pp_M)$ be the index set of the sampled particles. The event says that there exists a time $T$ between $t_2$ and $t_1$ such that all of the particles in the sample have a common ancestor at time~$T$, but no pair of particles in the sample have a common ancestor at time $T+\eps_N \ell_N$. Moreover, the common ancestor at time $T$ is the leader particle $(N,T)$. Additionally, the event says that the time $T$ is not particularly close to $t_1$ or $t_2$, in that $T\in[t_2+\ceil{\delta\ell_N},t_1-\ceil{\delta\ell_N}]$ for some $\delta>0$. We let
\begin{align}\label{eq: A2}
\Aa_2 = \Aa_2(t,N,M,\delta) := \bracing{\begin{array}{l}	
\exists T\in[t_2+\ceil{\delta\ell_N},t_1-\ceil{\delta\ell_N}]:\:\zeta_{\Pp_{i},t}(T) = N\;\, \forall i\in[M] \text{ and }
\\
\zeta_{\Pp_{i},t}(T + \eps_N \ell_N) \neq \zeta_{\Pp_{j},t}(T + \eps_N \ell_N)\: \forall i,j\in[M],\; i\neq j
\end{array}}.
\end{align} 
For convenience, we will often write $\Aa_1$ and~$\Aa_2$ for the two events above, omitting the arguments. We will prove the following result.
\begin{thm}\label{thm}
For all $\eta > 0$ and $M\in\N$ there exist $\gamma,\delta\in(0,1)$ such that for all $N\in\N$ sufficiently large and~$t\in\Nzero$ with $t > 4\ell_N$,
\begin{align*}
\prob(\Aa_1\cap\Aa_2) > 1 - \eta,
\end{align*}
where $\ell_N$ is given by \eqref{eq: LN}, and $\Aa_1 = \Aa_1(t,N,\eta,\gamma)$ and $\Aa_2 = \Aa_2(t,N,M,\delta)$ are defined in \eqref{eq: A1} and~\eqref{eq: A2} respectively.
\end{thm}

\begin{figure}[h]
	\begin{center}
\def\svgwidth{0.45\columnwidth}
\begingroup%
  \makeatletter%
  \providecommand\color[2][]{%
    \errmessage{(Inkscape) Color is used for the text in Inkscape, but the package 'color.sty' is not loaded}%
    \renewcommand\color[2][]{}%
  }%
  \providecommand\transparent[1]{%
    \errmessage{(Inkscape) Transparency is used (non-zero) for the text in Inkscape, but the package 'transparent.sty' is not loaded}%
    \renewcommand\transparent[1]{}%
  }%
  \providecommand\rotatebox[2]{#2}%
  \newcommand*\fsize{\dimexpr\f@size pt\relax}%
  \newcommand*\lineheight[1]{\fontsize{\fsize}{#1\fsize}\selectfont}%
  \ifx\svgwidth\undefined%
    \setlength{\unitlength}{640.78680084bp}%
    \ifx\svgscale\undefined%
      \relax%
    \else%
      \setlength{\unitlength}{\unitlength * \real{\svgscale}}%
    \fi%
  \else%
    \setlength{\unitlength}{\svgwidth}%
  \fi%
  \global\let\svgwidth\undefined%
  \global\let\svgscale\undefined%
  \makeatother%
  \begin{picture}(1,0.94429914)%
    \lineheight{1}%
    \setlength\tabcolsep{0pt}%
    \put(0,0){\includegraphics[width=\unitlength,page=1]{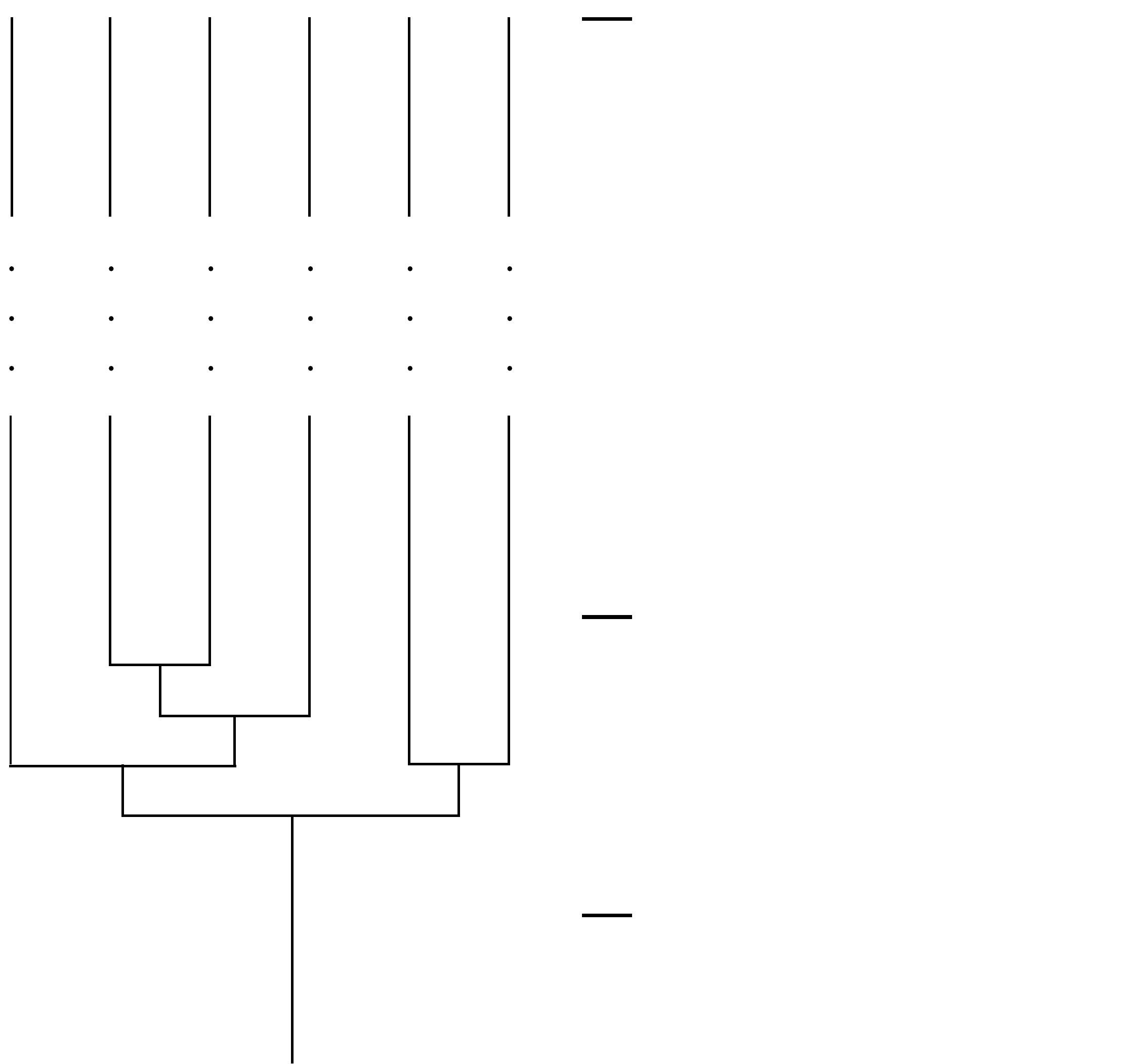}}%
    \put(0.60441824,0.11886713){\makebox(0,0)[lt]{\lineheight{1.25}\smash{\begin{tabular}[t]{l}$T$\end{tabular}}}}%
    \put(0.58752894,0.3915078){\makebox(0,0)[lt]{\lineheight{1.25}\smash{\begin{tabular}[t]{l}$T+\eps_N\ell_N$\end{tabular}}}}%
    \put(0.60815505,0.92238393){\makebox(0,0)[lt]{\lineheight{1.25}\smash{\begin{tabular}[t]{l}$t$\end{tabular}}}}%
    \put(0,0){\includegraphics[width=\unitlength,page=2]{star-shaped.pdf}}%
    \put(0.581309,0.26426446){\makebox(0,0)[lt]{\lineheight{1.25}\smash{\begin{tabular}[t]{l}$o(\ell_N)$\end{tabular}}}}%
    \put(0,0){\includegraphics[width=\unitlength,page=3]{star-shaped.pdf}}%
    \put(0.96847489,0.68777491){\rotatebox{-90}{\makebox(0,0)[lt]{\lineheight{1.25}\smash{\begin{tabular}[t]{l}$t-T\in[\ell_N,2\ell_N]$\end{tabular}}}}}%
    \put(0,0){\includegraphics[width=\unitlength,page=4]{star-shaped.pdf}}%
  \end{picture}%
\endgroup%

\caption{Coalescence of the ancestral lineages of $M=6$ particles. We go backwards in time from top to bottom in the figure. To each particle in the sample we associate a vertical line, representing its ancestral line. Two lines coalesce into one when the particles they are associated with have a common ancestor for the first time going backwards from time $t$. All coalescences of the lineages of the sample happen within a time window of size $o(\ell_N)$. Time $T$ is the generation of the most recent common ancestor of the majority of the whole population at time $t$. The three dots in each line indicate that the picture is not proportional: the time between $t$ and $T$ is of order $\ell_N$, whereas the time between all coalescences and $T$ is $o(\ell_N)$.}
\label{fig: star-shaped}
\end{center}
\end{figure}

We explained two additional results in Section~\ref{sect: intro_extra} which show the optimality of Theorem~\ref{thm}. We state these results precisely below.

We define the event $\Aa_2'$ as a modification of the event $\Aa_2$. Whereas $\Aa_2$ said that the coalescence time $T$ is roughly in $[t_2,t_1]$, the event $\Aa_2'$ says that $T$ is in the smaller interval $[t_2+\ceil{s_1\ell_N},t_2+\ceil{s_2\ell_N}]$ for $0<s_1<s_2<1$; and whereas $\Aa_2$ occurs with high probability, we will show that $\Aa_2'$ occurs with probability bounded away from $0$. For $M\in\N$ and $0 < s_1 < s_2 < 1$, we define
\begin{equation}\label{eq: A2prime}
\Aa_2' = \Aa_2'(t,N,M,s_1,s_2) :=  \bracing{\begin{array}{l}	
	\exists T\in[t_2+\ceil{s_1\ell_N},t_2+\ceil{s_2\ell_N}]:\:\zeta_{\Pp_{i},t}(T) = N\;\, \forall i\in[M] \text{ and }
	\\
	\zeta_{\Pp_{i},t}(T + \eps_N \ell_N) \neq \zeta_{\Pp_{j},t}(T + \eps_N \ell_N)\;\, \forall i,j\in[M],\; i\neq j
	\end{array}}.
\end{equation}

Proposition \ref{prop: A2prime} below says that for all $0 < s_1 < s_2 < 1$ and $r>0$, with probability bounded below by a constant depending on $r$ and $s_2-s_1$, the event $\Aa_2'$ occurs and the diameter at time $t_1$ is at least $ra_N$. 
The diameter of the particle cloud at time $n$ will be denoted by $d(\X(n))$; that is,
\begin{equation} \label{eq:diameterdefn}
d(\X(n)) := \XN(n) - \Xone(n).
\end{equation}

\begin{prop}\label{prop: A2prime}
	For all $0 < s_1 < s_2 < 1$, $M\in\N$ and $r>0$, there exists $\pi_{r,s_2-s_1}>0$ such that for $N$ sufficiently large and $t>4\ell_N$,
	\begin{equation*}
	\prob\left(\Aa_2' \cap \bracing{d(\X(t_1))\geq ra_N}\right) > \pi_{r,s_2-s_1},
	\end{equation*}
	where $\Aa_2'(t,N,M,s_1,s_2)$ is defined in \eqref{eq: A2prime}.
\end{prop}

Our second result about the diameter says that for all $r$, the probability that $d(\X(n)) \geq ra_N$ is bounded away from zero, and it tends to $1$ as $r\to 0$, and tends to $0$ as $r\to\infty$, if $N$ is sufficiently large and $n>3\ell_N$. This shows that the probability that after a long time the diameter is not of order $a_N$ is small, and therefore the part of Theorem~\ref{thm} that says most of the population is within distance $\eta a_N$ of the leftmost particle with high probability, for arbitrarily small $\eta>0$, is meaningful.

\begin{prop}\label{prop: diameter}
	There exist $0<p_r\leq q_r\leq 1$ such that $q_r\to 0$ as $r\to\infty$ and $p_r\to 1$ as $r\to 0$, and for all $r>0$,
	\begin{equation*}
	0 < p_r \leq \prob(d(\X(n))\geq ra_N) \leq q_r,
	\end{equation*}
	for $N$ sufficiently large and $n>3\ell_N$.
\end{prop}

\subsection{Heuristics for the proof of Theorem~\ref{thm}}\label{sect: proof_heuristics}
We first prove a simple lemma which will be helpful in the course of the proof of Theorem~\ref{thm} and also helpful for understanding the heuristics. The lemma says that the number of particles that are to the right of a given position at least doubles at every time step until it reaches $N$. The statement follows deterministically from the definition of the~$N$-BRW. The proof serves as a warm-up for several more deterministic arguments to come. 
For $x\in \R$ and $n\in \N_0$, we write the set of particles to the right of position $x$ at time $n$ as 
\begin{equation}\label{eq: G}
G_x(n) := \bracing{i\in[N]:\: \X_i(n)\geq x}.
\end{equation}
\begin{lemma}\label{lemma: descendants_est}
	Let $x\in\real$ and $n,k\in\mathbb{N}_0$. Then
	\begin{equation*}
	|G_x(n+k)| \geq \min\left(N,\; 2^k|G_x(n)|\right).
	\end{equation*}
\end{lemma}
\begin{proof}
	The statement is clearly true when $G_x(n) = \emptyset$. Now assume that $G_x(n) \neq \emptyset$.
	Let us first consider the case in which every descendant of the particles in $G_x(n)$ survives until time~$n+k$. Since there are $2^k|G_x(n)|$ such descendants, each of which is to the right of $x$ since all jumps are non-negative, in this case we have $|G_x(n+k)| \geq 2^k|G_x(n)|$.
	
	Now let us consider the case in which not every descendant of the particles in $G_x(n)$ survives until time~$n+k$. This means that there exist $m\in[n,n+k-1]$, $j\in [N]$ and $b\in\bracing{1,2}$ such that $(j,m)$ is a descendant of a particle in $G_x(n)$ and
	\begin{equation*}
	\X_j(m) + X_{j,b,m} \leq \Xone(m+1).
	\end{equation*}
	Since particle $(j,m)$ descends from $G_x(n)$, and all jumps are non-negative, we also have $x\leq\X_j(m) + X_{j,b,m}$, and therefore $x\leq \X_1(m+1)\leq \Xone(n+k)$, and the result follows.
\end{proof}

Now we turn to the heuristics for the proof of Theorem~\ref{thm}.  The heuristic picture to keep in mind when thinking about both the statement and the proof is Figure~\ref{fig: winningTribe}. 
As in Section~\ref{sect: heur_pic}, we let $T$ denote the last time at which a
particle takes the lead with a big jump of order $a_N$ before time $t_1$.
In Section~\ref{sect: heur_pic}, we argued that if $T\in [t_2,t_1]$ then with high probability, particle $(N,T)$
will be the common ancestor of almost every particle in the population at time $t$,
and almost the whole population at time $t$ is close to $\mathcal X_N(T)$ on the $a_N$ space scale.
We will use a rigorous version of this heuristic argument to show that the event $\mathcal A_1$ occurs with high probability,
and that the time $T$ satisfies the first line in the event $\Aa_2$ with high probability. That is, every particle from a uniform sample of fixed size $M$ at time $t$ descends from particle $(N,T)$ with high probability.

If $T$ is as described above, then we can only have $T\in [t_2,t_1]$ if there is a particle which takes the lead with a jump of order $a_N$ in the time interval $[t_2,t_1]$. It is not straightforward to show that this happens with high probability. It could be the case that the diameter is large on the $a_N$ space scale during the time interval $[t_2,t_1]$, say greater than $Ca_N$, where $C>0$ is large. In this situation, if the jumps of order $a_N$ in the time interval $[t_2,t_1]$ come from close to the leftmost particle, and they are all smaller than $Ca_N$, then these jumps will not make a new leader, and time $T$ will not be in the time interval $[t_2,t_1]$. We will prove that this is unlikely. A key property which is helpful in seeing this is the following. If no particle takes the lead with a big jump of order $a_N$ for $\ell_N$ time, e.g.~between times $s\in\N$ and $s+\ell_N$, then the diameter of the particle cloud will be very small on the $a_N$ space scale at time $s+\ell_N$. Indeed, all the $N$ particles, including the leftmost, are to the right of position $\XN(s)$ at time $s+\ell_N$ by Lemma~\ref{lemma: descendants_est}. But with high probability, particles cannot move far to the right from this position without making big jumps of order $a_N$. We will prove this in Corollary~\ref{cor: path}. Therefore, provided that no unlikely event happens, if no particle takes the lead with a big jump between times $s$ and $s+\ell_N$, then every particle will be near the position $\XN(s)$ at time $s+\ell_N$. We formally prove this in Lemma~\ref{lemma: noRecord}.

We will be able to use this property for $s=t_2-c'\ell_N$ with small $c'>0$. We will conclude that if no particle takes the lead with a jump of order $a_N$ in the time interval $[t_2-c'\ell_N,t_1-c'\ell_N]$ then the diameter at time $t_1-c'\ell_N$ is likely to be small on the $a_N$ space scale, i.e.~$d(\X(t_1-c'\ell_N))<ca_N$, for some $c>0$ which we can choose to be much smaller than $c'$. If the diameter is less than $ca_N$, then any particle performing a jump larger than $ca_N$ becomes the new leader. 

The expected number of jumps larger than $ca_N$ in $c'\ell_N$ time is $c'\ell_N2Nh(ca_N)^{-1}$, because there are $2N$ jumps at each time step and the jump distribution is given by \eqref{eq: poly_tail}, which is roughly $c'/c^\alpha$ for $N$ sufficiently large. If $c^\alpha$ is much smaller than $c'$, then with high probability there will be a jump of size greater than $ca_N$ in the time interval $[t_1-c'\ell_N,t_1]$, and the particle performing it will become the new leader. Therefore the last time before time $t_1$ when a particle becomes the leader with a jump of order $a_N$ will be after time $t_2$, which gives us $T\in[t_2,t_1]$.

The above idea works for the case where no particle takes the lead with a big jump of order $a_N$ in the time interval $[t_2-c'\ell_N, t_2]$ for some small $c'>0$. If instead there is such a particle then we will argue that in a short interval of length $c'\ell_N$ it is likely that the jump made by this particle will not be too large on the $a_N$ scale and therefore the particle's descendants will be surpassed by larger jumps of order $a_N$ at some point in the much longer time interval $[t_2,t_1]$.

In order to show that the coalescence is star shaped, we also need the second line of the event $\Aa_2$, which says that all coalescences of the lineages of a sample of $M$ particles at time $t$ happen within a time window of size $\eps_N\ell_N$; that is, instantaneously on the $\ell_N$ time scale (see Figure~\ref{fig: star-shaped}).

To prove that no pair of particles in the sample  of $M$ have a common ancestor at time $T + \eps_N \ell_N$, it will be enough to prove that every particle at time $T + \eps_N \ell_N$ has a number of time-$t$ descendants which is at most a very small proportion of the total population size $N$ (we will check this in Lemma~\ref{lemma: rewrite}). With high probability, most of the population at time $t$ descends from the leading $2^{\eps_N\ell_N} \approx N^{\eps_N}$ particles at time $T + \eps_N \ell_N$
(the descendants of particle $(N,T)$). If these particles share their time-$t$ descendants fairly evenly, then a particle in this leading tribe will have roughly $N^{1-\eps_N} = o(N)$ descendants. Indeed, we will prove using concentration results from \cite{McDiarmid1998} that with high probability the number of time-$t$ descendants of a particle from the leading tribe at time $T + \eps_N \ell_N$ will not exceed the order of $N^{1-\eps_N}$.

\subsection{Notation}\label{sect: notation}
We now introduce the notation we will be using throughout the proof of Theorem~\ref{thm}. We define the filtration $(\F_n)_{n\in\Nzero}$ by letting $\F_n$ be the $\sigma$-algebra generated by the random variables~$(X_{i,b,m},\: i\in[N], b\in\bracing{1,2}, m<n)$ from Section~\ref{sect: NBRWdef}. Since $\X(n)$ is defined in such a way that it only depends on jumps performed before time $n$, the process $(\X(n))_{n\in \Nzero}$ is adapted to the filtration $(\F_n)_{n\in\Nzero}$. 
Since $(X_{i,b,m},\: i\in[N], b\in\bracing{1,2}, m\in \N_0)$ are i.i.d.,
the jumps $(X_{i,b,n}, i\in [N], b\in \{1,2\})$ are independent of the $\sigma$-algebra $\F_n$. In Theorem~\ref{thm} we assume that $t > 4\ell_N$, as in the proof we will examine the process in the time interval $[t_4,t]$, where $t_4$ is given by \eqref{eq: ti}. Since jumps at time $t$ are not $\F_t$-measurable, we will be interested in jumps performed in the time interval $[t_4,t-1]$.

The jump of the $i$th particle's $b$th offspring at time $n$ will be referred to using the random variable $X_{i,b,n}$, or the triple $(i,b,n)$.
In order to study the genealogy of the $N$-BRW particle system, we will need a notation which says that two particles are related. Let us introduce the partial order $\lesssim$ on the set of pairs~$\bracing{(i,n), i\in[N], n\in\Nzero}$. First, for $i\in [N]$ and $n\in \Nzero$ we say that~${(i,n) \lesssim (i,n)}$ and, for $j\in[N]$, we write $(i,n) \lesssim (j,n+1)$ if and only if the $j$th particle at time $n+1$ is an offspring of the $i$th particle at time $n$. Then in general, for $n,k\in\Nzero$ and $i_0,i_k\in[N]$ we write $(i_0,n) \lesssim (i_k,n+k)$ if and only if particle $(i_k,n+k)$ is a descendant of particle $(i_0,n)$:
\begin{equation}\label{eq: ancestor}
(i_0,n) \lesssim (i_k,n+k) \quad \Longleftrightarrow \quad \exists i_1,\dots,i_{k-1}:\: (i_{j-1},n+j-1) \lesssim (i_j,n+j),\quad \forall j\in[k].
\end{equation}
Then the particles $( (i_j,n+j)$, $j\in[k])$ represent the ancestral line between $(i_0,n)$ and $(i_k,n+k)$. Recall that for $n,k\in\Nzero$ and $i\in[N]$ we denote the index of the time-$n$ ancestor of the particle $(i,n+k)$ by $\zeta_{i,n+k}(n)$. Thus, using our partial order above, we can write for $j\in[N]$,
\begin{equation}\label{eq: zeta}
\zeta_{i,n+k}(n) = j \quad \Longleftrightarrow \quad (j,n) \lesssim (i,n+k).
\end{equation}

We also introduce a slightly different (strict) partial order $\lesssim_b$, which will be convenient later on. For $i_0,i_k\in[N]$, $n\in\Nzero$ and $k\in\N$ we write~$(i_0,n) \lesssim_b (i_k,n+k)$ if and only if the $b$th offspring of particle $(i_0,n)$ is the time-$(n+1)$ ancestor of particle $(i_k,n+k)$. Note that if $(i_0,n) \lesssim_b (i_k,n+k)$ then there exists $i_1\in[N]$ such that
\begin{equation*}
\X_{i_1}(n+1) = \X_{i_0}(n) + X_{i_0,b,n} \text{ and }
(i_1,n+1) \lesssim (i_k,n+k).
\end{equation*}

Using the above partial order, we define the \textit{path} between particles $(i_0,n)$ and $(i_k,n+k)$ (and between positions $\X_{i_0}(n)$ and $\X_{i_k}(n+k)$), as the sequence of jumps connecting the two particles. For $i_0,i_k\in[N]$ and $n\in\Nzero$, if $k\in \N$ and $(i_0,n)\lesssim (i_k,n+k)$, we let
\begin{equation}\label{eq: P}
P_{i_0,n}^{i_k,n+k}  := \bracing{(i_j,b_j,n+j):\: j\in\bracing{0,\dots,k-1} \text{ and } (i_j,n+j)\lesssim_{b_j}(i_k,n+k)},
\end{equation}
and we let $P_{i_0,n}^{i_k,n+k} := \emptyset$ otherwise. Then if $k\in \N$ and $(i_0,n)\lesssim (i_k,n+k)$,
\begin{equation} \label{eq:pathsum}
\X_{i_k}(n+k) = \X_{i_0}(n) + \sum_{(j,b,m)\in P_{i_0,n}^{i_k,n+k}} X_{j,b,m}.
\end{equation}

For $i\in [N]$ and $n,k \in \Nzero$ with $n\le k$, let $\Nn_{i,n}(k)$ denote the set of descendants of particle $(i,n)$ at time $k$:
\begin{equation}\label{eq: N}
\Nn_{i,n}(k) := \bracing{j \in [N]:\: (i,n)\lesssim (j,k)},
\end{equation}
and if $n<k$, for $b\in \{1,2\}$, let $\Nn_{i,n}^b(k)$ be the set of time-$k$ descendants of the $b$th offspring of particle $(i,n)$:
\begin{equation}\label{eq: Nb}
\Nn_{i,n}^b(k) := \bracing{j \in [N]:\: (i,n)\lesssim_b (j,k)}.
\end{equation}
(Note that the sets $\Nn_{i,n}(k)$ and $\Nn_{i,n}^b(k)$ may be empty.)
We write $|\Nn_{i,n}(k)|$ and $|\Nn_{i,n}^b(k)|$ for the number of descendants in each case. 

Finally, as time is discrete, it will be useful to introduce a notation for the set of integers in an interval; for $0\leq s_1\leq s_2$, we let
\begin{equation*}
\bbracket{ s_1,s_2} := [s_1,s_2]\cap\Nzero.
\end{equation*}

\subsection{Big jumps and breaking the record}\label{sect: bigjumps}
As discussed in Section~\ref{sect: heur_pic}, the common ancestor of the majority of the population at time $t$ is a particle which made an unusually big jump, of order $a_N$, between times $t_2$ and $t_1$. 
The set of unusually big jumps will play an essential role in the proof of Theorem~\ref{thm}. We will be particularly interested in particles which become `leaders' after performing such jumps. These particles are the candidates to become the common ancestor of almost the whole population at time $t$.

We now introduce the necessary notation for the above concepts. In the definitions we will indicate the dependence on a new parameter $\rho \in (0,1)$, as the choice of~$\rho$ will be important later on. Furthermore, everything we define will depend on $N$ and $t$, which we do not always indicate.

For $\rho\in(0,1)$ we introduce the term \textit{big jump} for jumps of size greater than $\rho a_N$, and we denote the set of big jumps on an interval~${[s_1,s_2]\subseteq [t_4,t-1]}$ by $B_N^{[s_1,s_2]}$:
\begin{equation}\label{eq: BN1}
B_N^{[s_1,s_2]} = B_N^{[s_1,s_2]}(\rho) := \bracing{(k,b,s)\in\Nonetwo\times \bbracket{s_1,s_2}:\: X_{k,b,s} > \rho a_N},
\end{equation}
where $a_N$ is given by \eqref{eq: aN}. We also let
\begin{equation}\label{eq: BN2}
B_N := B_N^{[t_4,t-1]}.
\end{equation}

We say a particle \textit{breaks the record} if it takes the lead with a big jump. If one of the current leader's descendants makes a small jump (that is, a non-big jump) to become the leader, then that does not count as breaking the record in our terminology. Let $\Sbf_N$ denote the set of times when the record is broken by a big jump between times $t_4$ and $t$:
\begin{equation}\label{eq: RN}
\Sbf_N = \Sbf_N(\rho) := \bracing{\begin{array}{l}
s\in \bbracket{t_4,{t-1}}:\: \exists (k,b)\in\Nonetwo \text{ such that }
\\
(k,s) \lesssim_b (N,s+1) \text{ and } X_{k,b,s} > \rho a_N
\end{array}}.
\end{equation}

Next, we define $T$ as the last time when the leader broke the record with a big jump before time~$t_1$, if there is any such time. We let
\begin{equation}\label{eq: T}
T = T(\rho) :=
1 + \max \bracing{\Sbf_N(\rho) \cap [t_4,t_1-1]},
\end{equation}
and let $T=0$ if $\Sbf_N(\rho) \cap [t_4,t_1-1] = \emptyset$. Note that the big jump which takes the lead happens at time~$T-1$, and $T$ is the time right after the jump. In the proof it turns out that with high probability, $T\in[t_2+\ceil{\delta \ell_N},t_1-\ceil{\delta \ell_N}]$ for some $\delta>0$, and particle $(N,T)$ is the common ancestor of almost the whole population at time~$t$.

We will have a separate notation, $\hat \Sbf_N$, for the times when the leader is surpassed by a particle which performs a big jump. Note that this is not exactly the same set of times as $\Sbf_N$: it might happen that a particle $(i,s)$ has an offspring $(j,s+1)$, which beats the current leader $(N,s)$ with a big jump, but it does not become the next leader at time $s+1$ because it is beaten by another offspring particle which did not make a big jump. We define
\begin{equation}\label{eq: RNhat}
\hat \Sbf_N = \hat \Sbf_N(\rho) := \bracing{\begin{array}{l}
s\in \bbracket{t_4,{t-1}}:\: \exists (k,b)\in\Nonetwo \text{ such that } \\ X_{k,b,s} > \rho a_N \text{ and }
\X_k(s) + X_{k,b,s} > \XN(s)
\end{array}}.
\end{equation}
We will see in Corollary~\ref{cor: beatingLeader} below
that with high probability, $\Sbf_N$ and $\hat \Sbf_N$ coincide on certain time intervals. Sometimes we will  also need to refer to the set of times when big jumps do not take the lead or beat the current leader. Therefore, with a slight abuse of notation, we will write $\Sbf_N^c$ and~$\hat \Sbf_N^c$ to denote the sets of times $\bbracket{t_4,t-1}\setminus \Sbf_N$ and $\bbracket{t_4,t-1}\setminus \hat \Sbf_N$ respectively.

\subsection{Reformulation}
In this section, we break down the event $\Aa_2$ of Theorem~\ref{thm}. Our ultimate goal is to show, for a suitable choice of $\rho$, that~$T=T(\rho)$, as defined in~\eqref{eq: T}, has the properties required in $\Aa_2$. To this end we introduce new events which imply $\Aa_2$ with high probability, and only involve $T$ and the number of time-$t$ descendants of particle~$(N,T)$ and of the particles at time~$T+\eps_N\ell_N$. We will use the following notation:
\begin{equation}\label{eq: Teps}
\Teps = \Teps(\rho) := T(\rho) + \eps_N \ell_N,
\end{equation}
where $\eps_N$ is defined in \eqref{eq: epsN1}. Recalling~\eqref{eq: N}, for $i\in [N]$, we write
\begin{equation}\label{eq: Ni}
\Nn_i := \Nn_{i,\Teps}(t)
\end{equation}
for the set of time-$t$ descendants of the $i$th particle at time $\Teps$, and
\begin{equation}\label{eq: Di}
D_i = D_{i,\Teps}(t):= |\Nn_{i,\Teps}(t)|
\end{equation}
for the size of this set. 

For $\gamma,\delta,\rho\in (0,1)$, we introduce the event
\begin{equation}\label{eq: A3}
\Aa_3 = \Aa_3(t,N,\delta,\rho,\gamma) := \bracing{T(\rho)\in[t_2+\ceil{\delta \ell_N},t_1-\ceil{\delta \ell_N}]}\cap\bracing{|\Nn_{N,T(\rho)}(t)| \geq N - N^{1-\gamma}}.
\end{equation}
This event says that almost the whole population at time $t$ descends from particle $(N,T)$, which will imply with high probability that each particle in the uniform sample of $M$ particles in the event $\Aa_2$ is a descendant of $(N,T)$.
The final part of the definition of the event $\Aa_2$ says that no two particles at time $t$ in the uniform sample of $M$ particles share an ancestor at time $\Teps$. We now define an event which says that every time-$\Teps$ particle has at most a very small proportion of the $N$ surviving descendants at time $t$, so that with high probability none of them have two descendants in the sample of $M$ particles. For $\nu > 0$ and $\rho\in(0,1)$, we let
\begin{equation}\label{eq: A4}
\Aa_4(\nu) = \Aa_4(t,N,\rho,\nu) := \bracing{\max_{i\in\Nn_{N,T}(\Teps)}D_{i,\Teps}(t) \leq \nu N}.
\end{equation}
Note that in the definition of $\Aa_4(\nu)$ we take the maximum only over the time-$\Teps$ descendants of particle $(N,T)$. It will be easy to deal with the remaining particles at time $\Teps$, because the event $\Aa_3$ implies that for $\nu>0$, if $N$ is large, particles not descended from $(N,T)$ cannot have more than $\nu N$ descendants at time $t$. In the following result, we reduce the proof of Theorem~\ref{thm} to showing that $\Aa_1$, $\Aa_3$ and $\Aa_4(\nu)$ occur with high probability. 

As part of the proof we show that the probability that two particles in the sample of $M$ at time $t$ have a common ancestor at time $\Teps$ can be upper bounded by little more than the sum of the probabilities of the events $\Aa_3^c$ and $\Aa_4(\nu)^c$ when $\nu$ is small. We will use this intermediate result in another argument later on in Section~\ref{sect: significance}, so we state it as part of Lemma~\ref{lemma: rewrite} below.

\begin{lemma}\label{lemma: rewrite}
Take $M\in\N$ and $\gamma, \delta, \rho , \eta \in (0,1)$, and let $0< \nu < \eta / M^2$. Then for all $N$ sufficiently large and $t>4\ell_N$,
\begin{equation*}
\prob(\exists j,l\in[M],\: j\neq l:\:\zeta_{\Pp_{j},t}(\Teps) = \zeta_{\Pp_{l},t}(\Teps)) \leq \prob(\Aa_3^c) + \prob(\Aa_4(\nu)^c) + \eta/2,
\end{equation*}
and
\begin{equation*}
\prob(\Aa_2^c) \leq 2\prob(\Aa_3^c) + \prob(\Aa_4(\nu)^c) + \eta,
\end{equation*}
where $\Aa_2(t,N,M,\delta)$, $\Aa_3(t,N,\delta,\rho,\gamma)$ and $\Aa_4(t,N,\rho,\nu)$ are defined in \eqref{eq: A2}, \eqref{eq: A3} and \eqref{eq: A4} respectively, $\Pp_j$ is the index of a particle in the uniform sample of $M$ particles at time $t$, and $\zeta_{\Pp_{j},t}(\Teps)$ is the index of the time-$\Teps$ ancestor of particle $(\Pp_{j},t)$, defined in \eqref{eq: zeta}.
\end{lemma}

\begin{proof}
Fix $M\in\N$ and $\gamma, \delta, \rho , \eta \in (0,1)$. 
Note that by the definition of $\Aa_2$ in~\eqref{eq: A2},
\begin{multline*}
\bracing{T\in[t_2+\ceil{\delta \ell_N},t_1-\ceil{\delta \ell_N}]}\cap\bracing{\zeta_{\Pp_{j},t}(T) = N\: \forall j\in[M]}\\
\cap
\bracing{
\zeta_{\Pp_{j},t}(\Teps) \neq \zeta_{\Pp_{l},t}(\Teps)\: \forall j,l\in[M],\: j\neq l}
\subseteq \Aa_2.\numberthis\label{eq: inA2}
\end{multline*}
First we aim to show that for $N$ sufficiently large,
\begin{equation}\label{eq: A3c}
\prob(\bracing{T\notin[t_2+\ceil{\delta \ell_N},t_1-\ceil{\delta \ell_N}]}\cup\bracing{\exists j\in[M]:\:\zeta_{\Pp_{j},t}(T) \neq N}) \leq \prob(\Aa_3^c) + \eta/2.
\end{equation}
Note that if $\Aa_3$ occurs then $T\in[t_2+\ceil{\delta \ell_N},t_1-\ceil{\delta \ell_N}]$, and $\Aa_3$ is $\mathcal F_t$-measurable, so
\begin{multline*}
\prob(\bracing{T\notin[t_2+\ceil{\delta \ell_N},t_1-\ceil{\delta \ell_N}]}\cup\bracing{\exists j\in[M]:\:\zeta_{\Pp_{j},t}(T) \neq N}) 
\\
\leq\: 
\expect\left[ \mathds{1}_{\Aa_3} \prob(\exists j\in[M]:\:\zeta_{\Pp_{j},t}(T) \neq N\: |\: \F_t) \right]
+ \prob\left(\Aa_3^c\right).
\end{multline*}
Now, on the event $\Aa_3$, at most $N^{1-\gamma}$ time-$t$ particles are not descended from $(N,T)$, and therefore a union bound on the uniformly chosen sample (which is not $\F_t$-measurable) gives that the above is at most $M N^{1-\gamma}/N +  \prob\left(\Aa_3^c\right)$. This implies \eqref{eq: A3c} for~$N$ sufficiently large.

Now fix $\nu\in(0,\eta / M^2)$. Our second step is to prove that for $N$ sufficiently large,
\begin{align}\label{eq: A4c}
\prob(\exists j,l\in[M],\: j\neq l:\:\zeta_{\Pp_{j},t}(\Teps) = \zeta_{\Pp_{l},t}(\Teps)) \leq \prob(\Aa_3^c) + \prob(\Aa_4(\nu)^c) + \eta /2,
\end{align}
which is the first part of the statement of the lemma. The event on the left-hand side means that there is a particle at time $\Teps$ which has at least two descendants in the sample of $M$ particles at time $t$. That is
\begin{equation*}
\prob(\exists j,l\in[M],\: j\neq l:\:\zeta_{\Pp_{j},t}(\Teps) = \zeta_{\Pp_{l},t}(\Teps)) =
\prob\left( \exists i\in[N],\; j,l\in[M],\;j\neq l:\: \bracing{\Pp_j,\Pp_l}\subseteq \Nn_i\right).
\numberthis\label{eq: epsNcommonAncestor}
\end{equation*}

We will use that if all the $\Nn_i$ sets have size smaller than $\nu N$ then it is unlikely that two particles of the uniformly chosen sample will fall in the same $\Nn_i$ set. Since $D_i$ is $\mathcal F_t$-measurable for all $i$, a union bound gives
\begin{multline*}
\prob\left( \exists i\in[N],\; j,l\in[M],\;j\neq l:\: \bracing{\Pp_j,\Pp_l}\subseteq \Nn_i\right)
\\ 
\leq\:
\expect\bigg[ \mathds{1}_{\bracing{\max_{i\in[N]} D_i \leq \nu N}} \sum_{i=1}^N \sum_{\:\:1\leq j < l \leq M} \prob(\bracing{\Pp_j,\Pp_l}\subseteq \Nn_i\: |\: \F_t) \bigg]
+ \prob\left(\max_{i\in[N]} D_i > \nu N\right).\numberthis\label{eq: lemmaDi1}
\end{multline*}
Since the sample is chosen uniformly at random, the first term on the right-hand side is equal to
\begin{align*}
\expect\left[ \mathds{1}_{\bracing{\max_{i\in[N]} D_i \leq \nu N}} \sum_{i=1}^N {M \choose 2}  \frac{{D_i \choose 2}}{{N \choose 2}} \right] 
&\leq\:\expect\left[ \mathds{1}_{\bracing{\max_{i\in[N]} D_i \leq \nu N}} \max_{j\in[N]} D_j {M \choose 2} \frac{\sum_{i=1}^N D_i}{N(N-1)} \right]
\\ &
\leq\:
{M \choose 2}\frac{\nu N}{N-1}, \numberthis\label{eq: lemmaDi2}
\end{align*}
where in the second inequality we exploit the indicator and use that $\sum_{i=1}^N D_i =N$. In order to deal with the second term on the right-hand side of \eqref{eq: lemmaDi1}, note that the maximum is taken over all particles at time $\Teps$ (because of the definition of $D_i$ in \eqref{eq: Di}). 
Suppose $N$ is sufficiently large that $N^{1-\gamma}\le \nu N$. Then if the event $\Aa_3$ occurs, particles not descended from particle $(N,T)$ (i.e.~particles not in $\Nn_{N,T}(\Teps)$) have at most $\nu N$ descendants at time $t$. Therefore, by the definition of $\Aa_4(\nu)$,
\begin{equation}
\prob\left(\max_{i\in[N]} D_i > \nu N\right) \leq \prob(\Aa_4(\nu)^c) + \prob\left(\max_{i\in[N]\setminus\Nn_{N,T}(\Teps)} D_i > \nu N\right) \leq \prob(\Aa_4(\nu)^c) + \prob(\Aa_3^c),\numberthis\label{eq: A4cA3c}
\end{equation}
for $N$ sufficiently large.

Putting~\eqref{eq: epsNcommonAncestor}-\eqref{eq: A4cA3c} together, since we chose $\nu<\eta/M^2$ we have that \eqref{eq: A4c} holds for $N$ sufficiently large. By \eqref{eq: inA2}, \eqref{eq: A3c} and \eqref{eq: A4c}, the result follows.
\end{proof}

We now state the two main intermediate results in the proof of Theorem~\ref{thm}, which say
that, for well-chosen $\gamma$, $\delta$, and $\rho$, the events $\Aa_1$, $\Aa_3$ and~$\Aa_4(\nu)$ occur with high probability. In Sections~\ref{sect: deterministic} and \ref{sect: probabilities} we give the proof of Proposition~\ref{prop: A1A3}, and in Section~\ref{sect: star-shaped} we prove Proposition~\ref{prop: A4}.

\begin{prop}\label{prop: A1A3}
For $\eta\in(0,1]$ there exist $0<\gamma<\delta<\rho<\eta$ such that for $N$ sufficiently large and $t>4\ell_N$,
\begin{equation*}
\prob(\Aa_1\cap\Aa_3) > 1 - \eta,
\end{equation*}
where $\Aa_1(t,N,\eta,\gamma)$ and $\Aa_3(t,N,\delta,\rho,\gamma)$ are defined in \eqref{eq: A1} and \eqref{eq: A3} respectively.
\end{prop}
\begin{prop}\label{prop: A4}
Let $\eta\in(0,1]$ and $\nu>0$. Then for $\rho \in (0, \eta )$ as in Proposition~\ref{prop: A1A3}, for $N$ sufficiently large and~${t > 4\ell_N}$,
\begin{equation*}
\prob(\Aa_4(\nu)) > 1 - 2\eta,
\end{equation*}
where $\Aa_4(t,N,\rho,\nu)$ is defined in \eqref{eq: A4}.
\end{prop}

\begin{proof}[Proof of Theorem~\ref{thm}]
Lemma~\ref{lemma: rewrite}, Proposition~\ref{prop: A1A3} and Proposition~\ref{prop: A4} immediately imply Theorem~\ref{thm}.
\end{proof}

\subsection{Strategies for the proofs of Propositions~\ref{prop: A1A3} and \ref{prop: A4}}
Our strategy for the proof of Proposition~\ref{prop: A1A3} is based on the picture in Figure~\ref{fig: winningTribe}. For $t>4\ell_N$, 
we will show that the following happens between times $t_2$ and $t$ with probability close to 1.
\begin{enumerate}
	\item There will be particles which lead by a large distance at times in $[t_2,t_1]$. The last such particle will be at time~${T\in[t_2+\ceil{\delta \ell_N},t_1-\ceil{\delta \ell_N}]}$ with position $\XN(T)$.
	
	\item The descendants of this particle are close together and far away from the the rest of the population at time $t_1$, forming a small (size $o(N)$) leader tribe.
	
	\item At time $t$, the descendants of the small leader tribe from time $t_1$ form a big tribe of $N-o(N)$ particles, which descend from particle $(N,T)$ and are close to the leftmost particle. 
\end{enumerate}

The first part of the proof is a deterministic argument given in Section~\ref{sect: deterministic}, which shows that if `all goes well' between times $t_4$ and $t$, then steps 1-2-3 above roughly describe what happens, which will imply that the events~$\Aa_1$ and $\Aa_3$ in Proposition~\ref{prop: A1A3} occur. For the deterministic argument we will introduce a number of events, which will describe sufficient criteria for $\Aa_1$ and $\Aa_3$ to happen. Once we have shown that the intersection of these events is contained in $\Aa_1\cap\Aa_3$, it is enough to prove that the probability of this intersection is close to 1. This part will be carried out in Section~\ref{sect: probabilities}, and consists of checking that `all goes well' with high probability.

We describe our strategy for showing Proposition~\ref{prop: A4} in detail in Section~\ref{sect: star-shaped-strategy}. The main idea is to give a lower bound on the position of the leftmost particle at time $t$ with high probability, and then use concentration inequalities from \cite{McDiarmid1998} to bound the number of time-$t$ descendants of each particle in $\Nn_{N,T}(\Teps)$ which can reach that lower bound by time $t$. 
A key intermediate step will be to see that with high probability, particles can reach the lower bound only if they have an ancestor which made a jump larger than a certain size. 

\section{Deterministic argument for the proof of Proposition~\ref{prop: A1A3} }\label{sect: deterministic}
In this section we provide the main component of the proof of Proposition~\ref{prop: A1A3}. We follow the plan explained in the previous section; we define new events and show that they imply~$\Aa_1$ and~$\Aa_3$. In Section~\ref{sect: probabilities}, we will prove that the new events occur with high probability. The events describe a strategy designed to make sure that the majority of the population at time $t$ has a common ancestor at some time between $t_2$ and $t_1$; that is, to ensure that $\Aa_3$ occurs. The strategy will also show that most of the particles descended from particle $(N,T)$ cannot move too far from position $\XN(T)$ by time $t$. Thus it will be easy to see that these descendants are near the leftmost particle at time $t$, and so $\Aa_1$ must occur. So although the strategy is designed for the event~$\Aa_3$, it will imply $\Aa_1$ too.

In the course of the proof we will use several constants. We first give a guideline, which shows how the constants should be thought of throughout the rest of the paper, then we describe the specific assumptions we need for the rest of this section. Recall that we fixed $\alpha>0$ as in \eqref{eq: poly_tail} and that we have $\eta\in(0,1]$ from the statement of Proposition~\ref{prop: A1A3}. The other constants can be thought of as
\begin{equation}\label{eq: const}
0<\gamma < \delta \ll \rho \ll  c_1  \ll c_2 \ll c_3 \ll c_4 \ll c_5 \ll c_6 \ll \eta < 1 \quad \text{ and } \quad K\gg \rho^{-\alpha}.
\end{equation}
As everything is constant in \eqref{eq: const}, we only use $\ll$ as an informal notation to say that the left-hand side is much smaller than the right-hand side. 

More specifically, for the rest of this section we fix the constants $\gamma,\delta,\rho, c_1 ,c_2,\dots,c_6,\eta$ and $K$, and assume that they satisfy
\begin{equation}\label{eq: const1}
0 < \gamma < \delta < \rho,
\end{equation}
\begin{equation}\label{eq: const2}
10\rho <  c_1 ,
\end{equation}
\begin{equation}\label{eq: const4}
10c_j < c_{j+1}  < \eta < 1, \quad j = 1,\dots,5,
\end{equation}
\begin{equation}\label{eq: const5}
K > \rho^{-\alpha}.
\end{equation}
We will have additional conditions on these constants in Section~\ref{sect: probabilities}, which will be consistent with the assumptions \eqref{eq: const1}-\eqref{eq: const5}.

Every event we introduce below will depend on $N$, $t$ (with $t>4\ell_N$) and on some of the constants above. In the definitions we will not indicate this dependence explicitly. Note furthermore that in the statement of Proposition~\ref{prop: A1A3}, taking $N$ sufficiently large may depend on $\gamma$, $\delta$, or $\rho$. 

\subsection{Breaking down event $\Aa_3$}
We begin by breaking down the event $\Aa_3$ from Proposition~\ref{prop: A1A3} into two other events. Then we will define a strategy for showing that these two events occur. The first event describes the particle system at time~$t_1$; it says that there is a small leader tribe of size less than~$2N^{1-\delta}$, and every other particle is at least~$c_2a_N$ to the left of this tribe. Moreover, each particle in the leading tribe descends from the same particle,~$(N,T)$. The common ancestor $(N,T)$ is the last particle which breaks the record with a big jump before time $t_1$ (see \eqref{eq: T} and also Figure~\ref{fig: winningTribe}). We also require $T\in[t_2+\ceil{\delta\ell_N},t_1-\ceil{\delta\ell_N}]$, which is part of the event $\Aa_3$.

To keep track of the size of the leader tribe we introduce notation for the number of particles which are within distance $\eps a_N$ of the leader at time $n$:
\begin{equation}\label{eq: RepsN}
R_{\eps,N}(n) := \max\bracing{i\in[N]:\: \X_{N-i+1}(n) \geq \XN(n) - \eps a_N}, \text{ for } n\in\Nzero \text{ and } \eps > 0.
\end{equation}
Note that if $R_{\eps,N}(t_1)<N$ then particle $(N-R_{\eps,N}(t_1)+1,t_1)$
is within distance $\eps a_N$ of the leader, but particle $(N-R_{\eps,N}(t_1),t_1)$ is not. In the event we introduce below, we set $\eps=c_1$ and require the distance between these two particles to be at least $c_2a_N$, showing that there is a gap between the leader tribe and the other particles. The event is defined as follows:
\begin{equation}\label{eq: B1}
\Bb_1 := \left\{\begin{array}{l}
R_{ c_1 ,N}(t_1) \leq 
\min\bracing{N-1,
2N^{1-\delta}},\\
\mbox{$\X_{N-R_{ c_1 ,N}(t_1)}(t_1)\leq\X_{N-R_{ c_1 ,N}(t_1)+1}(t_1)  - c_2 a_N$,}
\\
T\in[t_2+\ceil{\delta\ell_N},t_1-\ceil{\delta\ell_N}] \text{ and }
\Nn_{N,T}(t_1) = \bracing{N-R_{ c_1 ,N}(t_1)+1,\dots,N}
\end{array}\right\},
\end{equation}
where $T=T(\rho)$ and $\Nn_{N,T}(t_1)$ are given by \eqref{eq: T} and \eqref{eq: N} respectively.

In the description of Figure~\ref{fig: winningTribe} in Section~\ref{sect: heur_pic}, we explained that the descendants of particle $(N,T)$ are likely to lead at time $t_1$. The event $\Bb_1$ requires more; it also says that the leading tribe leads by a large distance, which is important to ensure that no other tribes can interfere with our heuristic picture and will be useful in Section~\ref{sect: B1B2}. 
The most involved part of the deterministic argument in the remainder of Section~\ref{sect: deterministic} is to break up the event $\Bb_1$ into other events which happen with probability close to 1.

We now define another event which says that particles which are not in the leading tribe at time $t_1$ have at most~$N^{1-\gamma}$ (i.e. much less than $N$ for $N$ large) descendants in total at time $t$. This will imply that the leading tribe at time~$t_1$ will dominate the population at time $t$. We let
\begin{equation}\label{eq: B2}
\Bb_2 := \bracing{
\sum_{j=1}^{N - R_{ c_1 ,N}(t_1)} |\Nn_{j,t_1}(t)| \leq N^{1-\gamma}
},
\end{equation}
where $\Nn_{j,t_1}(t)$ is given by \eqref{eq: N}. The events which we will introduce to break down the event $\Bb_1$ will easily imply $\Bb_2$ as well. Before defining the new events we check that $\Bb_1$ and~$\Bb_2$ indeed imply $\Aa_3$.

\begin{lemma}\label{lemma: A}
Let $\Aa_3$, $\Bb_1$ and $\Bb_2$ be the events given by~\eqref{eq: A3},~\eqref{eq: B1} and~\eqref{eq: B2} respectively. Then for all $N\geq 2$ and $t>4\ell_N$,
\begin{equation*}
\Bb_1\cap\Bb_2 \subseteq \Aa_3.
\end{equation*}
\end{lemma}
\begin{proof}
On the event $\Bb_1$, the descendants of particle $(N,T)$ are the $R_{ c_1 ,N}(t_1)$ rightmost particles at time $t_1$. Thus $\Nn_{N,T}(t)$ is a disjoint union of the sets $\Nn_{j,t_1}(t)$ for $j\in\llbracket N-R_{ c_1 ,N}(t_1)+1, N\rrbracket$.
We deduce that on the event $\Bb_1\cap \Bb_2$,
\begin{equation*}
|\Nn_{N,T}(t)| = \sum_{j=N-R_{ c_1 ,N}(t_1)+1}^{N} |\Nn_{j,t_1}(t)| \ge N - N^{1-\gamma}.
\end{equation*}
Since $T\in[t_2+\ceil{\delta\ell_N},t_1-\ceil{\delta\ell_N}]$ on the event $\Bb_1$, the result follows.
\end{proof}


\subsection{Breaking down events $\Bb_1$ and $\Bb_2$}\label{sect: B1B2}
We now break down the events $\Bb_1$ and $\Bb_2$ into new events $\Cc_1$ to $\Cc_7$ whose probabilities will be easier to estimate. The majority of the work in this section consists of showing that the intersection of the new events implies $\Bb_1$. We can then quickly conclude that the intersection implies both $\Bb_2$ and $\Aa_1$. One of the new events will need to be further broken down in Section~\ref{sect: C1}. 

\subsubsection{New events $\Cc_1$ to $\Cc_7$}
Recall that $\bbracket{s_1,s_2}$ denotes the set of integers in the interval $[s_1,s_2]$ and that the constants $\gamma$, $\delta$, $\rho$, $c_1,c_2, \dots,c_6$, $\eta$ and $K$ satisfy \eqref{eq: const1}-\eqref{eq: const5}. We first introduce $\tau_1$ to denote the first time after~$t_2$ when a gap of size $2c_3a_N$ appears between the leader and the second rightmost particle:
\begin{equation}\label{eq: That}
\tau_1 := \inf\bracing{s \geq t_2+1:\: \XN(s) > \X_{N-1}(s) + 2c_3a_N}.
\end{equation}
The first new event we define says that such a gap appears by time~$t_1$, that is
\begin{equation}\label{eq: C1}
\Cc_1 := \bracing{\tau_1\in \bbracket{t_2+1,t_1}}.
\end{equation}

The next event $\Cc_2$ ensures that the current leading tribe keeps distance from the other tribes during the time interval $[\tau_1,t_1]$. This is important, since $\Bb_1$ requires a gap behind the leading tribe at time $t_1$. 
The event $\Cc_2$ says that if a particle is far away (at least $c_3a_N$) from the leader, then it cannot jump to within distance $2c_2a_N$ of the leader's position with a single big jump (recall from~\eqref{eq: const} that~$c_2\ll c_3$). That is, a particle far from the leader either stays at least~$2c_2a_N$ behind the leader, or it beats the leader by more than~$2c_2a_N$. Jumping close to the leader would require a large jump, of size greater than~$c_3a_N$, restricted to an interval of size $4c_2a_N$, which is much smaller than the size of the jump. We will see in Section~\ref{sect: probabilities} that the probability that such a jump occurs between times~$t_3$ and $t_1$ is small. Let~$Z_i(s)$ denote the gap between the rightmost and the $i$th particle at time $s$:
\begin{equation}\label{eq: Z}
Z_i(s) := \XN(s) - \X_i(s), \quad \text{for }s\in\Nzero \text{ and }i\in[N].
\end{equation}
Now we can define our next event
\begin{equation}\label{eq: C2}
\Cc_2 :=\bracing{\begin{array}{l}
	\nexists(i,b,s)\in\Nonetwo\times\bbracket{t_3,t-1} \text{ such that }\\
	Z_i(s)\geq c_3 a_N \text{ and } X_{i,b,s}\in(Z_i(s) - 2c_2a_N, Z_i(s) + 2c_2a_N]
	\end{array}}.
\end{equation}

We need to introduce several more events to make sure that `all goes well'; that is, particles which we do not expect to make big jumps indeed do not make big jumps, and smaller jumps do not make too much difference on the $a_N$ space scale. The next event says that if a particle makes a big jump, then it will not have a descendant which makes another big jump within~$\ell_N$ time:
\begin{equation}\label{eq: C3}
\Cc_3 :=\bracing{\begin{array}{l}
B_N\cap P_{k_1,s_1}^{k_2,s_2} = \bracing{(k_1,b_1,s_1)} \\
\forall (k_1,b_1,s_1)\in B_N,\; \forall s_2\in \bbracket{s_1+1,\min\bracing{s_1+\ell_N+1, t}},\; \forall k_2\in\Nn_{k_1,s_1}^{b_1}(s_2)
\end{array}},
\end{equation}
where $B_N$, $P_{k_1,s_1}^{k_2,s_2}$ and $\Nn_{k_1,s_1}^{b_1}(s_2)$ are defined in \eqref{eq: BN2}, \eqref{eq: P} and \eqref{eq: Nb} respectively.

The next event says the following. Take any path between two particles in the time interval $[t_4,t]$. If we omit the big jumps from the path then it does not move more than distance $ c_1  a_N$. In particular, if there are no big jumps at all then the path moves at most $ c_1  a_N$. The event is given by
\begin{equation}\label{eq: C4}
\Cc_4 :=\bracing{\begin{array}{l}
\sum_{(i,b,s)\in P_{k_1,s_1}^{k_2,s_2}} X_{i,b,s}\mathds{1}_{\bracing{X_{i,b,s}\leq \rho a_N}} \leq  c_1  a_N \\
\forall (k_1,s_1)\in [N]\times\bbracket{t_4,t-1},\; \forall s_2\in\bbracket{s_1+1,t},\; \forall k_2\in\Nn_{k_1,s_1}(s_2)
\end{array}},
\end{equation}
where $P_{k_1,s_1}^{k_2,s_2}$ and $\Nn_{k_1,s_1}(s_2)$ are defined in \eqref{eq: P} and \eqref{eq: N} respectively.

The last three events are simple. On $\Cc_5$, two big jumps cannot happen at the same time:
\begin{equation}\label{eq: C5}
\Cc_5 := \bracing{|B_N\cap\bracing{(k,b,s):\: (k,b)\in\Nonetwo}|\leq 1\: \forall s\in\bbracket{t_4,t-1}}.
\end{equation}
Then $\Cc_6$ excludes big jumps which happen either right after time $t_2$ or very close to time $t_1$:
\begin{equation}\label{eq: C6}
\Cc_6 := \bracing{
B_N^{[t_2,t_2+\ceil{\delta \ell_N}]} \cup
B_N^{[t_1-\ceil{\delta \ell_N},t_1+\ceil{\delta \ell_N}]} = \emptyset }
\end{equation}
where $B_N^{[s_1,s_2]}$ is defined in \eqref{eq: BN1}. Finally, $\Cc_7$ gives a bound on the number of big jumps:
\begin{equation}\label{eq: C7}
\Cc_7 := \bracing{|B_N|\leq K},
\end{equation}
where we recall that we chose $K$ to be a positive constant at the start of Section~\ref{sect: deterministic}.

Now we can state the main result of this subsection. It says that on the events $\Cc_1$ to $\Cc_7$ the events $\Bb_1$, $\Bb_2$ and $\Aa_1$ occur, and therefore $\Aa_3$ occurs as well. We have an additional event in Proposition~\ref{prop: B} below, which says that the diameter of the particle cloud at time $t_1$ is larger than $\tfrac{3}{2}c_3a_N$. As part of the proposition we also show that $\Cc_1$ to $\Cc_7$ imply this event, because it will be useful in another argument later on in Section~\ref{sect: significance}.

\begin{prop}\label{prop: B}
Let $\eta\in(0,1]$, and assume that the constants~${\gamma,\delta,\rho, c_1 ,c_2,\dots,c_6,K}$ satisfy~\eqref{eq: const1}-\eqref{eq: const5}.
Then for $N$ sufficiently large that $2K N^{-\delta}<N^{-\gamma}<1$ and $t>4\ell_N$,
\begin{equation*}
\bigcap_{j=1}^7 \Cc_j \subseteq \Bb_1\cap\Bb_2\cap\Aa_1\cap \bracing{d(\X(t_1)) \geq \tfrac{3}{2}c_3a_N} \subseteq \Aa_1 \cap \Aa_3\cap \bracing{d(\X(t_1)) \geq \tfrac{3}{2}c_3a_N},
\end{equation*}
where $\Bb_1$, $\Bb_2$, $\Aa_1$ and $\Aa_3$ are defined in \eqref{eq: B1}, \eqref{eq: B2}, \eqref{eq: A1} and \eqref{eq: A3} respectively, and $\Cc_1,\Cc_2,\dots,\Cc_7$ are given by \eqref{eq: C1} and \eqref{eq: C2}--\eqref{eq: C7}.
\end{prop}
Note that the second inclusion in Proposition~\ref{prop: B} follows directly from Lemma~\ref{lemma: A}.

\subsubsection{$\Cc_1$ to $\Cc_7$ imply $\Bb_1$, $\Bb_2$ and $\Aa_1$: proof of Proposition~\ref{prop: B}} \label{sect:proofpropB}

We start by proving some easy lemmas which hold on the event $\bigcap_{j=1}^7\Cc_j$, and which will be applied in the course of the proof of Proposition~\ref{prop: B}.

The first lemma gives another way of writing the event $\Cc_4$, which will be more convenient to use in this section. (The definition of $\Cc_4$ will be easier to work with when we show, in Section~\ref{sect: probabilities}, that $\Cc_4$ occurs with high probability.)
The lemma says that on the event $\Cc_4$, if a path moves more than $ c_1  a_N$ then it must contain a big jump.

\begin{lemma}\label{lemma: C4}
On the event $\Cc_4$, for all $(k_1,s_1)\in [N]\times\bbracket{t_4,t-1}$, $s_2\in\bbracket{s_1+1,t}$ and~$k_2\in\Nn_{k_1,s_1}(s_2)$, 
\begin{align*}
\X_{k_2}(s_2) > \X_{k_1}(s_1) +  c_1  a_N \Longrightarrow B_N\cap P_{k_1,s_1}^{k_2,s_2} \neq \emptyset,
\end{align*}
where $B_N$, $\Nn_{k_1,s_1}(s_2)$ and $P_{k_1,s_1}^{k_2,s_2}$  are defined in \eqref{eq: BN2}, \eqref{eq: N} and \eqref{eq: P} respectively.
\end{lemma}
\begin{proof}
Let $(k_1,s_1)\in [N]\times\bbracket{t_4,t-1}$, $s_2\in\bbracket{s_1+1,t}$, and~$k_2\in\Nn_{k_1,s_1}(s_2)$. Assume that $B_N\cap P_{k_1,s_1}^{k_2,s_2} =\emptyset$, and the event $\Cc_4$ occurs. Then by~\eqref{eq:pathsum},
\begin{align*}
\X_{k_2}(s_2) = \X_{k_1}(s_1) + \sum_{(i,b,s)\in P_{k_1,s_1}^{k_2,s_2}} X_{i,b,s} =  \X_{k_1}(s_1) +\sum_{(i,b,s)\in P_{k_1,s_1}^{k_2,s_2}} X_{i,b,s}\mathds{1}_{\bracing{X_{i,b,s}\leq \rho a_N}} \leq \X_{k_1}(s_1) +  c_1  a_N
\end{align*}
by the definition of the event $\Cc_4$,
which completes the proof.
\end{proof}

The next lemma says that if a path of length at most $\ell_N$ starts with a big jump then it moves distance at most~$ c_1  a_N$  after the big jump.
\begin{lemma}\label{lemma: C3C4}
On the event $\Cc_3 \cap \Cc_4$, for all $(k_1,b_1,s_1)\in B_N$, $s_2\in \bbracket{s_1+1,\min\bracing{s_1+\ell_N, t}} $ and $k_2\in\Nn_{k_1,s_1}^{b_1}(s_2)$,
\begin{equation*}
\X_{k_2}(s_2) \leq \X_{k_1}(s_1) + X_{k_1,b_1,s_1} +  c_1  a_N,
\end{equation*}
where $B_N$ and $\Nn_{k_1,s_1}^{b_1}(s_2)$  are defined in \eqref{eq: BN2} and~\eqref{eq: Nb} respectively.
\end{lemma}
\begin{proof}
Let $l\in [N]$ be such that $(k_1,s_1) \lesssim_{b_1} (l,s_1 + 1)$, so that
\begin{equation}\label{eq: XlXk}
\X_l(s_1+1) = \X_{k_1}(s_1) + X_{k_1,b_1,s_1}.
\end{equation}
If $s_2=s_1+1$ then we are done; from now on assume $s_2\ge s_1+2$.
Since $X_{k_1,b_1,s_1}$ is a big jump, on the event $\Cc_3$ there are no further big jumps on the path between particles $(l,s_1+1)$ and $(k_2,s_2)$, that is $B_N\cap P_{l,s_1+1}^{k_2,s_2}=\emptyset$. Therefore, by Lemma~\ref{lemma: C4} we have $\X_{k_2}(s_2) \leq \X_l(s_1+1) +  c_1  a_N$, which, together with \eqref{eq: XlXk}, completes the proof.
\end{proof}

In the next lemma, we describe how we can exploit the fact that on the event $\Cc_5$ there are never two big jumps at the same time. First, the event $\Cc_5$ tells us that if a particle makes a big jump, then the other particles move very little at the time of the jump. Second, it also implies that if a particle significantly beats the current leader with a big jump, then it becomes the new leader, and the gap behind this new leader will be roughly the distance by which it beat the previous leader. Both statements follow immediately from the setup, but will be useful for example in the proofs of Corollaries~\ref{cor: gaps} and \ref{cor: beatingLeader} below, and later on in the proofs of Propositions~\ref{prop: C} and \ref{prop: A2prime} as well.

\begin{lemma}\label{lemma: breakRecordGap}
On the event $\Cc_5$, for all $(k,b,s)\in B_N$, 
\begin{enumerate}[label=(\alph*), ref=\alph*]
\item\label{lemma: breakRecordGapA} $\X_j(s+1) \leq \XN(s) + \rho a_N$ for all $j\in[N]\setminus \mathcal N_{k,s}^b(s+1)$, and
\item\label{lemma: breakRecordGapB} if $\X_k(s)+X_{k,b,s} > \X_N(s) + ca_N$ for some $c>\rho$, then $(k,s)\lesssim_{b} (N,s+1)$ and $\X_N(s+1) - \X_{N-1}(s+1) > (c-\rho)a_N$.
\end{enumerate}
\end{lemma}

\begin{proof}
Assume that $\Cc_5$ occurs and fix $k,b,s$ as in the statement.  
Let $j\in[N]\setminus \mathcal N_{k,s}^b(s+1)$ be arbitrary. Assume that $i\in[N]$ and $b_i\in \{1,2\}$ are such that $(i,s)\lesssim_{b_i} (j,s+1)$, and so $\X_j(s+1) = \X_i(s) + X_{i,b_i,s}$,
with $(i,b_i)\in(\Nonetwo)\setminus\bracing{(k,b)}$. By the definition of the event $\Cc_5$, $X_{k,b,s}$ is the only big jump at time $s$. Thus we have $X_{i,b_i,s} \leq \rho a_N$, and by  bounding the $i$th particle's position at time $s$ by the rightmost position at time $s$ we get
\begin{equation*}
\X_j(s+1) = \X_i(s) + X_{i,b_i,s} \leq \XN(s) + \rho a_N,
\end{equation*}
which completes the proof of part \eqref{lemma: breakRecordGapA}. Furthermore, if the condition in \eqref{lemma: breakRecordGapB} holds, then we also have
\begin{equation}\label{eq: XjXlc}
\X_j(s+1) \leq \XN(s) + \rho a_N < \X_k(s)+X_{k,b,s} - (c-\rho)a_N.
\end{equation}
Since \eqref{eq: XjXlc} holds for any $j\in[N]\setminus \mathcal N_{k,s}^b(s+1)$ and we are assuming $c>\rho$, we conclude that $(k,s)\lesssim_{b} (N,s+1)$, and the result follows by taking $j=N-1$ in~\eqref{eq: XjXlc}.
\end{proof}

The next lemma says that if $\Cc_3 \cap \Cc_4$ occurs then all big jumps in the time interval $[t_3,t-1]$ come from close to the leftmost particle. Our heuristics suggest this should be true, because we expect most particles to be close to the leftmost particle at a typical time. However, the proof only relies on the assumption that the events $\Cc_3$ and $\Cc_4$ occur.

\begin{lemma}\label{lemma: bigJumpLeftmost}
On the event $\Cc_3 \cap \Cc_4$,
\begin{equation*}
\X_k(s) \leq \Xone(s) +  c_1  a_N  \quad \forall (k,b,s)\in B_N^{[t_3,t-1]}.
\end{equation*}
\end{lemma}

\begin{proof}
Take $s\in\bbracket{t_3,t-1}$, $k\in[N]$ and $b\in \{1,2\}$, and assume that we have $X_{k,b,s} > \rho a_N$. Let $i_k = \zeta_{k,s}(s - \ell_N)$ be the time-$(s - \ell_N)$ ancestor of particle $(k,s)$ (recall \eqref{eq: zeta}). Since $(k,b,s)\in B_N$, by the definition of the event $\Cc_3$, we must have $B_N \cap P^{k,s}_{i_k,s-\ell_N}=\emptyset$. Then by Lemma~\ref{lemma: C4} we have
\begin{equation}\label{eq: XkXi}
\X_k(s) \leq \X_{i_k}(s - \ell_N) +  c_1  a_N.
\end{equation}
Furthermore, at time $s$ every particle is to the right of $\X_{i_k}(s - \ell_N)$, by Lemma~\ref{lemma: descendants_est}. This means $\X_{i_k}(s - \ell_N)\leq \X_1(s)$, and so $\X_k(s) \leq \Xone(s) +  c_1  a_N$ by~\eqref{eq: XkXi}.
\end{proof}

We will use Lemma~\ref{lemma: bigJumpLeftmost} to prove the next result, which says that if the diameter of the cloud of particles is large and a particle makes a big jump, then either it takes the lead and will be significantly ahead of the second rightmost particle, or it stays significantly behind the leader. 

\begin{corollary}\label{cor: gaps}
On the event $\bigcap_{j=2}^5 \Cc_j$, if $(k,b,s)\in B_N^{[t_3,t-1]}$ and $d(\X(s)) \geq (c_3+c_1)a_N$ then
\begin{enumerate}[label=(\alph*), ref=\alph*]
\item\label{cor: gapsA} if $X_{k,b,s} > Z_k(s)$ then $\XN(s+1) = \X_k(s) + X_{k,b,s} > \X_{N-1}(s+1) + (2c_2 - \rho)a_N$, and
\item\label{cor: gapsB} if $X_{k,b,s} \leq Z_k(s)$ then $\X_k(s) + X_{k,b,s} \leq \XN(s) - 2c_2a_N$,
\end{enumerate}
where $Z_k(s)$ and $\Cc_2,\dots,\Cc_5$ are given by \eqref{eq: Z}--\eqref{eq: C5}.
\end{corollary}
\begin{proof}
Since $X_{k,b,s}$ is a big jump, by Lemma~\ref{lemma: bigJumpLeftmost} and the fact that $d(\X(s))=\X_N(s)-\X_1(s)\ge (c_3+c_1)a_N$,
\begin{equation*}
\X_k(s) \leq \Xone(s) +  c_1  a_N \le \XN(s) - c_3a_N.
\end{equation*}
Hence the gap between the $k$th particle and the rightmost particle is bounded below by $c_3 a_N$:
\begin{equation}\label{eq: Zk}
Z_k(s) \ge c_3a_N. 
\end{equation}
It follows that if $X_{k,b,s}>Z_k(s)$, by the definition of the event $\Cc_2$ we have
$X_{k,b,s} > Z_k(s) + 2c_2a_N$, which implies that
\begin{equation*}
X_{k,b,s} + \X_k(s) > \XN(s) + 2c_2a_N.
\end{equation*}
Since $2c_2>\rho$ by~\eqref{eq: const2} and~\eqref{eq: const4}, Lemma~\ref{lemma: breakRecordGap}(b) implies the statement of part (a). If instead $X_{k,b,s}\le Z_k(s)$, then by~\eqref{eq: Zk} and the definition of $\Cc_2$, we have $X_{k,b,s}\leq Z_k(s)-2c_2 a_N$, which completes the proof.
\end{proof}

The next result says that if the diameter of the cloud of particles is big at some time $s$, then if at time $s$ or $s-1$ a particle makes a big jump which beats the current leader, this particle becomes the new leader.
\begin{corollary}\label{cor: beatingLeader}
On the event $\bigcap_{j=2}^5 \Cc_j$, for all $s\in\bbracket{t_3+1,t-1}$,
if $d(\X(s)) \geq \frac{3}{2}c_3a_N$ then 
\begin{equation*}
s\in \Sbf_N \Longleftrightarrow s\in \hat \Sbf_N \quad \text{ and } \quad s-1\in \Sbf_N \Longleftrightarrow s-1\in \hat \Sbf_N,
\end{equation*}
where
$\Sbf_N$ and $\hat \Sbf_N$ are defined in \eqref{eq: RN} and \eqref{eq: RNhat}.
\end{corollary}
\begin{proof}
Take $s\in\bbracket{t_3+1,t-1}$ and suppose $d(\X(s))\ge \frac 32 c_3 a_N$.

If $s\in \Sbf_N$, then there exists $(k,b,s)\in B_N$ such that $\X_k(s) + X_{k,b,s} = \XN(s+1) \geq \XN(s)$, where we used monotonicity for the inequality. To show that $s\in \hat \Sbf_N$, we need to show that in fact $\X_k(s) + X_{k,b,s} > \XN(s)$, i.e.~the inequality is strict, but this follows from Corollary~\ref{cor: gaps}(b), which applies since $d(\X(s)) \geq \frac{3}{2}c_3a_N \geq (c_1+c_3)a_N$ by \eqref{eq: const4}.

Now suppose $s\in\hat\Sbf_N$. Since $d(\X(s)) \geq \frac{3}{2}c_3a_N \geq (c_1+c_3)a_N$, and by the definition of $\hat \Sbf_N$, the conditions of Corollary~\ref{cor: gaps}(a) hold for $(k,b,s)$, for some $(k,b)\in\Nonetwo$. Then Corollary~\ref{cor: gaps}(a) implies that $s\in \Sbf_N$, and therefore the first equivalence in the statement holds. 

If $d(\X(s-1))\geq (c_3+c_1)a_N$, then we can repeat the proof of the first equivalence to show that $s-1\in \Sbf_N \Longleftrightarrow s-1\in \hat \Sbf_N$. 

If instead $d(\X(s-1))< (c_3+c_1)a_N$ we argue as follows. Suppose $s-1\in\Sbf_N$. Then there exists $(k,b,s-1)\in B_N$ such that
\begin{equation*}
\X_k(s-1) + X_{k,b,s-1} = \XN(s) \geq \XN(s-1),
\end{equation*} 
which means $X_{k,b,s-1} \geq Z_{k}(s-1)$. Now $X_{k,b,s-1} = Z_{k}(s-1)$ is impossible because, with the assumption that $d(\X(s-1))< (c_3+c_1)a_N$, it would imply
\begin{equation*}
\XN(s) = \X_k(s-1) + Z_{k}(s-1) = \XN(s-1)< \Xone(s-1) + (c_3+c_1)a_N < \Xone(s) + \tfrac 32 c_3a_N
\end{equation*}
by monotonicity and \eqref{eq: const4}. This contradicts the assumption $d(\X(s)) \geq \frac 32 c_3 a_N$ from the statement of this corollary. Hence, we must have $X_{k,b,s-1} > Z_{k}(s-1)$, and so $s-1\in \hat\Sbf_N$.

Now suppose $s-1\in \hat \Sbf_N$, and take $(k,b,s-1)\in B_N$ such that $\X_k(s-1)+X_{k,b,s-1}>\X_N(s-1)$.
Then by Lemma~\ref{lemma: breakRecordGap}(a) and the assumption on $d(\X(s-1))$, for all $j\in [N]\setminus \mathcal N^b_{k,s-1}(s)$ we have
\begin{equation*}
\X_j(s) \leq \XN(s-1) + \rho a_N < \Xone (s-1) + (c_3+c_1 + \rho)a_N.
\end{equation*}
By monotonicity, \eqref{eq: const2} and \eqref{eq: const4} this is strictly smaller than $\Xone(s) + \tfrac{3}{2}c_3a_N$. 
Thus, at time $s$, all particles not in $\mathcal N^b_{k,s-1}(s)$ are closer than distance $\frac{3}{2}c_3a_N$ to the leftmost particle. Hence, since we assumed that $d(\X(s))\ge \frac{3}{2}c_3a_N$, we must have $(k,s-1)\lesssim_b (N,s)$, which means that $s-1\in \Sbf_N$.
\end{proof}

The last property we state before the proof of Proposition~\ref{prop: B} says the following. First, if no particle beats the leader with a big jump for a time interval of length at most $\ell_N$, then the leader's position does not change much during this time. We will use the extra condition that the diameter is not too small to prove this easily; if the diameter is too small then jumps that are ``almost big'' could complicate matters. Second, the lemma says that if the diameter becomes small at some point, then it cannot become too large within $\ell_N$ time, if there is no particle which beats the leader with a big jump.
Recall the definition of $\hat \Sbf_N$ from~\eqref{eq: RNhat}.

\begin{lemma}\label{lemma: noRecord}
On the event $\bigcap_{j=2}^5 \Cc_j$, for all $s\in\bbracket{t_3,t_1-1}$ and $\Delta s\in[ \ell_N]$, if~$s+\Delta s \leq t_1$ and $\bbracket{s,s+\Delta s-1} \subseteq \hat \Sbf_N^c$ then the following statements hold:
\begin{enumerate}[label=(\alph*), ref=\alph*]
\item If $d(\X(r)) \geq \frac{3}{2}c_3a_N$ for all $r\in\bbracket{s,s+\Delta s-1}$, then $\XN(s + \Delta s) \leq \XN(s) +  c_1  a_N$. In particular, if $\Delta s = \ell_N$ then $d(\X(s + \ell_N )) \leq  c_1  a_N$.
\item If there exists $r\in\bbracket{s,s+\Delta s-1}$ such that $d(\X(r)) \leq \frac{3}{2}c_3a_N$, then $d(\X(s+\Delta s)) \leq \tfrac{3}{2}c_3a_N + 2 c_1  a_N$.
\end{enumerate}
\end{lemma}
\begin{proof}
First we prove part (a). Let $i,j\in[N]$ with $(i,s)\lesssim(j,s+\Delta s)$. Assume that there is a big jump on the path between $\X_i(s)$ and $\X_j(s+\Delta s)$ at time $s'\in\bbracket{s,s+\Delta s-1}$, i.e.~there exists $(k',b',s')\in B_N\cap P_{i,s}^{j,s+\Delta s}$. Since we assume $s'\in \hat \Sbf_N^c$, we have $\X_{k'}(s') + X_{k',b',s'} \leq \XN(s')$. Then since we assume $d(\X(s'))\ge \frac 32 c_3 a_N >(c_3+c_1)a_N$ by \eqref{eq: const4},  we can apply Corollary~\ref{cor: gaps}(b) to obtain
\begin{equation}\label{eq: XN2c1}
\X_{k'}(s') + X_{k',b',s'} \leq \XN(s') - 2c_2a_N.
\end{equation}
Therefore, first by Lemma~\ref{lemma: C3C4}, second by \eqref{eq: XN2c1}, and third by monotonicity and \eqref{eq: const4} we get
\begin{equation*}
\X_j(s+\Delta s) \leq \X_{k'}(s') +  X_{k',b',s'} +  c_1  a_N \leq \XN(s') - 2c_2a_N +  c_1  a_N < \XN(s+\Delta s).
\end{equation*}
Hence $j\neq N$, which means that the leader at time $s+\Delta s$ must be a particle which does not have an ancestor which made a big jump in the time interval $[s,s+\Delta s-1]$. That is, $B_N\cap P_{i,s}^{N,s+\Delta s} = \emptyset$ for all~$i\in[N]$. But then by Lemma~\ref{lemma: C4} we must have
\begin{equation*}
\XN(s+\Delta s) \leq \XN(s) +  c_1  a_N,
\end{equation*}
which shows the first statement of part (a). By Lemma~\ref{lemma: descendants_est} we also have $\Xone(s+\ell_N) \geq \XN(s)$, and the second statement of part (a) follows.

Now we prove part (b). Let $\tau_d$ denote the last time before $s+\Delta s$ when the diameter is at most~$\frac{3}{2}c_3a_N$, that is
\begin{equation*}
\tau_d = \sup\bracing{r \leq s + \Delta s:\: d(\X(r))\leq \tfrac{3}{2}c_3a_N}.
\end{equation*}
By our assumption in part (b) we have $\tau_d \geq s$.

If $\tau_d = s + \Delta s$ then we are done. Assume instead that $\tau_d<s + \Delta s$. Then we can estimate the leftmost particle position at time $s+\Delta s$ using monotonicity and the definition of $\tau_d$:
\begin{equation}\label{eq: X1spDs}
\Xone(s + \Delta s) \geq \Xone(\tau_d) \geq \XN(\tau_d) - \tfrac{3}{2}c_3a_N.
\end{equation}
To estimate the rightmost position, we first use the fact that $\tau_d\in \llbracket s,s+\Delta s-1 \rrbracket \subseteq \hat \Sbf _N^c$ and $d(\X(\tau_d+1)) > \tfrac{3}{2}c_3a_N$ by the definition of $\tau_d$. Hence, the second equivalence of Corollary~\ref{cor: beatingLeader} implies that $\tau_d\in \Sbf_N^c$; that is, no big jump takes the lead at time $\tau_d+1$. Thus, for 
some $(k,b)\in\Nonetwo$ we have
\begin{equation}\label{eq: XNtaud}
\XN(\tau_d+1) = \X_k(\tau_d) + X_{k,b,\tau_d} \leq \XN(\tau_d) + \rho a_N.
\end{equation}
Now \eqref{eq: X1spDs}, \eqref{eq: XNtaud} and \eqref{eq: const2} show that if $\tau_d = s+\Delta s - 1$ then we are done. Assume instead that~$\tau_d<s + \Delta s-1$. 
Then we can apply part (a) for the time interval $[\tau_d+1,s+\Delta s]$, because $d(\X(r))> \frac{3}{2}c_3a_N$ $\forall r\in \llbracket \tau_d+1,s+\Delta s \rrbracket$ by the definition of $\tau_d$. So by part (a) and then by~\eqref{eq: XNtaud} we have
\begin{equation}\label{eq: XNspDs}
\XN(s+\Delta s) \leq \XN(\tau_d + 1) +  c_1  a_N \leq \XN(\tau_d) + (\rho +  c_1 )a_N.
\end{equation}
Now \eqref{eq: XNspDs}, \eqref{eq: X1spDs} and \eqref{eq: const2} yield part (b).
\end{proof}

\begin{proof}[Proof of Proposition~\ref{prop: B}]

The main effort of this proof is in showing that the $\Cc_i$ events imply $\Bb_1$. So we want to see a leader tribe at time $t_1$ in which all the particles are descended from particle $(N,T)$, and are significantly to the right of all the particles not descended from particle $(N,T)$. We begin by giving an outline of how this will be proved. 

\vspace{3mm}

\noindent \textbf{Outline of proof that $\Cc_1$ to $\Cc_7$ imply $\Bb_1$}

\noindent
Assume the event $\cap _{j=1}^7 \Cc_j$ occurs. On the event $\Cc_1$ there will be a time $\tau_1\in [t_2+1,t_1]$ when the leader, particle $(N,\tau_1)$, is a distance more than $2c_3a_N$ ahead of the second rightmost (and every other) particle. Having this gap at time $\tau_1$ will ensure that the back of the population is further than $\frac{3}{2}c_3a_N$ away from the leader at all times up to $t_1$. That is, the diameter cannot be too small after time $\tau_1$, and so we will be able to apply Corollary~\ref{cor: gaps}.

It is a possibility that on the time interval $[\tau_1,t_1]$, every particle not descended from $(N,\tau_1)$ stays further than roughly $2c_2a_N$ to the left of the tribe descending from $(N,\tau_1)$. Then we will have the desired leader tribe with a gap behind it at time~$t_1$. Alternatively, the tribe of particle $(N,\tau_1)$ may be surpassed by other particles. But then, by Corollary~\ref{cor: gaps}(a), the leader must be beaten by at least roughly $2c_2a_N$. The new leader's descendants might be surpassed too, but again by at least $2c_2a_N$. Then, after the last time~$T$ when a tribe is surpassed before~$t_1$~(i.e.~the last time when a big jump takes the lead, see \eqref{eq: T}), no particle will make a big jump that gets closer to the leader tribe than~$2c_2a_N$, by Corollary~\ref{cor: gaps}(b). We will see that this implies that at time $t_1$, the leader tribe will be further away than $c_2a_N$ from all the other particles. This argument works if the particles of the tribes do not move far from the position of their ancestor which made a big jump. We have this property due to Lemma~\ref{lemma: C3C4}.

Therefore, the proof will expand on the following steps:
\begin{enumerate}[label=(\roman*), ref=\roman*]
\item The record is broken by a big jump at time $\tau_1$. Therefore time $T$, the last time  when the record is broken by a big jump before time $t_1$, is either at time $\tau_1$ or later.
	
\item The diameter is at least $\frac{3}{2}c_3a_N$ between times $\tau_1$ and $t_1$.

We will show that the back of the population stays far behind $\XN(\tau_1)$, because of the small number of big jumps compared to the number of particles. This is useful, because most of the lemmas and corollaries above will apply if the diameter is not too small.

\item At time $T$, the last time before $t_1$ when a particle takes the lead with a big jump, there will be a gap of size at least $\frac{3}{2}c_2a_N$ between the leader $(N,T)$ and the second rightmost particle $(N-1,T)$. 

This step follows by Corollary~\ref{cor: gaps}(a), which we can apply because of step (ii). If the diameter is big and the leader is beaten, then the new leader will lead by a large distance. 

\item Every other particle stays at least distance $c_2 a_N$ behind the descendants of particle $(N,T)$ until time $t_1$.

This is mainly due to steps (ii) and (iii) and Corollary~\ref{cor: gaps}(b): if the diameter is big and the leader is not beaten by a big jump, then big jumps will arrive far behind the leader. Therefore, the gap behind the leader tribe created in step~(iii) will remain until time $t_1$.

\item The leading tribe has the size required by the event $\Bb_1$, and thus the event $\Bb_1$ occurs.
\end{enumerate}

\noindent \textbf{Proof that $\Cc_1$ to $\Cc_7$ imply $\Bb_1$}

\noindent
We now give a detailed proof, following steps (i)-(v) above, that
\begin{equation}\label{eq: CB1}
\bigcap_{j=1}^7 \Cc_j \subseteq \Bb_1.
\end{equation}
Assume that $\bigcap_{j=1}^7 \Cc_j$ occurs. We first check that we have $T\in [t_2+\ceil{\delta\ell_N},t_1-\ceil{\delta\ell_N}]$ by proving the following statement.

\begin{enumerate}[label=\emph{Step (\roman*).}, ref=\roman*, itemindent=6.5mm]
\item \label{step: B1tau1T} \emph{We have $t_2 + \ceil{\delta \ell_N} < \tau_1 \leq T \leq t_1 - \ceil{\delta \ell_N}$, where $\tau_1$ and $T$ are defined in \eqref{eq: That} and \eqref{eq: T}.}
\end{enumerate}

\noindent
In order to see this, we will use the following simple property:
\begin{equation}\label{eq: XNsm1}
\X_j(s-1) \leq \X_{N-1}(s) \quad \forall s\in \N \text{ and } j\in[N].
\end{equation}
Indeed, since all jumps are non-negative, and particle $(N,s-1)$ has two offspring, there are at least two particles to the right of (or at) position $\XN(s-1)$ at time $s$. Thus $\XN(s-1)\leq \X_{N-1}(s)$, which shows \eqref{eq: XNsm1}.

By the definition of the event $\Cc_1$,
we have $\tau_1\in\bbracket{t_2+1,t_1}$. Let $(\hat{J},\hat b)\in\Nonetwo$ be such that $(\hat J,\tau_1-1)\lesssim_{\hat b}(N,\tau_1)$, and so $\XN(\tau_1) = \X_{\hat{J}}(\tau_1-1) + X_{\hat{J},\hat b,\tau_1-1}$.
It also follows from \eqref{eq: XNsm1} that $\X_{\hat{J}}(\tau_1-1)\leq\X_{N-1}(\tau_1)$. Hence the definition of $\tau_1$ in~\eqref{eq: That} implies that $X_{\hat{J},\hat b,\tau_1-1}>2c_3a_N$, 
which means that $X_{\hat{J},\hat b,\tau_1-1}$ is a big jump, and so cannot happen on the time interval $[t_2,t_2+\ceil{\delta\ell_N}]$ by the definition of $\Cc_6$. This implies the first inequality in Step~(\ref{step: B1tau1T}). We also notice that~$X_{\hat{J},\hat b,\tau_1-1}$ is a big jump which takes the lead at time $\tau_1$, that is $\tau_1-1\in \Sbf_N$ (see \eqref{eq: RN}). Then we have~$T\geq \tau_1$ by the definition of $T$ in~\eqref{eq: T}, which shows the second inequality of Step~(\ref{step: B1tau1T}). Furthermore, the definition of $T$ also shows that $T>t_1 - \ceil{\delta \ell_N}$ is not possible on~$\Cc_6$, which concludes the third inequality and the proof of Step~(\ref{step: B1tau1T}).

\vspace{3mm}

\noindent Since we now know that $T\neq 0$, particle $(N,T)$ is the last particle which broke the record with a big jump before time $t_1$. Take $(J,b^*)\in\Nonetwo$ such that $(J,T-1)\lesssim_{b^*}(N,T)$, so
\begin{equation}\label{eq: XNTXJ}
\XN(T) = \X_J(T-1) + X_{J,b^*,T-1},
\end{equation}
with $X_{J,b^*,T-1} > \rho a_N$. That is, at time $T-1$ the $J$th particle's $b^*$th offspring performed a big jump~$X_{J,b^*,T-1}$, with which it became the leader at time $T$ at position $\XN(T)$. We will show that at time $t_1$ there is a leader tribe in which every particle descends from particle $(N,T)$. Our next step towards this statement is to show that the diameter is large between times $\tau_1$ and $t_1$.

\begin{enumerate}[label=\emph{Step (\roman*).}, ref=\roman*, itemindent=7.5mm]
\setcounter{enumi}{1}
\item\label{step: B1BigJumps} \emph{We have } $d(\X(s)) \geq \tfrac{3}{2}c_3a_N$ \emph{ for all } $s \in \bbracket{\tau_1,t_1}$.
\end{enumerate}

\noindent
We prove Step~(\ref{step: B1BigJumps}) by showing that the number of particles within distance $\frac{3}{2}c_3a_N$ of the leader is strictly smaller than $N$ at all times in $\bbracket{\tau_1,t_1}$.

Let $s\in\bbracket{\tau_1,t_1}$. Consider an arbitrary particle $(i,s)$ in the population at time $s$. We first claim that if 
\begin{equation}\label{eq: PBi1}
\X_i(s) > \XN(\tau_1) - \tfrac{3}{2}c_3a_N,
\end{equation}
then particle $(i,s)$ has an ancestor which made a big jump at some time $\tilde{s}\in\llbracket \tau_1-1,s-1\rrbracket$. That is, if \eqref{eq: PBi1} holds then
\begin{equation}\label{eq: bigJumpAncestor}
B_N\cap P_{j,\tau_1-1}^{i,s} \neq \emptyset, \quad \text{for some $j\in[N]$.}
\end{equation}
To see this, we notice that
\begin{align}\label{eq: c2gap}
\X_j(\tau_1-1) \leq \X_{N-1}(\tau_1) < \XN(\tau_1) - 2c_3a_N \quad \forall j\in[N],
\end{align}
where the first inequality follows by \eqref{eq: XNsm1}, and the second from the definition  of $\tau_1$. Therefore, by~\eqref{eq: PBi1}, \eqref{eq: c2gap} and \eqref{eq: const4}, we have
\begin{align}
\X_i(s) > \X_j(\tau_1 - 1) +  c_1  a_N \quad \forall j\in[N].
\end{align} 
In particular, this holds for $j\in [N]$ such that $(j,\tau_1-1)\lesssim (i,s)$. Therefore \eqref{eq: bigJumpAncestor} must hold by Lemma~\ref{lemma: C4}, showing that our claim is true.

Thus, every particle which is to the right of $\XN(\tau_1) - \frac{3}{2}c_3a_N$ at time $s$ has an ancestor which made a big jump between times $\tau_1-1$ and $s-1$. This gives us the following bound:
\begin{align}\label{eq: bigJumpDesc}
\#\bracing{i\in[N]:\: \X_i(s) > \XN(\tau_1) - \tfrac{3}{2}c_3a_N} \leq \sum_{(l,b,r)\in B_N^{[\tau_1 - 1,s-1]}} |\Nn_{l,r}^b(s)|,
\end{align}
where $\Nn_{l,r}^b(s)$ and $B_N^{[\tau_1 - 1,s-1]}$ are defined in \eqref{eq: Nb} and \eqref{eq: BN1} respectively. On the right-hand side we sum the number of descendants of all particles which made a big jump between times~$\tau_1-1$ and $s-1$. We want to show that this is smaller than $N$, because that means that there must be at least one particle to the left of (or at) $\XN(\tau_1) - \frac{3}{2}c_3a_N$ at time $s$. 

Since $[\tau_1-1,s]\subseteq [t_2 + \ceil{\delta \ell_N},t_1]$ by Step~(\ref{step: B1tau1T}), any particle at a time in $[\tau_1-1,s-1]$ has at most $2^{t_1 - (t_2 + \ceil{\delta \ell_N})}$ descendants at time $s$. Furthermore, the number of big jumps in the time interval $[\tau_1-1,s-1]$ is at most $K$, by the definition of $\Cc_7$. Hence, by \eqref{eq: bigJumpDesc} and then since $t_1-t_2=\ell_N$,
\begin{align}\label{eq: bigJumpDesc2}
\#\bracing{i\in[N]:\: \X_i(s) > \XN(\tau_1) - \tfrac{3}{2}c_3a_N} \leq K2^{t_1 - (t_2 + \ceil{\delta \ell_N})} \leq 2KN^{1-\delta} < N,
\end{align}
by our assumption on $N$ in the statement of Proposition~\ref{prop: B}. Therefore, by \eqref{eq: bigJumpDesc2} and monotonicity we must have $\Xone(s) \leq \XN(\tau_1) - \tfrac{3}{2}c_3a_N \leq \XN(s) - \tfrac{3}{2}c_3a_N$, 
which concludes the proof of Step~(\ref{step: B1BigJumps}). 

\vspace{3mm}

\noindent
Next we show that there is a gap between the two rightmost particles at time $T$.

\begin{enumerate}[label=\emph{Step (\roman*).}, ref=\roman*, itemindent=8.5mm]
\setcounter{enumi}{2}
\item\label{step: B1Tgap} \emph{We have $\X_{N-1}(T) + \frac{3}{2}c_2a_N < \XN(T)$.}
\end{enumerate}

\noindent
Note that we have $\tau_1 \leq T$ by Step~(\ref{step: B1tau1T}). If $T = \tau_1$ then the statement of Step~(\ref{step: B1Tgap}) holds by the definition of $\tau_1$ and \eqref{eq: const4}. 

Suppose instead that $T>\tau_1$. We now check the conditions of Corollary~\ref{cor: gaps}(a). Recall from~\eqref{eq: XNTXJ} that~$X_{J,b^*,T-1}$ is a big jump. Since the particle performing the jump $X_{J,b^*,T-1}$ becomes the leader at time~$T$, we have  $X_{J,b^*,T-1} \ge Z_J(T-1)$, where $Z_J(T-1)$ is the gap between the~$J$th particle and the leader at time $T-1$. Also note that ${(J,b^*,T-1)\in B_N^{[t_2,t_1]}}$, and that by Step~(\ref{step: B1BigJumps}) and \eqref{eq: const4} we have~$d(\X(T-1)) > (c_3+c_1)a_N$.
Therefore Corollary~\ref{cor: gaps}(a) and (b) imply
\begin{equation*}
\XN(T) = \X_J(T-1) + X_{J,b^*,T-1} > \X_{N-1}(T) + (2c_2 - \rho)a_N,
\end{equation*}
which together with \eqref{eq: const2} and~\eqref{eq: const4} shows the statement of Step~(\ref{step: B1Tgap}).

\vspace{3mm}

\noindent
In Step~(\ref{step: B1t1gap}) we show that every particle which does not descend from particle $(N,T)$ is to the left of $\XN(T) - c_2a_N$ at time $t_1$.

\begin{enumerate}[label=\emph{Step (\roman*).}, ref=\roman*, itemindent=8mm]
\setcounter{enumi}{3}
\item \label{step: B1t1gap} \emph{Let $i\in[N-1]$ and $j\in[N]$. If $(i,T)\lesssim (j,t_1)$ then $\X_j(t_1) \leq \XN(T) - c_2a_N$.}
\end{enumerate}

\noindent
First we will use Lemma~\ref{lemma: noRecord}(a) to bound $\XN(t_1)$. Since $T$ is the last time when a particle took the lead with a big jump before time $t_1$, we have $\bbracket{T,t_1-1}\subseteq \Sbf_N^c$, where $\Sbf_N$ is defined in \eqref{eq: RN}. By Corollary~\ref{cor: beatingLeader} and Steps~(\ref{step: B1tau1T}) and~(\ref{step: B1BigJumps}), it follows that $\llbracket T, t_1-1 \rrbracket \subseteq \hat \Sbf^c_N$. Therefore the conditions of Lemma~\ref{lemma: noRecord}(a) hold with $s = T$ and $\Delta s = t_1 - T$. Then Lemma~\ref{lemma: noRecord}(a) yields
\begin{equation}\label{eq: XNt1XNT}
\XN(t_1) \leq \XN(T) +  c_1  a_N.
\end{equation}

Now we prove the upper bound on $\X_j(t_1)$ in the statement of Step~(\ref{step: B1t1gap}). Let us first consider the case in which there is no big jump in the path between particles $(i,T)$ and~$(j,t_1)$, i.e. $B_N\cap P_{i,T}^{j,t_1} = \emptyset$. Then, by Lemma~\ref{lemma: C4}, Step~(\ref{step: B1Tgap}) and \eqref{eq: const4} we have 
$$\X_j(t_1) \leq \X_i(T) +  c_1  a_N < \XN(T) - \tfrac{3}{2}c_2a_N +  c_1  a_N < \XN(T) - c_2a_N,$$
which shows that the statement of Step~(\ref{step: B1t1gap}) holds in this case.

Now suppose instead that there exists a big jump on the path between particles $(i,T)$ and~$(j,t_1)$, so assume we have some $(l,b,r)\in B_N\cap P_{i,T}^{j,t_1}$. We will show that, even with the big jump $X_{l,b,r}$, particle $(j,t_1)$ cannot arrive close to the leader particle $(N,t_1)$ at time $t_1$. This fact together with~\eqref{eq: XNt1XNT} will imply Step~(\ref{step: B1t1gap}).

We know that $\llbracket T, t_1-1 \rrbracket \subseteq \hat \Sbf^c_N$, and so, in particular, the leader at time $r$ is not beaten by the big jump $X_{l,b,r}$. Hence by the definition of~$Z_l(r)$ in~\eqref{eq: Z} we have ${X_{l,b,r}\leq Z_l(r)}$. Therefore, because of Steps~(\ref{step: B1tau1T}) and~(\ref{step: B1BigJumps}) and by~\eqref{eq: const4}, Corollary~\ref{cor: gaps}(b) applies, which implies
\begin{equation}\label{eq: XkiXNc1}
\X_{l}(r) + X_{l,b,r} \leq \XN(r) - 2c_2a_N.
\end{equation}
Now by Lemma~\ref{lemma: C3C4} and since $t_1-T <\ell_N$ by Step~(\ref{step: B1tau1T}), then by~\eqref{eq: XkiXNc1}, and finally by monotonicity,
\begin{align}\label{eq: XjXNt1}
\X_j(t_1) \leq \X_{l}(r) + X_{l,b,r} +  c_1  a_N \leq \XN(r) - 2c_2a_N +  c_1  a_N \leq  \XN(t_1) - 2c_2a_N +  c_1  a_N.
\end{align}
Putting \eqref{eq: XjXNt1} and \eqref{eq: XNt1XNT} together and then using \eqref{eq: const4}, we obtain
\begin{align*}
\X_j(t_1) \leq \XN(T) - 2c_2a_N + 2 c_1  a_N \leq \XN(T) - c_2 a_N,
\end{align*}	
which finishes the proof of Step~(\ref{step: B1t1gap}).

\begin{enumerate}[label=\emph{Step (\roman*).}, ref=\roman*, itemindent=7mm]
\setcounter{enumi}{4}
\item \label{step: B1conclude} \emph{The event $\Bb_1$, as defined in \eqref{eq: B1}, occurs.}
\end{enumerate} 

\noindent
Let us simplify the notation by writing $R = R_{ c_1 ,N}(t_1)$, where $R_{ c_1 ,N}(t_1)$ is given by \eqref{eq: RepsN}. To prove that $\Bb_1$ occurs, we first show that
\begin{align}\label{eq: NNTt1}
\Nn_{N,T}(t_1) =\{j\in [N]:\X_j(t_1)\ge \X_N(t_1)-c_1 a_N\}= \bracing{N-R+1,\dots,N}.
\end{align}
The second equality follows directly from the definition of $R$; we will prove the first equality.

Note that Step~(\ref{step: B1t1gap}) implies that every descendant of particle $(N,T)$ survives until time $t_1$, that is~$|\Nn_{N,T}(t_1)| = 2^{t_1-T} > 1$. Indeed, by Step~(\ref{step: B1tau1T}) and our assumption on $N$ we have $2^{t_1-T} \leq 2N^{1-\delta} < N$, thus at time $t_1$ there are at least $2^{t_1-T}$ particles to the right of (or at) position $\XN(T)$ by Lemma~\ref{lemma: descendants_est}. By Step~(\ref{step: B1t1gap}), particles not descended from particle $(N,T)$ are to the left of position $\XN(T)$ at time $t_1$. Therefore, particle $(N,T)$ must have $2^{t_1-T}$ surviving descendants at time $t_1$, since otherwise there would not be $2^{t_1-T}$ particles to the right of (or at) position $\XN(T)$. 

The above argument also implies that the leader at time $t_1$ must be a descendant of particle $(N,T)$, i.e.~$N\in\Nn_{N,T}(t_1)$. Furthermore, as all jumps are non-negative, and by \eqref{eq: XNt1XNT}, we have
\begin{align}\label{eq: XNdescinterval}
\X_k(t_1) \in [\XN(T), \XN(T)+ c_1  a_N] \quad \forall k\in \Nn_{N,T}(t_1).
\end{align}
Hence, we must have $\X_k(t_1) \geq  \XN(t_1) -  c_1  a_N$ for all $k\in\Nn_{N,T}(t_1)$. 

By Step~(\ref{step: B1t1gap}) and then by monotonicity and~\eqref{eq: const4}, 
\begin{align*}
\X_j(t_1) \leq \XN(T) - c_2a_N < \XN(t_1) - c_1 a_N \quad \forall j\in[N]\setminus \Nn_{N,T}(t_1),
\end{align*}
and~\eqref{eq: NNTt1} follows.

Next we check that 
\begin{align}\label{eq: tribeGap}
\X_{N-R}(t_1) \leq \X_{N-R+1}(t_1) - c_2a_N.
\end{align}
By \eqref{eq: NNTt1} we see that $N-R+1 \in \Nn_{N,T}(t_1)$ and $N-R \notin \Nn_{N,T}(t_1)$. Therefore, Step~(\ref{step: B1t1gap}) and \eqref{eq: XNdescinterval} imply \eqref{eq: tribeGap}. 

Finally, we need to show that
\begin{align}\label{eq: RNdelta}
R \leq 2N^{1-\delta}.
\end{align}
We have that
\begin{align*}
R =|\bracing{N-R + 1,\dots,N}| = |\Nn_{N,T}(t_1)| \leq 2^{t_1 - (t_2+\ceil{\delta \ell_N})} \leq 2N^{1-\delta},
\end{align*}
where in the second equality we used \eqref{eq: NNTt1}, and the inequality follows since $T> t_2 + \ceil{\delta \ell_N}$ by Step~(\ref{step: B1tau1T}). Therefore by Step~(\ref{step: B1tau1T}), \eqref{eq: NNTt1}, \eqref{eq: tribeGap} and \eqref{eq: RNdelta}, $\Bb_1$ occurs, which concludes Step~(\ref{step: B1conclude}).

\vspace{3mm}

\noindent This completes the proof of \eqref{eq: CB1}.

\vspace{3mm}

\noindent \textbf{Proof that $\Cc_1$ to $\Cc_7$ imply $\Bb_2$}

\noindent Recall the definition of the event $\Bb_2$ in \eqref{eq: B2}.
We now prove that 
\begin{equation}\label{eq: CB2}
\Bb_1\cap\Cc_4\cap\Cc_6\cap\Cc_7 \subseteq \Bb_2,
\end{equation}
which implies $\bigcap_{j=1}^7 \Cc_j\subseteq \Bb_2$ because of \eqref{eq: CB1}.

Assume that $\Bb_1\cap\Cc_4\cap\Cc_6\cap\Cc_7$ occurs.
Again write $R = R_{ c_1 ,N}(t_1)$, where $R_{ c_1 ,N}(t_1)$ is defined using \eqref{eq: RepsN}. Take $j\in [N-R]$ and consider particle $(j,t_1)$. Then, by the definition of the event $\Bb_1$ in~\eqref{eq: B1}, and since the leader at time $t_1$ is to the right of every particle at time $t_1$, we have
\begin{equation}\label{eq: XNjXNR}
\X_j(t_1) \leq \X_{N-R+1}(t_1) - c_2a_N \leq \XN(t_1) - c_2a_N.
\end{equation}
Now suppose that the $i$th particle at time $t$ is a descendant of particle $(j,t_1)$, i.e. $i\in \Nn_{j,t_1}(t)$. Lemma~\ref{lemma: descendants_est} implies that every particle at time $t$ is to the right of (or at) $\XN(t_1)$. Thus we have $\X_i(t) \geq \XN(t_1)$, which together with \eqref{eq: XNjXNR} and~\eqref{eq: const4} implies
\begin{align*}
\X_i(t) > \X_j(t_1) +  c_1  a_N.
\end{align*}
Thus, by Lemma~\ref{lemma: C4}, there must be a big jump in the path between particles $(j,t_1)$ and $(i,t)$; that is, we must have $B_N\cap P_{j,t_1}^{i,t} \neq \emptyset$.

Therefore we can bound the number of time-$t$ descendants of particles $(1,t_1),(2,t_1),\ldots , (N-R,t_1)$ by the number of descendants of particles which made a big jump between times $t_1$ and $t-1$:
\begin{align}\label{eq: bigJumpDesc3}
\sum_{j=1}^{N - R} |\Nn_{j,t_1}(t)|\leq \sum_{(k,b,s)\in B_N^{[t_1,t-1]}} |\Nn_{k,s}^b(t)|.
\end{align}
By the definition of the event $\Cc_6$, no particle makes a big jump in the time interval $[t_1-\ceil{\delta \ell_N},t_1+\ceil{\delta \ell_N}]$. Hence, any particle which made a big jump between times $t_1$ and $t-1$ can have at most $2^{t - (t_1+\ceil{\delta\ell_N})}$ descendants at time~$t$. Furthermore, by the definition of~$\Cc_7$, $|B_N^{[t_1,t-1]}|\le K$. Putting these observations together with~\eqref{eq: bigJumpDesc3} we obtain
\begin{align} \label{eq:bounddescbig}
\sum_{j=1}^{N - R} |\Nn_{j,t_1}(t)|\leq 2KN^{1-\delta} < N^{1-\gamma},
\end{align}
by our assumption on $N$ in the statement of the proposition. This completes the proof of \eqref{eq: CB2}.

\vspace{3mm}

\noindent \textbf{Proof that $\Cc_1$ to $\Cc_7$ imply $\Aa_1$}

\noindent Recall the definition of
$\Aa_1$ in \eqref{eq: A1}. We now complete the proof of Proposition~\ref{prop: B} by showing that 
\begin{equation}\label{eq: CA1}
\bigcap_{j=1}^{7}\Cc_j \subseteq \Aa_1.
\end{equation}
Assume $\bigcap_{j=1}^{7}\Cc_j$ occurs.
Let $i,j\in[N]$ be such that $(j,t_1) \lesssim (i,t)$. 
Assume first that $B_N\cap P_{j,t_1}^{i,t} = \emptyset$. Then, by Lemma~\ref{lemma: C4} and using the leader's position as an upper bound, we obtain
\begin{equation*}
\X_i(t) \leq \X_j(t_1) +  c_1  a_N \leq \XN(t_1) +  c_1  a_N \le \X_1(t)+c_1 a_N,
\end{equation*}
where the last inequality follows by Lemma~\ref{lemma: descendants_est}.
Thus, recalling the definition  of $L_{ c_1 ,N}(t)$ in~\eqref{eq: LepsN}, we have $i\in[L_{ c_1 ,N}(t)]$. Therefore, if 
$i>L_{c_1,N}(t)$ then we must have $B_N\cap P_{j,t_1}^{i,t} \neq \emptyset$.
It follows that
$$
N-L_{c_1,N}(t) \le \sum_{(k,b,s)\in B_N^{[t_1,t-1]}} |\mathcal N^b_{k,s}(t)|<N^{1-\gamma}
$$
by the same argument as for~\eqref{eq:bounddescbig}.
Since we took $c_1<\eta$ in~\eqref{eq: const4}, we now have $L_{\eta,N}(t)\ge N-N^{1-\gamma}$, which
finishes the proof of~\eqref{eq: CA1}. The proof of Proposition~\ref{prop: B} then follows from \eqref{eq: CB1}, \eqref{eq: CB2},~\eqref{eq: CA1} and Step~(\ref{step: B1BigJumps}).
\end{proof}

\subsection{Breaking down event $\Cc_1$}\label{sect: C1}
We have now broken down the events $\Bb_1$, $\Bb_2$ and $\Aa_1$ into simpler events $\Cc_1$ to $\Cc_7$. In Section~\ref{sect: probabilities} we will be able to show directly that the events $\Cc_2$ to $\Cc_7$ occur with high probability. However, we will need to break $\Cc_1$ down further, into simpler events that we will show occur with high probability in Section \ref{sect: probabilities}. In this section we carry out the task of breaking down $\Cc_1$, which says that a gap of size $2c_3a_N$ appears behind the rightmost particle at some point during the time interval $[t_2+1, t_1]$ (see~\eqref{eq: C1}), into simpler events. 
Recall that we assumed $t>4\ell_N$, and that the constants $\eta\in(0,1]$,  $\gamma$, $\delta$, $\rho$, $c_1,c_2,\dots,c_6$, $K$ satisfy \eqref{eq: const1}-\eqref{eq: const5}.

The first event we introduce is the same as the event $\Cc_2$ in \eqref{eq: C2}, except with larger gaps and jumps. That is, if a particle is more than $c_4a_N$ away from the leader, then it does not jump to within distance $3c_3a_N$ of the leader's position with a single big jump (recall that $c_3\ll c_4$). We let
\begin{equation}\label{eq: D1}
\Dd_1 := \bracing{\begin{array}{l}
\nexists(i,b,s)\in\Nonetwo\times\bbracket{t_3,t-1} \text{ such that } \\
X_{i,b,s}\in (Z_i(s) - 3c_3a_N, Z_i(s) + 3c_3a_N] \text{ and }Z_i(s)\geq c_4 a_N
\end{array}},
\end{equation}
where $Z_i(s)$ is the gap between the $i$th and the rightmost particle. The reason behind the definition of $\Dd_1$ is the following. Assume that a big jump beats the leader at a time when the diameter is fairly big ($>\frac{3}{2}c_4a_N$). Then the event $\Dd_1$, together with the events $\Cc_3$, $\Cc_4$ and $\Cc_5$, implies that this particle must become the new leader and it will lead by at least $(3c_3 -\rho) a_N$, which will be enough to show that $\Cc_1$ occurs. We state this as a corollary below, which we will use later on in this section.

\begin{corollary}\label{cor: gaps2}
On the event~$\Dd_1\cap \Cc_3 \cap \Cc_4\cap \Cc_5$, if $(k,b,s)\in B_N^{[t_3,t-1]}$, $d(\X(s)) \geq (c_4+c_1)a_N$ and $X_{k,b,s} > Z_k(s)$, then
\begin{equation*}
\XN(s+1) = \X_k(s) + X_{k,b,s} > \X_{N-1}(s+1) + (3c_3 - \rho)a_N,
\end{equation*}
where $Z_k(s),\Dd_1$ and $\Cc_3,\Cc_4,\Cc_5$ are given by \eqref{eq: Z}, \eqref{eq: D1} and \eqref{eq: C3}-\eqref{eq: C5} respectively, and $B_N^{[t_3,t-1]}$ is defined in \eqref{eq: BN1}.
\end{corollary}
\begin{proof}
The statement follows by exactly the same argument as for Corollary~\ref{cor: gaps}(a), if we replace $\Cc_2$ by $\Dd_1$, $c_3$ by~$c_4$ and $2c_2$ by $3c_3$.
\end{proof}

The next two events will ensure that the record is broken in the time interval $[t_2+1,t_1]$. The first event says that there is a jump of size greater than $2c_4a_N$ in every interval of length $c_5\ell_N$ in $[t_3, t_1]$ (recall $c_4 \ll c_5$). We define
\begin{equation}\label{eq: D2}
\Dd_2 := \bracing{\forall s\in\bbracket{t_3,t_1-c_5\ell_N},\; \exists(k,b,\hat{s})\in\Nonetwo\times\bbracket{s,s+c_5\ell_N}:\: X_{k,b,\hat{s}} > 2c_4a_N}.
\end{equation}
The event $\Dd_2$ will be useful if at some point in the time interval $[t_2,t_1]$ the diameter is not too large ($\le \frac{3}{2}c_4a_N$). If $\Dd_2$ occurs then shortly after this point a jump of size larger than $2c_4a_N$ happens. We will show that this jump breaks the record, and the particle performing this jump will lead by at least $2c_3a_N$. The reason for this is that the jump size~($> 2c_4a_N$) is much greater than the preceding diameter ($\le \frac{3}{2}c_4a_N$), and that $c_3\ll c_4$.

The next event says that there will be a jump of size greater than $2c_6a_N$ between times $t_2$ and $t_2 + \ceil{\ell_N/2}$ (recall~$c_6\gg c_5$). Let
\begin{equation}\label{eq: D3}
\Dd_3 := \bracing{\exists (i,b,s)\in\Nonetwo \times \bbracket{t_2, t_2 + \ceil{\ell_N/2}}:\: X_{i,b,s} > 2c_6a_N}.
\end{equation}
The next event says that there is no jump of size greater than $c_6a_N$ shortly before time $t_2$. We let
\begin{equation}\label{eq: D4}
\Dd_4 := \bracing{\nexists (i,b,s)\in\Nonetwo\times \bbracket{t_2 - \ceil{c_5\ell_N}, t_2}:\: X_{i,b,s} > c_6a_N}
\end{equation}
(recall $c_5\ll c_6$). 
Our last event excludes jumps of size in a certain small range in a certain short time interval. The starting point of this time interval will be the first time after $t_2$ when the diameter is at most $\frac{3}{2} c_4a_N$:
\begin{equation}\label{eq: tau}
\tau_2 := \inf\bracing{s\geq t_2:\: d(\X(s)) \leq \tfrac{3}{2}c_4a_N},
\end{equation}
and we define the event
\begin{equation}\label{eq: D5}
\Dd_5 := \bracing{\begin{array}{l}
\nexists (k,b,s)\in\Nonetwo \times \bbracket{\tau_2,\tau_2 + c_5\ell_N}:\: \\ X_{k,b,s}\in (2c_4a_N,2c_4a_N + 3c_3a_N]
\end{array}}.
\end{equation}

We can now state the main result of this subsection.

\begin{prop}\label{prop: C}
Let $\eta\in(0,1]$, and assume that the constants $\gamma,\delta,\rho, c_1 ,c_2,\dots,c_6,K$ satisfy \eqref{eq: const1}-\eqref{eq: const5}. For all $N\geq 2$ sufficiently large that $\ell_N-\lceil c_5 \ell_N \rceil \ge \lceil \ell_N/2\rceil $ and $t>4\ell_N$,
\begin{align*}
\bigcap_{j=2}^7 \Cc_j \cap \bigcap_{i=1}^5 \Dd_i \subseteq \Cc_1,
\end{align*}
where $\Dd_1,\dots,\Dd_5$
are defined in \eqref{eq: D1}-\eqref{eq: D4} and \eqref{eq: D5} respectively, and $\Cc_1,\dots,\Cc_7$ are defined in~\eqref{eq: C1} and \eqref{eq: C2}-\eqref{eq: C7} respectively.
\end{prop}

Before giving a precise proof of Proposition~\ref{prop: C}, we give an outline of the argument, which is divided into four separate cases.
Suppose $\bigcap_{j=2}^7 \Cc_j \cap \bigcap_{i=1}^5 \Dd_i$ occurs.\\

\noindent
\textbf{Case 1:} Suppose there is a time $\tau_2\in[t_2,t_1-c_5\ell_N]$ when the diameter is not too large (at most $\frac{3}{2}c_4a_N$). 
Then shortly after time $\tau_2$, there will be a jump of size larger than $2c_4a_N$, by the definition of the event $\Dd_2$. We will show that the particle making this jump breaks the record and will lead by a distance larger than $2c_3a_N$. The proof will also use the definition of the event $\Dd_5$.\\

\noindent
\textbf{Case 2(a):} Suppose the diameter is larger than $\frac{3}{2}c_4a_N$ at all times in $[t_2,t_1-c_5\ell_N]$, but the record is broken by a big jump at some point in this time interval.
Then Corollary~\ref{cor: gaps2} tells us that there will be a gap of size greater than $2c_3a_N$ behind the new record.\\

\noindent
\textbf{Case 2(b):} Suppose the diameter is larger than $\frac{3}{2}c_4a_N$ at all times in $[t_2,t_1-c_5\ell_N]$. If the record is not broken on the time interval $[t_2-\ceil{c_5\ell_N},t_1-c_5\ell_N]$,
then using Lemma~\ref{lemma: noRecord}, we can show that the diameter is less than $\frac 32 c_4  a_N$ at time $t_1-\ceil{c_5\ell_N}$, giving us a contradiction. Thus this case is impossible.\\

\noindent
\textbf{Case 2(c):} Suppose the diameter is larger than $\frac{3}{2}c_4a_N$ at all times in $[t_2,t_1-c_5\ell_N]$. Now consider the case that the record is not broken on the time interval $[t_2,t_1-c_5\ell_N]$, but is broken shortly before $t_2$, during the time interval $[t_2-\ceil{c_5\ell_N},t_2-1]$. By the definition of the event $\Dd_4$, this jump cannot be very big. Therefore, we will see that the new leader will be beaten by the first jump of size greater than $2c_6a_N$, if the record has not already been broken before that. There will be a jump of size greater than $2c_6a_N$ before time $t_2 + \ceil{\ell_N/2}$ because of the event $\Dd_3$, so the record must be broken by a big jump before time $t_1-c_5 \ell_N$. This again gives us a contradiction, meaning that Case~2(c) is also impossible.\\

We now prove Proposition~\ref{prop: C}, using cases 1, 2(a), 2(b) and 2(c) as described above.

\begin{proof}[Proof of Proposition~\ref{prop: C}]
Fix $\eta\in(0,1]$ and take constants $\gamma,\delta,\rho, c_1 ,c_2,\dots,c_6,K$ as in \eqref{eq: const1}-\eqref{eq: const5}. Let us assume that $\bigcap_{j=2}^7 \Cc_j \cap \bigcap_{i=1}^5 \Dd_i$ occurs.\\
	
\noindent	
\textbf{Case 1:} $t_2\leq\tau_2 \leq t_1-c_5 \ell_N$.\\
\noindent
In this case, by the definition of $\tau_2$ we have
\begin{equation}\label{eq: dXtau}
d(\X(\tau_2))\leq\tfrac{3}{2}c_4a_N.
\end{equation}
Let us now consider the first jump of size greater than $2c_4a_N$ after time $\tau_2$; that is, let
\begin{align}\label{eq: s*}
s^* = \inf\bracing{s \geq \tau_2: \: \exists (k,b)\in\Nonetwo \text{ such that }X_{k,b,s} > 2c_4a_N}
\in \llbracket \tau_2 , \tau_2 +c_5 \ell_N \rrbracket 
\end{align}
by the definition of the event $\Dd_2$ in~\eqref{eq: D2}.
Take $(k^*,b^*)\in\Nonetwo$ such that $X_{k^*,b^*,s^*} > 2c_4a_N$ (there is a unique choice of the pair $(k^*,b^*)$ by the definition of the event~$\Cc_5$).
We will show that the jump $X_{k^*,b^*,s^*}$ creates a gap of size larger than $2c_3a_N$ behind the leader. We do this in two steps. First we show that the diameter is not too large right before the jump $X_{k^*,b^*,s^*}$ occurs; then we show that a gap is created.

\begin{enumerate}[label=(\roman*), ref=\roman*]
\item We claim that
\begin{equation}\label{eq: dXs1Case1}
d(\X(s^*)) \leq 2c_4 a_N + c_2 a_N.
\end{equation} 
Now we prove the claim. 
By~\eqref{eq: dXtau} we can assume $s^*>\tau_2$.
Let $j\in [N]$ be arbitrary, and then take $i\in [N]$ such that $(i,\tau_2)\lesssim (j,s^*)$. We will show that particle $(j,s^*)$ is within distance $(2c_4+c_2)a_N$ of the leftmost particle at time $s^*$. We consider two cases, depending on whether there is a big jump on the path between $\X_i(\tau_2)$ and $\X_j(s^*)$.
\begin{itemize}
\item If $B_N\cap P_{i,\tau_2}^{j,s^*} = \emptyset$, then by Lemma~\ref{lemma: C4}, \eqref{eq: dXtau} and monotonicity,
\begin{align*}
\X_j(s^*)  \leq \X_i(\tau_2) +  c_1  a_N  &
\leq \Xone(\tau_2) + \tfrac{3}{2}c_4a_N +  c_1  a_N \leq \Xone(s^*) + \tfrac{3}{2}c_4a_N +  c_1  a_N.\numberthis\label{eq: XjX1Case1a}
\end{align*}

\item  If $B_N\cap P_{i,\tau_2}^{j,s^*}\neq \emptyset$, then take $(k',b',s')\in B_N\cap P_{i,\tau_2}^{j,s^*}$. Then $\X_{k'}(s')$ is the position of the parent of the particle that makes the jump $X_{k',b',s'}$. Since (by~\eqref{eq: s*}) $X_{k^*,b^*,s^*}$ is the first jump of size greater than $2c_4a_N$ after time~$\tau_2$, and since $s'<s^*$, we have $X_{k',b',s'} \leq 2c_4a_N$. 
Then since $s^*-s'\le s^* -\tau_2 \le c_5 \ell_N$, by Lemma~\ref{lemma: C3C4} we have
\begin{equation*}
\X_j(s^*) \leq \X_{k'}(s') + X_{k',b',s'} +  c_1  a_N \leq \X_{k'}(s') + 2c_4a_N + c_1  a_N.
\end{equation*}
Now Lemma~\ref{lemma: bigJumpLeftmost} and monotonicity imply that this is at most
\begin{equation}
\Xone(s') + 2c_4a_N + 2 c_1  a_N \le \Xone(s^*) + 2c_4a_N + 2 c_1  a_N.\label{eq: XjCase1b}
\end{equation}
\end{itemize} 
By~\eqref{eq: XjX1Case1a},~\eqref{eq: XjCase1b} and our choice of constants in \eqref{eq: const4}, we conclude that for any particle position $\X_j(s^*)$ in the population at time $s^*$, $\X_j(s^*) \leq \Xone(s^*) + 2c_4a_N + c_2a_N$, which implies \eqref{eq: dXs1Case1}.

\item 
We claim that
\begin{equation} \label{eq:propCstepii} 
\X_{N-1}(s^*+1) + 2c_3a_N < \XN(s^*+1).
\end{equation}
By the definition of $(k^*,b^*,s^*)$, we have $X_{k^*,b^*,s^*} > 2c_4a_N$, and we also know~$X_{k^*,b^*,s^*}\notin(2c_4a_N, 2c_4a_N+3c_3a_N]$ by the definition of the event $\Dd_5$, because $s^*\in\bbracket{\tau_2,\tau_2+c_5\ell_N}$. Therefore we have
\begin{equation}\label{eq: Xks1}
X_{k^*,b^*,s^*} > 2c_4a_N+3c_3a_N.
\end{equation}
Then by~\eqref{eq: Xks1} and \eqref{eq: dXs1Case1},
\begin{equation}\label{eq: XlXNs1}
 \X_{k^*}(s^*) + X_{k^*,b^*,s^*} > \Xone(s^*) + (2c_4 + 3c_3)a_N \geq \XN(s^*) + (3c_3-c_2)a_N.
\end{equation}
Note that $3c_3-c_2>\rho$ by \eqref{eq: const2}-\eqref{eq: const4}. Hence by~\eqref{eq: Xks1}, \eqref{eq: XlXNs1} and Lemma~\ref{lemma: breakRecordGap}(b), we have $(k^*,s^*)\lesssim_{b^*}(N,s^*+1)$ and
\begin{equation*}
\XN(s^*+1) > \X_{N-1}(s^* + 1) + (3c_3 - c_2 - \rho)a_N,
\end{equation*}
which is larger than $\X_{N-1}(s^* + 1) + 2c_3a_N$ by \eqref{eq: const2}-\eqref{eq: const4}.
This finishes the proof of~\eqref{eq:propCstepii}.
\end{enumerate}

\noindent Recall from~\eqref{eq: s*} that $s^*\in\bbracket{\tau_2,\tau_2+c_5\ell_N}$. Furthermore, event $\Cc_6$ tells us that $s^*\notin [ t_1 - \ceil{\delta\ell_N},t_1]$. Therefore, by the assumption of Case 1 that $\tau_2 \in [t_2, t_1-c_5 \ell_N]$,  we conclude $t_2+1 \leq s^* + 1 \leq t_1$, which together with~\eqref{eq:propCstepii} shows that $\Cc_1$ occurs. We conclude that Proposition~\ref{prop: C} holds in Case 1.\\




\noindent
\textbf{Case 2(a):} $\tau_2 > t_1 - c_5\ell_N$ and $[t_2,t_1 - c_5\ell_N] \cap \hat \Sbf_N \neq \emptyset$, where $\hat \Sbf_N$ is defined in \eqref{eq: RNhat}.

\noindent This means that there exists $(\hat{k},\hat{b},\hat{s})\in B_N^{[t_2,t_1 - c_5\ell_N]}$ with $X_{\hat{k},\hat{b},\hat{s}} > Z_{\hat k}(\hat s)$ (recall \eqref{eq: Z}). Since $\tau_2 > t_1 - c_5\ell_N$, we have $d(\X(\hat s)) > \frac{3}{2}c_4a_N$. Then by~\eqref{eq: const4}, we can apply Corollary~\ref{cor: gaps2} to obtain
\begin{align*}
\XN(\hat s + 1) = \X_{\hat k}(\hat s) + X_{\hat{k},\hat{b},\hat{s}} > \X_{N-1}(\hat s+1) + (3c_3 - \rho)a_N.
\end{align*}
By our choice of constants in \eqref{eq: const2}-\eqref{eq: const4}, and because $\hat s + 1\in\bbracket{t_2+1,t_1}$, this shows that~$\Cc_1$ occurs. Therefore we are done with the proof of Proposition~\ref{prop: C} in Case 2(a).\\

\noindent
\textbf{Case 2(b): } $\tau_2 > t_1 - c_5\ell_N$ and $[t_2-\lceil c_5\ell_N\rceil,t_1 -  c_5\ell_N] \cap \hat \Sbf_N = \emptyset$.

\noindent
We will apply Lemma~\ref{lemma: noRecord} with $s = t_2 - \lceil c_5\ell_N\rceil$ and $\Delta s = \ell_N$. By assumption we have $[s,s+\Delta s-1]\subseteq \hat \Sbf_N^c$, and therefore applying either part (a) or part (b) of Lemma~\ref{lemma: noRecord} as appropriate, we have
\begin{equation*}
d(\X(s+\Delta s)) = d(\X(t_1-\lceil c_5 \ell_N\rceil)) \leq  \max\big\{c_1 a_N, \tfrac{3}{2} c_3 a_N + 2 c_1 a_N\big\}
\end{equation*}
which is smaller than $\tfrac{3}{2}c_4a_N$ by~\eqref{eq: const4}, contradicting the assumption that $\tau_2 > t_1 - c_5\ell_N$. This shows that Case 2(b) cannot occur.\\

\noindent
\textbf{Case 2(c): } $\tau_2 > t_1 - c_5\ell_N$ and $[t_2,t_1 - c_5\ell_N ] \cap \hat \Sbf_N = \emptyset$, but $[t_2-\lceil c_5\ell_N\rceil,t_2 - 1 ] \cap \hat \Sbf_N \neq \emptyset$.

\noindent
Define
\begin{equation}\label{eq: T2}
\tau_3 :=\inf \big\{s\leq t_2 :\bbracket{s,t_2} \subseteq \hat \Sbf_N^c\big\} \in (t_2-\lceil c_5\ell_N\rceil,t_2].
\end{equation}
Suppose, aiming for a contradiction, that there exists~$r\in \bbracket{{\tau_3},{t_2}}$ such that $d(\X(r)) \leq \frac{3}{2}c_3a_N$. Then since $\llbracket \tau_3,t_2 \rrbracket \subseteq \hat \Sbf_N^c$, Lemma~\ref{lemma: noRecord}(b) applies and says that~$d(\X(t_2))\leq \frac{3}{2}c_3a_N + 2 c_1  a_N$. By~\eqref{eq: const4}, this contradicts the assumption that $\tau_2>t_1-c_5\ell_N$. Thus (by~\eqref{eq: const4} again for $r>t_2$) we must have
\begin{align}\label{eq: dXr}
d(\X(r)) \geq \tfrac{3}{2}c_3a_N \quad \forall r\in \llbracket \tau_3,t_1-c_5\ell_N \rrbracket .
\end{align}
Now note that $\tau_3 - 1 \in \hat \Sbf_N$. Then by~\eqref{eq: dXr}, the second equivalence in Corollary~\ref{cor: beatingLeader} implies that in fact $\tau_3 - 1 \in \Sbf_N$. Hence, by the definition of $\Sbf_N$ in~\eqref{eq: RN}, there exists $(k,b)\in\Nonetwo$ such that
\begin{equation}\label{eq: XNtau3}
\XN(\tau_3) = \X_k(\tau_3 - 1) + X_{k,b,\tau_3 -1},
\end{equation}
where $X_{k,b,\tau_3 -1} > \rho a_N$. Now Lemma~\ref{lemma: bigJumpLeftmost} provides a bound on $\X_k(\tau_3 - 1)$, and the definition of $\Dd_4$ together with the fact that $\tau_3 - 1 \in [ t_2 - \ceil{c_5\ell_N},t_2]$ gives us a bound on $X_{k,b,\tau_3-1}$, so that we obtain
\begin{equation}\label{eq: XNT2}
\XN(\tau_3) \leq \Xone(\tau_3 - 1) + ( c_1  + c_6)a_N.
\end{equation}

Now, on the event $\Dd_3$, there exists $(\tilde i,\tilde b,\tilde s)\in [N]\times\{1,2\}\times\llbracket t_2,t_2+\lceil \ell_N/2\rceil\rrbracket$ such that 
\begin{equation}\label{eq: stilde}
X_{\tilde i,\tilde b,\tilde s}>2c_6 a_N>\rho a_N
\end{equation}
by~\eqref{eq: const2}-\eqref{eq: const4}.
We show that the particle performing this big jump beats the leader at time $\tilde s$. By our assumption that $\ell_N-\lceil c_5 \ell_N \rceil \ge \lceil \ell_N/2\rceil $ and by~\eqref{eq: T2}, we have $\bbracket{\tau_3,\tilde s} \subseteq \hat \Sbf_N^c$ and $\tilde s - \tau_3 \leq \ell_N$. Therefore, by~\eqref{eq: dXr} we can apply Lemma~\ref{lemma: noRecord}(a) with $s=\tau_3$ and $\Delta s=\tilde s - \tau_3$, and then by \eqref{eq: XNT2} we have 
\begin{equation}\label{eq: XNT3}
\XN(\tilde s) \le \XN(\tau_3)+c_1 a_N \le \X_1(\tau_3-1) + (2c_1 + c_6)a_N.
\end{equation}
By~\eqref{eq: const4}, it follows that
\begin{equation*}
\XN(\tilde{s}) < \Xone(\tau_3-1) + 2c_6a_N  <  \Xone(\tilde{s}) + X_{\tilde{i},\tilde{b},\tilde{s}} \leq \X_{\tilde{i}}(\tilde{s}) + X_{\tilde{i},\tilde{b},\tilde{s}},
\end{equation*}
where in the second inequality we use monotonicity and~\eqref{eq: stilde}. Therefore, by the assumptions that $\tilde s\in\llbracket t_2,t_2+\lceil \ell_N/2\rceil\rrbracket$ and $\ell_N-\lceil c_5 \ell_N \rceil \ge \lceil \ell_N/2\rceil $, and by the definition of $\hat \Sbf_N$ in~\eqref{eq: RNhat}, we have ${\tilde{s}\in \hat \Sbf_N\cap[t_2,t_1-c_5\ell_N]}$, which contradicts the assumption of Case 2(c).

We have now shown that if $\bigcap_{j=2}^7 \Cc_j \cap \bigcap_{i=1}^5 \Dd_i $ occurs then Cases 2(b) and 2(c) are impossible, whereas Cases 1 and 2(a) imply that $\Cc_1$ must occur. This concludes the proof of Proposition~\ref{prop: C}.
\end{proof}

\section{Probabilities of the events from the deterministic argument}\label{sect: probabilities}
In the deterministic argument in Section~\ref{sect: deterministic} we have provided a strategy which ensures that the events $\Aa_1$ and $\Aa_3$ occur. In this section we check that the events $\Cc_2$ to $\Cc_7$ and $\Dd_1$ to $\Dd_5$ which make up this strategy all occur with high probability, and use this to finish the proof of Proposition~\ref{prop: A1A3}.

When bounding the probabilities of these events, it will be useful to consider branching random walks (BRWs) without selection, where at each time step all particles have two offspring, the offspring particles make i.i.d.~jumps from their parents' locations, and every offspring particle survives. Below we describe a construction of the $N$-BRW from $N$ independent BRWs, which will allow us to consider our events on the probability space on which the BRWs are defined.
(A similar construction was used in~\cite{Berard2010}.)

\subsection{Construction of the $N$-BRW from $N$ independent BRWs}\label{sect: BRW}
Consider a binary tree with the following labelling. Let 
\begin{align*}
\U_0 := \bigcup_{n=0}^\infty \bracing{1,2}^n,
\end{align*}
and for convenience we write e.g. $121$ instead of $(1,2,1)$. Then the root of the binary tree has label $\emptyset$, and for all $u\in\U_0$ the two children of vertex $u$ have labels $u1$ and $u2$. We will use the partial order $\preceq$ on the set $\U_0$; we write $u\preceq v$ if either $u=v$ or the vertex with label $u$ is an ancestor of the vertex with label $v$ in the binary tree. We also write $u\prec v$ if $u\preceq v$ and $u\neq v$.

The particles of the $N$ independent BRWs will have labels from the set $[N]\times\U_0$, and we have a lexicographical order on the set of labels. We also let $\U := \U_0\setminus\bracing{\emptyset}$. The jumps of the BRWs will be given by random variables $(Y_{j,u})_{j\in[N], u\in\U}$, which are i.i.d.~with common law given by \eqref{eq: poly_tail}.

The $N$ initial particles of the $N$ independent BRWs are labelled with the pairs $(j,\emptyset)$ with $j\in[N]$. For each $j\in[N]$, we let $\Y_j(\emptyset)\in\R$ be the initial location of particle $(j,\emptyset)$. Then, at each time step $n\in \N_0$, each particle $(j,u)$ with $j\in [N]$ and $u\in \{1,2\}^n$ has two offspring labelled $(j,u1)$ and $(j,u2)$, which make jumps $Y_{j,u1}$, $Y_{j,u2}$ from the location $\Y_j(u)$. The locations of the offspring particles $(j,u1)$ and $(j,u2)$ will be $\Y_j(u1) = \Y_j(u) + Y_{j,u1}$ and $\Y_j(u2) = \Y_j(u) + Y_{j,u2}$. Note that for $u\prec v$, the path between particles $(j,u)$ and $(j,v)$ is given by the jumps $Y_{j,w}$ with $u\prec w \preceq v$, i.e.~ $\Y_j(v) - \Y_j(u) = \sum_{u \prec w \preceq v} Y_{j,w}$.

Now we construct the $N$-BRW by defining the surviving set of particles for each time $n\in\Nzero$ as the $N$-element set $H_n \subseteq [N]\times \bracing{1,2}^n$, constructed iteratively as follows. Let $H_0 := \bracing{(1,\emptyset),\dots,(N,\emptyset)}$. Given $H_n$ for some $n\in\N_0$, we let $H'_n$ denote the set of offspring of the particles in the set $H_n$:
\begin{align*}
H'_n := \bigcup_{(j,u)\in H_n}\bracing{(j,u1),(j,u2)}.
\end{align*}
Then $H_{n+1}\subseteq H_n'$ consists of the particles with the $N$ largest values in the collection $(\Y_j(u))_{(j,u)\in H_n'}$, where ties are broken based on the lexicographical order of the labels. In this way an $N$-BRW is constructed from the initial configuration $(\Y_j(\emptyset))_{j\in[N]}$ and the jumps $(Y_{j,u})_{j\in [N], u\in \U}$.

For $n\in \N$, we let $\mathcal F '_n$ denote the $\sigma$-algebra generated by $(Y_{j,u})_{j\in [N], u \in \cup_{m=1}^n \{1,2\}^m}$. Note that $H_n$ is $\mathcal F'_n$-measurable for each $n$.

Returning to our original notation in Section~\ref{sect: NBRWdef}, we can say the following. For all $n\in\Nzero$, let $\X(n)$ denote the ordered set which contains the values $(\Y_j(u))_{(j,u)\in H_n}$ in ascending order:
\begin{align}\label{eq: XnY}
\X(n) = \bracing{\Xone(n)\leq\dots\leq\XN(n)} := \bracing{\Y_{j_1}(u_1)\leq\dots\leq\Y_{j_N}(u_N)},
\end{align}
where $H_n=\{(j_i,u_i):i\in [N]\}$, and again ties are broken based on the lexicographical order of the labels. Then we define the map $\sigma$ which associates the pair $(i,n)\in [N]\times \Nzero$ with particle $(j_i,u_i)\in H_n$, where  $\Y_{j_i}(u_i)$ has the $i$th position in the ordered set $\X(n)$. That is, for $(i,n)\in [N]\times \Nzero$ we let
\begin{equation}\label{eq: sigma}
\sigma(i,n) = (j_i,u_i)\in H_n \subset [N]\times\U_0,
\end{equation}
where $(j_i,u_i)$ is as in \eqref{eq: XnY}. The jumps in our original notation are then given by
\begin{equation}\label{eq: XY}
X_{i,1,n} := Y_{j_i,u_i1} \quad \text{ and } \quad X_{i,2,n} := Y_{j_i,u_i2},
\end{equation}
if $\sigma(i,n) = (j_i,u_i)$. 

Finally, recall that we introduced the partial order $\lesssim$ in \eqref{eq: ancestor} in Section~\ref{sect: notation} to denote that two particles are related in the $N$-BRW. This partial order corresponds to the partial order $\preceq$ in the $N$ independent BRWs as follows. For all $n,k\in\Nzero$ and $i_0,i_k\in[N]$, we have $(i_0,n) \lesssim (i_k,n+k)$ if and only if 
for some $j\in[N]$ and $u,v\in\U_0$, we have $\sigma(i_0,n) = (j,u)$, $\sigma(i_k,n+k) = (j,v)$, and $u\preceq v$. Furthermore, for $b\in\bracing{1,2}$ we have $(i_0,n)\lesssim_b (i_k,n+k)$ if and only if the above holds and additionally $k\geq 1$ and $ub \preceq v$.

Now we can consider the $N$-BRW constructed from $N$ independent BRWs with the notation introduced in Sections~\ref{sect: NBRWdef} and \ref{sect: notation}. It follows from our construction that for any path in the $N$-BRW, there is a path in one of the $N$ independent BRWs that consists of the same sequence of jumps as the path in the $N$-BRW. We state and prove this simple property below. Recall the notation $P_{i_1,n_1}^{i_2,n_2}$ from \eqref{eq: P}.

\begin{lemma}\label{lemma: BRWpath}
For all $k\in\N$, $i_0,i_k\in[N]$ and $n\in\Nzero$, if $(i_0,n)\lesssim (i_k,n+k)$ with $P_{i_0,n}^{i_k,n+k}=\{(i_l,b_l,n+l):l\in \{0,\ldots,k-1\}\}$, then there exists $j\in [N]$ and $(u_l)_{l=0}^k \subseteq \U_0$ such that
\begin{enumerate}[label=(\arabic*), ref=\arabic*]
	\item $(j,u_l) \in H_{n+l}$, for all $l \in \bracing{0,\dots,k}$,
	\item $u_l b_l\preceq u_k$, for all $l \in \bracing{0,\dots,k-1}$, and
	\item $X_{i_l,b_l,n+l} = Y_{j,u_lb_l}$, for all $l \in \bracing{0,\dots,k-1}$.
\end{enumerate}
\end{lemma}

\begin{proof}
Take $(i_l,b_l,n+l) \in P_{i_0,n}^{i_k,n+k}$ (with $l\in\bracing{0,\dots,k-1}$). Then $(i_l,n + l) \lesssim_{b_l} (i_k,n+k)$. Thus, there exist $j\in[N]$ and $u_l,u_k\in\U_0$ such that $\sigma(i_l,n+l) = (j,u_l)$, $\sigma(i_k,n+k) = (j,u_k)$, and $u_l b_l\preceq u_k$. This implies $X_{i_l,b_l,n+l} = Y_{j,u_lb_l}$ (see \eqref{eq: XY}) and also $(j,u_l)\in H_{n+l}$ and $(j,u_k)\in H_{n+k}$ by the definition \eqref{eq: sigma} of $\sigma$. Since $(i_l,b_l,n+l) \in P_{i_0,n}^{i_k,n+k}$ was arbitrary, the result follows.
\end{proof}

\subsection{Paths with regularly varying jump distribution}

One of the most important components of the deterministic argument in Section~\ref{sect: deterministic} is that paths cannot move very far without big jumps; this is the meaning of the event $\Cc_4$ defined in \eqref{eq: C4}. Corollary \ref{cor: path} is the main result of this section and will be used to bound from below the probability that the event $\Cc_4$ occurs.

As in~\cite{Berard2014}, we use Potter's bounds to give useful estimates on the regularly varying function $h$ (with index $\alpha$) defined in~\eqref{eq: poly_tail}. We will use the following elementary consequence of Potter's bounds.

\begin{lemma}\label{lemma: Potter}
For $\epsilon>0$, there exist $B(\epsilon)>1$ and $C_1(\epsilon), C_2(\epsilon)>0$ such that 
\begin{equation*}
\frac{1}{h(x)} \leq C_1 x^{\epsilon - \alpha} \quad \text{ and } \quad h(x) \leq C_2 x^{\alpha + \epsilon} \quad
\forall x \geq B.
\end{equation*}	
\end{lemma}

\begin{proof}
Let $\epsilon > 0$ be arbitrary. By Potter's bounds \cite[Theorem~1.5.6(iii)]{bingham_goldie_teugels_1987}, there exists $x_0>0$ depending only on $\epsilon$ such that
\begin{equation}\label{eq: Potter}
\frac{h(y)}{h(x)} \leq 2\max\left((y/x)^{\alpha + \epsilon}, (y/x)^{\alpha - \epsilon}\right)
\quad \forall x,y \ge x_0.
\end{equation}
Let $x \geq x_0$ be arbitrary and let $y = x_0$ in~\eqref{eq: Potter}. Then we have $y/x \le 1$ and so $(y/x)^{\alpha + \epsilon} \le (y/x)^{\alpha - \epsilon}$, and the first inequality in the statement of the lemma holds with $C_1 = 2 x_0^{\alpha - \epsilon}h(x_0)^{-1}$ and $B = x_0+1$. Similarly, since we have~$x/y \ge 1$, we have  $(x/y)^{\alpha - \epsilon} \le (x/y)^{\alpha + \epsilon}$, and hence by~\eqref{eq: Potter} (with $x$ and $y$ exchanged) the second inequality holds with $C_2 = 2 h(x_0) x_0^{-(\alpha + \epsilon)}$ and~$B=x_0+1$.
\end{proof}

 In order to show that~$\Cc_4$ occurs with high probability, we prove a lemma about a random walk with the same jump distribution as our $N$-BRW, but in which jumps larger than a certain size are discarded and count as a jump of size zero. The lemma gives an upper bound on the probability that this random walk moves a large distance $x_N$ in of order $\ell_N$ steps, if the jumps larger than $rx_N$ are discarded (for some $r\in (0,1)$). For an arbitrarily large $q>0$, the parameter $r$  can be taken sufficiently small that the above probability is smaller than $N^{-q}$ (for large $N$). Our lemma is similar to the lemma on page 168 of~\cite{Durrett1983MaximaOB}, where the jump distribution is truncated; jumps greater than a threshold value are not allowed at all, instead of being counted as zero. We use ideas from the proof of Theorem~3 in~\cite{gantert2000}, which is a large deviation result for sums of random variables with stretched exponential tails. 

Recall that $\prob(X>x)=h(x)^{-1}$ for $x\ge 0$, where $h$ is regularly varying with index $\alpha>0$.
\begin{lemma}\label{lemma: durrett}
Let $X_1,X_2,\dots$ be i.i.d.~random variables with $X_1 \stackrel{d}{=}X$.  For any $m\in\mathbb{N}$, $q>0$, $\lambda>0$, $0< r < 1\wedge\frac{\lambda(1\wedge\alpha)}{8q}$, for $N$ sufficiently large, if $x_N>N^\lambda$ then
\begin{equation*}
\prob\Bigg(\sum_{j=1}^{m\ell_N}X_j\mathds{1}_{\bracing{X_j\leq r x_N}} \ge x_N \Bigg) \leq N^{-q}.
\end{equation*}
\end{lemma}
Before proving Lemma~\ref{lemma: durrett}, we
 now state and prove an elementary identity which will be used in the proof. This identity was also used in the proof of Theorem~3 in \cite{gantert2000}.

\begin{lemma}\label{lemma: Gantert_identity}
	Suppose $Y$ is a non-negative random variable. For $v>0$ and $0<K_1<K_2<\infty$, 
	\begin{equation} \label{eq:Gantert}
	\expect[\exp(vY\mathds{1}_{\bracing{Y\leq K_2}})\mathds{1}_{\bracing{Y\geq K_1}}] = \int_{K_1}^{K_2} v e^{vu} \prob(Y>u)du + e^{vK_1}\prob(Y\geq K_1) - (e^{vK_2}-1)\prob(Y > K_2).
	\end{equation}
\end{lemma}
\begin{proof}
	First note that the random variable in the expectation on the left-hand side of~\eqref{eq:Gantert} takes the value $1$ if $Y>K_2$. The expectation can be written as
	\begin{align*}
	\expect[\exp(vY\mathds{1}_{\bracing{Y\leq K_2}})\mathds{1}_{\bracing{Y\geq K_1}}] 
	 =
	\expect\left[ e^{vY}\mathds{1}_{\bracing{K_1\leq Y \leq K_2}} \right] + \prob(Y>K_2). \numberthis \label{eq: Gantert_expect}
	\end{align*}	
	Now we will work on the integral on the right-hand side of~\eqref{eq:Gantert}. First, by Fubini's theorem we have
	\begin{align*}
	\int_{K_1}^{K_2} v e^{vu} \prob(Y>u)du
	 = \expect\left[\int_{K_1}^{K_2} v e^{vu} \mathds{1}_{\bracing{Y>u}}du\right] = \expect\left[\int_{K_1}^{K_2\wedge Y} v e^{vu}du \mathds{1}_{\bracing{Y\geq K_1}} \right].
	\end{align*}
	By calculating the integral, it follows that
	\begin{align*}
	\int_{K_1}^{K_2} v e^{vu} \prob(Y>u)du & = 
	\expect\left[\left( e^{v(K_2\wedge Y)} - e^{vK_1}\right)\mathds{1}_{\bracing{Y\geq K_1}} \right] \\ & 
	= \expect\left[ e^{vY}\mathds{1}_{\bracing{K_1\leq Y \leq K_2}} \right] + \expect\left[e^{vK_2}\mathds{1}_{\bracing{Y > K_2}} \right] - \expect\left[ e^{vK_1}\mathds{1}_{\bracing{Y\geq K_1}} \right].
	\end{align*}
	The result follows by \eqref{eq: Gantert_expect}.
\end{proof}
\begin{proof}[Proof of Lemma~\ref{lemma: durrett}]	
Let
$\tilde X := X\mathds{1}_{\bracing{X\leq r x_N}}$ and $\tilde X_j := X_j\mathds{1}_{\bracing{X_j\leq r x_N}}$ for all $j\in\N$.
Take $N$ sufficiently large that $\ell_N \le 2 \log_2 N$.
 Then by Markov's inequality and since $\tilde X_1, \tilde X_2, \ldots $ are i.i.d.~with $\tilde X_1 \stackrel{d}{=} \tilde X$, for $c>0$,
\begin{align*}
\prob\left(\sum_{j=1}^{m\ell_N}\tilde X_j \ge  x_N \right) 
& =
\prob\left(\exp\left(c \ell_N x_N^{-1}\sum_{j=1}^{m\ell_N}\tilde X_j\right) \ge e^{c\ell_N}\right) 
\\ &
\leq 
e^{-c\ell_N} \expect\left[ e^{c \ell_N x_N^{-1} \tilde X} \right]^{m\ell_N}
\\ &
\leq
N^{-\frac{c}{\log 2} + \frac{2m}{\log 2}\log \expect\left[e^{c \ell_N x_N^{-1} \tilde X} \right]}, \numberthis\label{eq: expMarkov}
\end{align*}
since $\log_2 N \leq \ell_N \leq 2\log_2 N$. We will show that with an appropriate choice of $c>0$, for $N$ sufficiently large, the right-hand side of \eqref{eq: expMarkov} is smaller than~$N^{-q}$. First we require
\begin{align}\label{eq: cr1}
c > 2q \log 2.
\end{align}
Second, we will have another condition on $c$ which ensures that 
$\expect [e^{c \ell_N x_N^{-1} \tilde X} ] \le 1+O(N^{-\epsilon})$ as $N \to \infty$ for some $\epsilon>0$.
We now estimate this expectation and determine the choice of $c$.

Take $0< \epsilon < \frac{\lambda(1\wedge\alpha)}{2(\lambda + 1)}$, and take $B=B(\epsilon)>1$ and $C_1=C_1(\epsilon)>0$ as in Lemma~\ref{lemma: Potter}. 
Suppose $N$ is sufficiently large that $rx_N>B$.
We apply Lemma~\ref{lemma: Gantert_identity} with $Y=X$, $v=c\ell_N x_N^{-1}$, $K_1 = B$ and~$K_2 = rx_N$, and then use~\eqref{eq: poly_tail},  to obtain
\begin{align*}
\expect\left[e^{c \ell_N x_N^{-1} \tilde X} \right] 
& \le \expect\left[e^{c \ell_N x_N^{-1}  X\mathds{1}_{\bracing{X\leq r x_N}}} \mathds{1}_{\bracing{X\ge B}}\right] + e^{Bc\ell_N x_N^{-1}}\prob(X < B)
\\ &
\leq 
\int_{B}^{rx_N} c\ell_N x_N^{-1} e^{c\ell_N x_N^{-1}u} h(u)^{-1} du + e^{Bc\ell_N x_N^{-1}}. \numberthis\label{eq: GantertIntegral}
\end{align*}

We will choose $c$ such that the first term on the right-hand side of \eqref{eq: GantertIntegral} is close to zero. 
By Lemma~\ref{lemma: Potter}, and then since $r<1$, we have
 \begin{align*}
\int_{B}^{rx_N} c\ell_N x_N^{-1} e^{c\ell_N x_N^{-1}u} h(u)^{-1} du & \leq
\int_{B}^{rx_N}C_1c\ell_N x_N^{-1} e^{c\ell_N x_N^{-1}u} u^{-\alpha + \epsilon}du 
\\ &
\leq C_1 c \ell_N x_N^{-1} \int_{B}^{rx_N}e^{c\ell_N x_N^{-1}(rx_N)}x_N^\epsilon u^{-\alpha}du .
\end{align*}
Integrating the right-hand side, since we took $N$ sufficiently large that $\ell_N \le 2 \log_2 N$, we conclude
\begin{align}\label{eq: GantertInt1}
\int_{B}^{rx_N} c\ell_N x_N^{-1} e^{c\ell_N x_N^{-1}u} h(u)^{-1} du \leq 
\left\{\begin{array}{ll}
\frac{C_1 c}{1-\alpha}\ell_N N^{\frac{2cr}{\log 2}} \left(r^{1-\alpha} x_N^{\epsilon - \alpha} - B^{1-\alpha} x_N^{\epsilon - 1}\right), & \text{ if }\alpha\neq 1,
\vspace{10pt}\\
C_1 c \ell_N x_N^{\epsilon - 1} N^{\frac{2cr}{\log 2}}\log x_N, & \text{ if }\alpha=1,
\end{array}
\right.
\end{align}
where in the $\alpha = 1$ case we use that $B > 1$ and that $r<1$.

Now, since $x_N > N^\lambda$ and $\epsilon<1\wedge \alpha$, the right-hand side of \eqref{eq: GantertInt1} is at most of order $N^{-\epsilon}$ if
\begin{align}\label{eq: cr2}
\frac{2cr}{\log 2} + \lambda(\epsilon - (1\wedge\alpha)) < -\epsilon.
\end{align}
Since $r < \frac{\lambda(1\wedge\alpha)}{8q}$ by the assumptions of the lemma,  we can find $c$ such that
\begin{align*}
2q\log 2 < c < \frac{\lambda(1\wedge\alpha)\log 2}{4r}.
\end{align*}
Then since we chose $\epsilon < \frac{\lambda(1\wedge\alpha)}{2(\lambda + 1)}$,
$c$ satisfies~\eqref{eq: cr1} and \eqref{eq: cr2}. Note furthermore that since $x_N>N^\lambda$, the second term on the right-hand side of \eqref{eq: GantertIntegral} is close to $1$ for $N$ large; for $N$ sufficiently large we have
\begin{equation}\label{eq: GantertInt2}
e^{Bc\ell_N x_N^{-1}} \leq e^{Bc\ell_N N^{-\lambda}} \leq 1 + 2Bc\ell_N N^{-\lambda}.
\end{equation}

Hence, \eqref{eq: GantertIntegral}, \eqref{eq: GantertInt1} and the choice of $c$, and \eqref{eq: GantertInt2} with the fact that $\epsilon<\lambda$ show that there exists a constant $A>0$ such that
\begin{equation*}
\expect\left[e^{c \ell_N x_N^{-1} \tilde X} \right]  \leq 1 + A N^{-\epsilon}
\end{equation*}
for $N$ sufficiently large and $x_N>N^\lambda$. Therefore, by \eqref{eq: expMarkov} and \eqref{eq: cr1} we have
\begin{align*}
\prob\left(\sum_{j=1}^{m\ell_N}\tilde X_j \ge  x_N \right) \leq N^{-2q + \frac{2m}{\log 2}\log(1 + A N^{-\epsilon})} \leq N^{-2q + \frac{2m}{\log 2}A N^{-\epsilon}} < N^{-q},
\end{align*}
for $N$ sufficiently large, which concludes the proof.
\end{proof}
We now apply Lemma~\ref{lemma: durrett} to the $N$-BRW, to give us a convenient form of the result which we will use later in this section and also in Section~\ref{sect: star-shaped}.
\begin{corollary}\label{cor: path}
Let $\lambda > 0$ and $0<r<1\wedge \frac{\lambda(1\wedge\alpha)}{48}$. Then there exists~${C>0}$ such that for $N$ sufficiently large, if $x_N > N^\lambda$,
\begin{align*}
&\prob\left(
\begin{array}{l}
\exists (k_1,s_1)\in [N]\times \llbracket t_4,t-1 \rrbracket ,\;s_2\in \llbracket s_1+1,t \rrbracket \text{ and } k_2\in\Nn_{k_1,s_1}(s_2):\: \\
\sum_{(i,b,s)\in P_{k_1,s_1}^{k_2,s_2}} X_{i,b,s}\mathds{1}_{\bracing{X_{i,b,s}\leq r x_N}} \ge x_N
\end{array}
\right) 
\leq CN^{-1},
\end{align*}
where $P_{k_1,s_1}^{k_2,s_2}$ and $\Nn_{k_1,s_1}(s_2)$ are defined in \eqref{eq: P} and \eqref{eq: N} respectively.
\end{corollary}
\begin{proof}
Take $(k_1,s_1),(k_2,s_2) \in [N]\times \llbracket t_4,t-1 \rrbracket	$ with
$(k_1,s_1)\lesssim(k_2,s_2)$, and let $k' = \zeta_{k_1,s_1}(t_4)$ be the index of the time-$t_4$ ancestor of $(k_1,s_1)$ (see \eqref{eq: zeta} for the notation). 
If the path between particles $(k_1,s_1)$ and $(k_2,s_2)$ moves at least $x_N$ even with discarding jumps greater than $r x_N$, then the path between $(k',t_4)$ and $(k_2,s_2)$ does the same, because all jumps are non-negative. Therefore we only need to consider paths starting with the $N$ particles of the population at time $t_4$:
\begin{multline*}
\prob\left(
\begin{array}{l}
\exists (k_1,s_1)\in [N]\times \llbracket t_4,t-1 \rrbracket ,\;s_2\in \llbracket s_1+1,t \rrbracket \text{ and } k_2\in\Nn_{k_1,s_1}(s_2):\: \\
\sum_{(i,b,s)\in P_{k_1,s_1}^{k_2,s_2}} X_{i,b,s}\mathds{1}_{\bracing{X_{i,b,s}\leq r x_N}} \ge x_N
\end{array}
\right) \\
\leq
\prob\left(
\begin{array}{l}
\exists k'\in[N],\;s_2\in \llbracket t_4+1,t \rrbracket \text{ and } k_2\in\Nn_{k',t_4}(s_2):\: \\
\sum_{(i,b,s)\in P_{k',t_4}^{k_2,s_2}} X_{i,b,s}\mathds{1}_{\bracing{X_{i,b,s}\leq r x_N}} \ge x_N
\end{array}
\right).\numberthis\label{eq: path1}
\end{multline*}
Now consider the $N$-BRW constructed from $N$ independent BRWs (see Section~\ref{sect: BRW}). Assume that $k'\in[N]$, $s_2\in  \llbracket t_4+1,t \rrbracket$ and $k_2\in\Nn_{k',t_4}(s_2)$ are such that
\begin{align*}
\sum_{(i,b,s)\in P_{k',t_4}^{k_2,s_2}} X_{i,b,s}\mathds{1}_{\bracing{X_{i,b,s}\leq r x_N}} \ge x_N.
\end{align*}
Then by Lemma~\ref{lemma: BRWpath} there exists a path in one of the $N$ independent BRWs that contains the same jumps as the path $P_{k',t_4}^{k_2,s_2}$. Thus Lemma~\ref{lemma: BRWpath} implies that there exist $(j,u)\in H_{t_4}$ and $(j,v)\in H_{s_2}$ such that $u\prec v$ and
\begin{equation*}
\sum_{u \prec w \preceq v} Y_{j,w} \1_{\bracing{Y_{j,w}\leq rx_N}} \ge x_N.
\end{equation*}
That is, there is a path in the $N$ independent BRWs between times $t_4$ and $s_2$ which moves at least $x_N$ even with discarding jumps of size greater than $rx_N$. This means that there must be a path with the same property between times $t_4$ and $t$ as well, because all jumps are non-negative. Therefore
\begin{align*}
& \prob\left(
\begin{array}{l}
\exists k'\in[N],\;s_2\in \llbracket t_4+1,t \rrbracket \text{ and } k_2\in\Nn_{k',t_4}(s_2):\: \\
\sum_{(i,b,s)\in P_{k',t_4}^{k_2,s_2}} X_{i,b,s}\mathds{1}_{\bracing{X_{i,b,s}\leq r x_N}} \ge x_N
\end{array}
\right) \\ &
\leq 
\prob \bigg(
\exists (j,u)\in H_{t_4} \text{ and } v\in\bracing{1,2}^t \text{ with } v \succ u:\: 
\sum_{u \prec w \preceq v} Y_{j,w} \1_{\bracing{Y_{j,w}\leq rx_N}} \ge x_N
\bigg).\numberthis\label{eq: path2}
\end{align*}

Let~$X_i$,~$i=1,2,\dots$ be i.i.d.~with distribution given by \eqref{eq: poly_tail}, and take $\lambda > 0$, $x_N>N^\lambda$, and $0<r < 1\wedge \frac{\lambda(1\wedge\alpha)}{48}$. Note that the random variables $Y_{j,w}$ are all distributed as the $X_i$ random variables, and that there are $4\ell_N$ terms in the sum on the right-hand side of \eqref{eq: path2}. We will give a union bound for the probability of the event on the right-hand side of \eqref{eq: path2}, using that $H_{t_4}$ is a set of $N$ elements and that a particle in the set $H_{t_4}$ has $2^{4\ell_N}$ descendants in a BRW (without selection) at time $t$, which means $2^{4\ell_N}$ possible labels for $v$ for each $(j,u)\in H_{t_4}$. Then by \eqref{eq: path1}, \eqref{eq: path2} and by conditioning on $\mathcal F'_{t_4}$ and using a union bound,
\begin{multline*}
\prob\left(
\begin{array}{l}
\exists (k_1,s_1)\in [N]\times \llbracket t_4,t-1 \rrbracket ,\;s_2\in  \llbracket s_1+1,t \rrbracket \text{ and } k_2\in\Nn_{k_1,s_1}(s_2):\: \\
\sum_{(i,b,s)\in P_{k_1,s_1}^{k_2,s_2}} X_{i,b,s}\mathds{1}_{\bracing{X_{i,b,s}\leq r x_N}} \ge x_N
\end{array}
\right) \\
\leq
N2^{4\ell_N}\prob\Bigg(\sum_{j = 1}^{4\ell_N} X_j\mathds{1}_{\bracing{X_j\leq r  x_N}} \ge  x_N\Bigg).\numberthis\label{eq: existsPath}
\end{multline*}
Then by Lemma~\ref{lemma: durrett} with~${m=4}$ and ${q=6}$, we have that for $N$ sufficiently large,
	\begin{align*}
	\prob\Bigg(\sum_{j = 1}^{4\ell_N} X_j\mathds{1}_{\bracing{X_j\leq r  x_N}} \ge  x_N\Bigg) \leq N^{-6}.
	\end{align*}
	The result follows by \eqref{eq: existsPath}.
\end{proof}

\subsection{Simple properties of the regularly varying function $h$}\label{subsec: h_properties}
In order to bound the probabilities of the events $\Cc_2$ to $\Cc_7$ and $\Dd_1$ to $\Dd_5$, we will need to use several properties of the function $h$ from \eqref{eq: poly_tail}. Recall that $h$ is regularly varying with index $\alpha>0$, and that it determines the jump distribution of the $N$-BRW in the sense that for each jump $(i,b,s)$,
\begin{equation} \label{eq:Xibstailh}
\prob (X_{i,b,s}>x )=h(x)^{-1} \quad \forall x \ge 0.
\end{equation}
Recall that $a_N=h^{-1}(2N \ell_N)$, and note that $a_N \to \infty$ as $N \to \infty$. 
Indeed, by the definition of $h^{-1}$ in~\eqref{eq: hinv}, $a_N$ is non-decreasing, and since $h$ is non-decreasing by~\eqref{eq: poly_tail}, $a_N$ cannot converge to a finite limit $a\in \R$, because this would imply $h(a+1)\ge 2N\ell_N$ $\forall N$. Moreover, letting $C_2=C_2(\alpha)$ as in Lemma~\ref{lemma: Potter}, for $N$ sufficiently large that $a_N+1 \ge B=B(\alpha)$,
\begin{equation}\label{eq: aNhaN}
2N\ell_N < h(a_N + 1) \leq C_2(a_N+1)^{2 \alpha},
\end{equation}
where in the first inequality we use the definition \eqref{eq: hinv} of $h^{-1}$ and that $h$ is non-decreasing, and the second inequality follows by the second inequality of Lemma~\ref{lemma: Potter}.

Since $h$ is regularly varying with index $\alpha$, we have
\begin{equation}\label{eq: haNasymp}
\frac{2N\ell_N}{h(a_N)} \to 1 \quad \text{ as } N\to\infty.
\end{equation}
Indeed, since $h$ is non-decreasing, for any $\eps \in (0,1)$, by \eqref{eq: regvar} and by the definition of $a_N$ we have
\begin{equation*}
(1-\eps)^\alpha - \eps \leq  \frac{h(a_N(1-\eps))}{h(a_N)} \leq \frac{2N\ell_N}{h(a_N)} \leq \frac{h(a_N(1+\eps))}{h(a_N)} \leq (1+\eps)^\alpha + \eps,
\end{equation*}
for $N$ sufficiently large. Often in our proofs it will be enough to use that \eqref{eq: haNasymp} implies
\begin{equation}\label{eq: haN}
\frac{1}{2} < \frac{2N\ell_N}{h(a_N)} < 2,
\end{equation}
for $N$ sufficiently large.

For convenience we state a few other simple properties of $h$, which we will apply several times. Let $r\in (0,1)$ and $\eta<1/1000$. First, we have
\begin{equation}\label{eq: hr}
\frac{1}{h(ra_N)} < \frac{1}{h(a_N)}(r^{-\alpha} + \eta^4) < \frac{1}{h(a_N)}2r^{-\alpha},
\end{equation}
for $N$ sufficiently large, by \eqref{eq: regvar}. Second, for $N$ sufficiently large, we also have
\begin{equation}\label{eq: hr2Nl}
\frac{2N\ell_N}{h(ra_N)} < \frac{2N\ell_N}{h(a_N)}(r^{-\alpha} + \eta^4) < (1 + \eta^4)(r^{-\alpha} + \eta^4) < 2r^{-\alpha},
\end{equation}
by \eqref{eq: hr} and \eqref{eq: haNasymp}. Furthermore, by the same argument as for~\eqref{eq: hr2Nl}, for $N$ sufficiently large,
\begin{equation}\label{eq: hr12Nl}
\frac{2N\ell_N}{h(ra_N)} > \frac{r^{-\alpha}}{2}.
\end{equation}

\subsection{Probabilities and proof of Proposition~\ref{prop: A1A3}}\label{sect: probCD}

Next we will go through the events $(\Cc_j)_{j=2}^7$ and $(\Dd_i)_{i=1}^5$, which we defined in Section~\ref{sect: deterministic}, one by one. We will prove upper bounds on the probabilities of their complement events, which will then allow us to prove Proposition~\ref{prop: A1A3}. Recall that the events $(\Cc_j)_{j=2}^7$ and $(\Dd_i)_{i=1}^5$ all depend on the constants $\eta, K, \gamma, \delta, \rho, c_1 \ldots, c_6$ introduced in~\eqref{eq: const1}-\eqref{eq: const5}, and  Propositions~\ref{prop: B} and~\ref{prop: C} hold when the constants satisfy the conditions~\eqref{eq: const1}-\eqref{eq: const5}. In order to show that the events in question occur with high probability, the constants need to satisfy some extra conditions which are consistent with~\eqref{eq: const1}-\eqref{eq: const5}. We now specify these choices.

Recall that $\alpha >0$. First we assume that $\eta\in(0,1]$ is very small; in particular, that it is small enough to satisfy
\begin{equation}\label{eq: eta}
\eta^2 < \min\left(\left(2^{\alpha + 2}\log\left(\frac{1000}{\eta}\right)\right)^{-1/ \alpha}, \frac{\eta}{1000\cdot 2^{\alpha}}\right).
\end{equation}
Then we choose the remaining constants as follows:
\begin{enumerate}[label=(\alph*), ref=\alph*]
	\item\label{c5} $c_6:=\eta^2 $,
	\item\label{c4} $c_5:=\eta^{6(1\vee\alpha)}$,
	\item\label{c3} $c_4 := c_5^{4/(1\wedge\alpha)}$,
	\item\label{c2} take $c_3>0$ small enough to satisfy $c_3 < c_4^{4(1\vee\alpha)}$ and $\left(1- 6c_3/c_4\right)^\alpha \geq 1 - 12\alpha c_3/c_4$,
	\item\label{c1} take $c_2>0$ small enough to satisfy $c_2 < c_3^{4(1\vee\alpha)}$ and $\left(1- 4c_2/c_3\right)^\alpha \geq 1 - 8\alpha c_2/c_3$,
	\item\label{sigtilde} $ c_1  := c_2^2$,
	\item\label{sigma} $\rho :=  c_1 (1\wedge\alpha)^2 / (100\alpha)$,
	\item\label{delta} $\delta := \rho^{\alpha + 1}$ ,
	\item\label{gamma} $\gamma := \delta / 2$,
	\item\label{K} $K := \rho^{-\alpha - 1}$.
\end{enumerate}
Note that the constants with the choices above can be thought of as in \eqref{eq: const}. We state a few simple consequences of these choices, which will be useful in proving upper bounds on the probabilities of the complement events of $\Cc_2$ to $\Cc_7$ and $\Dd_1$ to $\Dd_5$. First, by \eqref{eq: eta}, we have
\begin{equation}\label{eq: eta1000}
\eta < \frac{1}{1000\cdot 2^\alpha} < \frac{1}{1000},
\end{equation}
and note that all constants $\gamma, \delta, \rho, c_1 \ldots, c_6$ and $1/K$ are at most $\eta^2$. Thus, from \eqref{c5}-\eqref{sigtilde} and \eqref{eq: eta1000}, for $j=1,\dots,5$, we have
\begin{equation}\label{eq: cj}
c_j \leq c_{j+1}^2 \leq c_{j+1}\eta^2 < \frac{c_{j+1}}{10^6\cdot 2^{2\alpha}},
\end{equation}
which also means 
\begin{equation}\label{eq: cjeta}
c_j < \frac{\eta^2}{10^6}
\end{equation}
for $j=1,\dots , 5$. In particular we will need that
\begin{equation}\label{eq: c2c3}
\frac{c_2}{c_3} < \frac{1}{10^6(1\vee\alpha)}
\end{equation}
and
\begin{equation}\label{eq: c3c4}
\frac{c_3}{c_4} < \frac{1}{10^6(1\vee\alpha)},
\end{equation} 
which both follow by \eqref{eq: cj} and by the fact that $2^{2\alpha}\geq e^\alpha \ge 1\vee\alpha$ for $\alpha>0$. We will also use that from \eqref{c1} we have
\begin{equation}\label{eq: c3alpha}
c_3^{-\alpha-1}c_2 < c_3^{-2(1\vee\alpha) + 4(1\vee\alpha)} \leq c_3^2 < \frac{c_4}{10^6\cdot 2^{2\alpha}}\frac{c_4}{10^6\alpha} < \frac{\eta^4}{16\alpha 2^\alpha},
\end{equation}
where we applied \eqref{eq: cj} and that $2^{2\alpha}\geq\alpha$, and then that $c_4<\eta^2$. Then similarly, from \eqref{c2} we have
\begin{equation}\label{eq: c4alpha}
c_4^{-\alpha-1}c_3 < \frac{\eta^4}{24\alpha 2^\alpha}.
\end{equation}
Finally, from \eqref{sigma} and \eqref{eq: cjeta} we have
\begin{equation}\label{eq: rho}
\rho < c_1 < \frac{\eta}{10^6}.
\end{equation}

Considering the choices \eqref{c5}-\eqref{K} together with the consequences \eqref{eq: eta1000} and \eqref{eq: cj}, and noticing that \eqref{sigma} implies $\rho \leq c_1 /100$, we conclude that the constants $\eta,K, \gamma, \delta, \rho, c_1 \ldots, c_6$ satisfy \eqref{eq: const1}-\eqref{eq: const5}, so we will be able to apply Propositions~\ref{prop: B} and~\ref{prop: C} with this choice of constants. 

We can now show that the events $\Cc_2$ to $\Cc_7$ and $\Dd_1$ to $\Dd_5$ occur with high probability.
\begin{lemma}\label{lemma: probC}
Suppose the constants $\eta$, $K$, $\gamma$, $\delta$, $\rho$, $c_1,\dots,c_6>0$ satisfy \eqref{eq: eta} and \eqref{c5}-\eqref{K}. Then for $N$ sufficiently large and $t>4\ell_N$,
\begin{align*}
\prob(\Cc_j^c) < \frac{\eta}{1000} \quad \text{and} \quad \prob(\Dd_i^c) < \frac{\eta}{1000}
\end{align*}
for all $j\in\bracing{2,\dots,7}$ and $i\in\bracing{1,\dots,5}$,
where the events $(\Cc_j)_{j=2}^7$ and $(\Dd_i)_{i=1}^5$ are defined in~\eqref{eq: C2}-\eqref{eq: C7} and \eqref{eq: D1}-\eqref{eq: D5} respectively.
\end{lemma}

\begin{proof}
Assume that $\eta>0$ satisfies \eqref{eq: eta}. We consider the events $(\Cc_j)_{j=2}^7$ and $(\Dd_i)_{i=1}^5$ with the constants $K$, $\gamma$, $\delta$, $\rho$, $c_1,\dots,c_6$, and we assume that these constants satisfy \eqref{c5}-\eqref{K}. We will upper bound the probabilities of the events $(\Cc_j^c)_{j=2}^7$ and $(\Dd_i^c)_{i=1}^5$ using \eqref{eq: cj}-\eqref{eq: rho} above, and the properties of the regularly varying function $h$ described in Section \ref{subsec: h_properties}.\\

\noindent
\textbf{The event $\Cc_2^c$} (see \eqref{eq: C2}) says that there is a time $s\in[t_3,t-1]$ when a particle at distance at least~$c_3a_N$ behind the leader  jumps to within distance $2c_2a_N$ of the leader's position. We use Markov's inequality, and sum over all the jumps happening between times $t_3$ and $t-1$ to  bound the probability of this event. We have
\begin{align*}
\prob(\Cc_2^c) & \leq \expect \left[\#\bracing{ \begin{array}{l}
(i,b,s)\in\Nonetwo\times \llbracket t_3,t-1 \rrbracket \text{ such that }\\
Z_i(s)\geq c_3 a_N \text{ and } X_{i,b,s}\in (Z_i(s) - 2c_2a_N, Z_i(s) + 2c_2a_N]
\end{array}}\right] \\ &
= \sum_{(i,b,s)\in\Nonetwo\times \llbracket t_3,t-1 \rrbracket } \expect \left[ \mathds{1}_{\{Z_i(s)\geq c_3 a_N\}} \mathds{1}_{\bracing{X_{i,b,s}\in (Z_i(s) - 2c_2a_N, Z_i(s) + 2c_2a_N]}}  
\right].
\end{align*}
Recall from Section~\ref{sect: notation} that for $s\in \N$ and $i\in [N]$,
the distance $Z_i(s)$ of the $i$th particle from the leader is $\F_s$-measurable, but the jumps performed at time $s$, $X_{i,1,s}$ and $X_{i,2,s}$, are independent of $\F_s$. Hence by~\eqref{eq:Xibstailh},
\begin{align*} 
\prob(\Cc_2^c) 
& \le \sum_{(i,b,s)\in\Nonetwo\times \llbracket t_3,t-1 \rrbracket }\expect \left[ \expect \left[ \left. \mathds{1}_{\{Z_i(s)\geq c_3 a_N\}} \mathds{1}_{\bracing{X_{i,b,s}\in (Z_i(s) - 2c_2a_N, Z_i(s) + 2c_2a_N]}}    \right| \mathcal{F}_s
\right]
\right] \\ &
= \sum_{(i,b,s)\in\Nonetwo\times \llbracket t_3,t-1 \rrbracket }\expect \left[ \mathds{1}_{\{Z_i(s) \geq c_3a_N\}} \left(h(Z_i(s) - 2c_2a_N)^{-1} - h(Z_i(s) + 2c_2a_N)^{-1}\right) 
\right].\numberthis\label{eq: Eh}
\end{align*}
Since $h$ is monotone non-decreasing, for any $z \geq c_3 a_N$ we have
\begin{align}\label{eq: PC2a}
h(z - 2c_2a_N)^{-1} - h(z + 2c_2a_N)^{-1} \leq h((c_3 - 2c_2)a_N)^{-1}\left(1 - \frac{h(z-2c_2a_N)}{h(z+2c_2a_N)}\right).
\end{align}
Take $\epsilon >0$.
For the fraction on the right-hand side of~\eqref{eq: PC2a} we have that for $N$ sufficiently large, for $z\ge c_3 a_N$,
\begin{align} \label{eq:hfracs}
1 \geq \frac{h(z-2c_2a_N)}{h(z+2c_2a_N)} \geq 
\frac{h\left((z+2c_2a_N)\cdot\frac{(c_3-2c_2)a_N}{(c_3+2c_2)a_N}\right)}{h(z+2c_2a_N)}
\geq \left(1- \frac{4c_2}{c_3+2c_2}\right)^\alpha - \epsilon \geq 1 - 8\alpha\frac{c_2}{c_3} - \epsilon,
\end{align}
where we first use the monotonicity of $h$, and in the second inequality we use that $z\geq c_3a_N$, that the function $y\mapsto (y-2c_2a_N)/(y+2c_2a_N)$ is increasing in $y$, and we again use the monotonicity of~$h$. The third inequality follows by \eqref{eq: regvar}, and the fourth holds by the definition of $c_2$ in \eqref{c1}. Then, by~\eqref{eq: hr} and the lower bound in~\eqref{eq:hfracs} with $\epsilon = \eta^4(c_3-2c_2)^\alpha$,  we see from \eqref{eq: PC2a} that for $N$ sufficiently large, for $z\ge c_3 a_N$,
\begin{align*}
h(z - 2c_2a_N)^{-1} - h(z + 2c_2a_N)^{-1} &\leq\; 2(c_3-2c_2)^{-\alpha}h(a_N)^{-1}\left(8\alpha\frac{c_2}{c_3}+\eta^4(c_3-2c_2)^\alpha\right)
\\ & 
\leq\;
h(a_N)^{-1}(16\alpha2^\alpha c_3^{-\alpha-1}c_2 + 2\eta^4)\numberthis\label{eq: hzc1},
\end{align*}
where for the first term of the second inequality we used the fact that $(c_3-2c_2)^{-\alpha} < (c_3/2)^{-\alpha}$, because $2c_2<c_3/2$ by \eqref{eq: cj}.

Now let us return to \eqref{eq: Eh} and notice that we sum over $6N\ell_N$ jumps. Therefore, by~\eqref{eq: hzc1} we conclude that for $N$ sufficiently large,
\begin{align*}
\prob(\Cc_2^c) & \leq \frac{6N\ell_N}{h(a_N)}(16\alpha2^\alpha c_3^{-\alpha-1}c_2 + 2\eta^4) 
\leq 6(16\alpha2^\alpha c_3^{-\alpha-1}c_2 + 2\eta^4) 
< 18\eta^4
<\frac{\eta}{1000},
\end{align*}
where we used \eqref{eq: haN} in the second inequality, \eqref{eq: c3alpha} in the third, and~\eqref{eq: eta1000} in the fourth.\\

\noindent
\textbf{The event $\Cc_3^c$} (see \eqref{eq: C3}) says that there exists a big jump in the time interval $[t_4,t-1]$ such that a descendant also performs a big jump during the time interval $[t_4,t-1]$, within time $\ell_N$ of the first big jump. 

Consider the $N$-BRW constructed from $N$ independent BRWs (see Section~\ref{sect: BRW}). 
If $\Cc_3^c$ occurs then there must be
two big jumps in the $N$-BRW as above; that is, we must have $(i_1,s_1)\lesssim_{b_1} (i_2,s_2)$ with $s_1\in\bbracket{t_4,t-2}$ and $s_2\in\bbracket{s_1+1,\min\bracing{s_1+\ell_N,t-1}}$,
and $(i_1,b_1,s_1),(i_2,b_2,s_2)\in B_N$, where $B_N$ is the set of big jumps defined in \eqref{eq: BN2}. Then by Lemma~\ref{lemma: BRWpath} there are two big jumps with the same properties in the $N$ independent BRWs; that is,
there exist $j\in[N]$, $u_1,u_2\in\U_0$ such that $(j,u_1)\in H_{s_1}$, $(j,u_2)\in H_{s_2}$, $u_1 b_1\preceq u_2$, $X_{i_1,b_1,s_1} = Y_{j,u_1b_1}$ and $X_{i_2,b_2,s_2} = Y_{j,u_2b_2}$. Therefore, since  $s_2\in\bbracket{s_1+1,\min\bracing{s_1+\ell_N,t-1}}\subseteq \bbracket{s_1+1,{s_1+\ell_N}}$ and $H_{s_2}\subseteq [N] \times \bracing{1,2}^{s_2}$, we have
\begin{align}\label{eq: C3BRW}
\prob(\Cc_3^c) \leq \prob\left(
\begin{array}{l}
\exists s_1\in\bbracket{t_4,t-2}, (j,u_1)\in H_{s_1}, b_1\in\bracing{1,2} \text{ and } \\
s_2\in \bbracket{s_1+1,{s_1+\ell_N}}, u_2\in \bracing{1,2}^{s_2}, u_2\succ u_1, b_2\in\bracing{1,2}: \\
Y_{j,u_1b_1} > \rho a_N \text{ and } Y_{j,u_2b_2} > \rho a_N
\end{array}
\right).
\end{align}
Recall the definition of $\mathcal F'_n$ in Section~\ref{sect: BRW}.
By a union bound over the possible values of $s_1$, $s_2$, $b_1$ and $b_2$, and then conditioning on $\mathcal F'_{s_1}$ and applying another union bound over the possible values of $(j,u_1)$ and $u_2$,
\begin{align*}
\prob(\Cc_3^c) \leq
\sum_{\substack{s_1\in\bbracket{t_4,t-2},s_2\in \bbracket{s_1+1,{s_1+\ell_N}},\\ b_1,b_2\in \{1,2\}}}
\expect \left[\sum_{(j,u_1)\in H_{s_1}, u_2\in \bracing{1,2}^{s_2}, u_2\succ u_1}\prob\left(
Y_{j,u_1b_1} > \rho a_N, Y_{j,u_2b_2} > \rho a_N \big| \mathcal F'_{s_1} \right) \right].
\end{align*}
Then since $(Y_{j,u})_{j\in [N], u\in \cup_{m>s_1}\{1,2\}^m}$ are independent of $\mathcal F'_{s_1}$, for $(j,u_1)\in H_{s_1}$ and $u_2\in \bracing{1,2}^{s_2}$ we have
$$
\prob\left(
Y_{j,u_1b_1} > \rho a_N, Y_{j,u_2b_2} > \rho a_N \big| \mathcal F'_{s_1} \right)=h(\rho a_N)^{-2}.
$$
Hence by summing over the $4\ell_N-1$ possible values for $s_1$, and the two possible values for $b_1$ and $b_2$,
and since $|H_{s_1}|=N$, and for $u_1\in \{1,2\}^{s_1}$ there are $2^{s_2-s_1}$ possible values of $u_2\in \{1,2\}^{s_2}$ with $u_2 \succ u_1$, for $N$ sufficiently large we have
\begin{align*}
\prob(\Cc_3^c) \leq 
4\ell_N \cdot 4 \sum_{s_2\in \bbracket{s_1+1,{s_1+\ell_N}}}N 2^{s_2-s_1} h(\rho a_N)^{-2} 
&\leq 16 N\ell_N\cdot 2\cdot 2^{\log_2 N +1}h(\rho a_N)^{-2}
\\ &=\left(\frac{2N\ell_N }{h(\rho a_N)}\right)^2 16 \ell_N^{-1}
 \leq 4\rho^{-2\alpha}\cdot  16 \ell_N^{-1}
<\: \frac{\eta}{1000},\numberthis\label{eq: PC3}
\end{align*}
where in the third inequality we used \eqref{eq: hr2Nl}.\\

\noindent
\textbf{The event $\Cc_4^c$} (see \eqref{eq: C4}) can be bounded using Corollary~\ref{cor: path}. We apply the corollary with ${x_N =  c_1  a_N}$, $r = \rho /  c_1 $ and $\lambda = 1 / (2\alpha)$. We can make this choice for $\lambda$, because we have
\begin{align}\label{eq: aN2alpha}
 c_1  a_N > N^{1/(2\alpha)}
\end{align}
for all $N$ sufficiently large by \eqref{eq: aNhaN}. By our choice of $\rho$ in \eqref{sigma}, we have $r<1\wedge \frac{\lambda (1\wedge\alpha)}{48}$, and so Corollary~\ref{cor: path} tells us that for some constant $C>0$, for $N$ sufficiently large,
\begin{align}\label{eq: PC4}
\prob(\Cc_4^c) \leq CN^{-1} < \frac{\eta}{1000}.
\end{align}

\noindent
\textbf{The event $\Cc_5^c$} (see \eqref{eq: C5}) says that two big jumps occur at the same time, that is
$$
\Cc_5^c = \left\{\exists s\in \llbracket t_4,t-1 \rrbracket,\;(k_1,b_1) \neq (k_2,b_2)\in\Nonetwo:\: X_{k_1,b_1,s} > \rho a_N \text{ and }X_{k_2,b_2,s} > \rho a_N
\right\}.
$$
By a union bound over the $4\ell_N$ time steps and the possible pairs of jumps at each time step,
\begin{align*}
\prob(\Cc_5^c) 
& \leq  4 \ell_N{2N \choose 2} h(\rho a_N)^{-2}
\leq \left(\frac{2N\ell_N}{h(\rho a_N)}\right)^2 2\ell_N^{-1}
\leq 4\rho^{-2\alpha}\cdot 2 \ell_N^{-1}
< \frac{\eta}{1000}
\numberthis\label{eq: PC5}
\end{align*}
for $N$ sufficiently large, where the third inequality follows by~\eqref{eq: hr2Nl}.\\

\noindent
\textbf{The event $\Cc_6^c$} (see \eqref{eq: C6}) says that a big jump happens in (at least) one of two very short time intervals, $[t_2,t_2+\ceil{\delta\ell_N}]$ and  $[t_1 - \ceil{\delta\ell_N} ,t_1+\ceil{\delta\ell_N}]$. In total there are  $2N\cdot(3\ceil{\delta\ell_N} + 2)$ jumps performed during these two time intervals. By a union bound over these jumps, we get
\begin{align*}
\prob(\Cc_6^c) & = \prob(\exists (i,b,s)\in\Nonetwo \times (\llbracket t_2,t_2+\ceil{\delta\ell_N} \rrbracket \cup \llbracket t_1-\ceil{\delta \ell_N}, t_1 + \ceil{\delta \ell_N} \rrbracket ):\: X_{i,b,s} > \rho a_N) \\ &
\leq 2N(3\delta \ell_N +5) h(\rho a_N)^{-1} 
\leq 6\delta\rho^{-\alpha} (1+2\delta^{-1} \ell_N^{-1}) 
< \frac{\eta}{1000},\numberthis\label{eq: PC6}
\end{align*}
for $N$ sufficiently large,
where in the second inequality we used  \eqref{eq: hr2Nl}, and the last inequality follows by the choice of $\delta$ in \eqref{delta} and by \eqref{eq: rho}.\\

\noindent
\textbf{The event $\Cc_7$} gives an upper bound on the number of big jumps (see \eqref{eq: C7}). There are $8N\ell_N$ jumps performed in the time interval $[t_4,t-1]$; by Markov's inequality and then by~\eqref{eq: hr2Nl}, we have
\begin{align*}
\prob(\Cc_7^c) & = \prob(\#\bracing{(i,b,s)\in \Nonetwo \times \llbracket t_4, t-1 \rrbracket : X_{i,b,s}> \rho a_N} > K) \\ &
\leq \frac{8N \ell_N h(\rho a_N)^{-1}}{K} 
\leq \frac{8}{K}\rho^{-\alpha} 
< \frac{\eta}{1000}\numberthis\label{eq: PC7}
\end{align*}
for $N$ sufficiently large,
where the last inequality follows by the choice of $K$ in \eqref{K} and by \eqref{eq: rho}.
\\

\noindent
\textbf{The event $\Dd_1$} (see \eqref{eq: D1}) has the same definition as that of $\Cc_2$ (see \eqref{eq: C2}), except with different constants. By the same argument as for \eqref{eq: hzc1}, using the definition of $c_3$ in~\eqref{c2}, for $N$ sufficiently large we have
\begin{align}\label{eq: hzc2}
h(z-3c_3a_N)^{-1} - h(z+3c_3a_N)^{-1} \leq h(a_N)^{-1}(24\alpha\cdot 2^\alpha c_4^{-\alpha-1}c_3 + 2\eta^4) \quad \forall z \geq c_4a_N.
\end{align}
Then continuing in the same way as after \eqref{eq: hzc1} we obtain 
\begin{align}\label{eq: PD1}
\prob(\Dd_1^c) \leq 6 (24\alpha 2^\alpha c_4^{-\alpha-1}c_3 + 2\eta^4)  < 18\eta^4 < \frac{\eta}{1000},
\end{align}
for $N$ sufficiently large, by \eqref{eq: c4alpha} and \eqref{eq: eta1000}.\\

\noindent
\textbf{The event $\Dd_2$} in \eqref{eq: D2} says that in every interval of length $c_5\ell_N$ in $[t_3,t_1]$ there is a particle which performs a jump of size greater than~$2c_4a_N$. We introduce a slightly different event to show that $\Dd_2$ happens with high probability. Let us divide the interval $[t_3,t_1]$ into subintervals of length
$\frac{1}{2}c_5 \ell_N$, to get $\ceil{4c_5^{-1}}$ subintervals (the last subinterval may end after time $t_1$). If a jump of size greater than $2c_4a_N$ happens in each of these subintervals then $\Dd_2$ occurs. We describe this formally by the following event:
\begin{equation*}
\tilde \Dd_2 := \bracing{ \begin{array}{l}
\forall m \in\bracing{1,2,\dots,\ceil{4c_5^{-1}}}, \\ \exists(k,b,s)\in\Nonetwo\times \llbracket t_3+(m-1)\frac{1}{2}c_5 \ell_N,t_3+m\frac{1}{2}c_5 \ell_N \rrbracket :\: \\
X_{k,b,s} > 2c_4a_N
\end{array}
 };
\end{equation*}
as mentioned above, if $\tilde \Dd_2$ occurs then $\Dd_2$ occurs.
The complement event of $\tilde \Dd_2$ is that there is a subinterval in which every jump made by a particle has size at most $2c_4a_N$. 
Note that in each subinterval $\llbracket t_3+(m-1)\frac{1}{2}c_5 \ell_N,t_3+m\frac{1}{2}c_5 \ell_N \rrbracket$, there are at least $2N \cdot \frac 12 c_5 \ell_N$ jumps.
Therefore, by a union bound, we have
\begin{align*}
\prob(\Dd_2^c) \leq \prob(\tilde{\Dd}_2^c)  
\leq \ceil{4c_5^{-1}} \left(1 - \frac{1}{h(2c_4a_N)}\right)^{c_5 \ell_N N}
& 
\leq (4c_5^{-1}+1) \exp\left(-\frac{c_5N \ell_N}{h(2c_4a_N)}\right)
\\ &
\leq
5c_5^{-1} \exp\left(-\frac{c_5(2c_4)^{-\alpha}}{4}\right),\numberthis\label{eq: PD2}
\end{align*}
where in the third inequality we use that $1-x\le e^{-x}$ for $x\geq 0$, and the fourth inequality follows by~\eqref{eq: hr12Nl} for $N$ sufficiently large and since $c_5<1$. Now note that by \eqref{c3},
\begin{equation*}
c_5 c_4^{-\alpha} = c_5^{1-4\alpha/(1\wedge\alpha)} \ge c_5^{-3} > 2^{2+\alpha}\log\Big(\frac{5000}{c_5\eta}\Big),
\end{equation*}
where the last inequality holds because $c_5^{-1} > 2^{2+\alpha}$ by \eqref{eq: cj}, $0 < \log x < x$ for $x>1$, and $c_5^{-1} > \frac{5000}{\eta}$ by \eqref{eq: cjeta}. Substituting this into \eqref{eq: PD2} shows that $\prob(\Dd_2^c) < \eta/1000$.\\

\noindent
\textbf{The event $\Dd_3^c$} defined in \eqref{eq: D3} says that every jump in the time interval $[t_2, t_2 + \ceil{\ell_N/2}]$ has size at most $2c_6 a_N$. There are at least $N\ell_N$ jumps in this time interval, and so for $N$ sufficiently large, since $e^{-x}\ge 1-x$ for $x\ge 0$, and then by~\eqref{eq: hr12Nl},
\begin{align*}
\prob(\Dd_3^c) \leq \left(1 - \frac{1}{h(2c_6a_N)}\right)^{N\ell_N} \leq  \exp\left(-\frac{N\ell_N}{h(2c_6a_N)}\right) 
\leq \exp\left(-\frac{(2c_6)^{-\alpha}}{4}\right). \numberthis\label{eq: PD3}
\end{align*}
Now \eqref{c5} and \eqref{eq: eta} tell us that $c_6^{-\alpha} = \eta^{-2\alpha} > 2^{\alpha+2}\log(\frac{1000}{\eta})$, and substituting this into \eqref{eq: PD3} shows that $\prob(\Dd_3^c)<\eta/1000$.\\

\noindent
\textbf{The event $\Dd_4^c$} 
(see \eqref{eq: D4}) says that in the time interval $[t_2-\ceil{c_5\ell_N},t_2]$, a particle performs a jump of size greater than $c_6a_N$ (recall from~\eqref{c5} and~\eqref{c4} that $c_5 \ll c_6$). 
Since there are at most $2N(\ceil{c_5 \ell_N}+1)\leq 2N(c_5 \ell_N+2)$ jumps in the time interval $[t_2-\ceil{c_5 \ell_N}, t_2]$, by a union bound,
\begin{align*}
\prob(\Dd_4^c) & = \prob(\exists (i,b,s)\in\Nonetwo \times \llbracket t_2 - \lceil c_5 \ell_N \rceil , t_2 \rrbracket :\: X_{i,b,s} > c_6a_N) \\ &
\leq \frac{2N(c_5 \ell_N +2)}{h(c_6 a_N)} 
\leq 2c_5c_6^{-\alpha}(1+2c_5^{-1} \ell_N^{-1}) 
\leq 4\eta^{6(1\vee\alpha)}\eta^{-2\alpha} 
< \frac{\eta}{1000},\numberthis\label{eq: PD4}
\end{align*}
for $N$ sufficiently large, where in the second inequality we use \eqref{eq: hr2Nl}, the third inequality holds by the choices in \eqref{c4} and \eqref{c5} for $N$ sufficiently large, and the fourth follows by \eqref{eq: eta1000}.\\

\noindent
\textbf{The event $\Dd_5^c$} (see \eqref{eq: D5}) says that in a short time interval after time $\tau_2$ (defined in \eqref{eq: tau}) a jump is performed whose size falls into a small interval, $(2c_4a_N,(2c_4+3c_3)a_N]$. We can see from the definition of $\tau_2$ as the first time after $t_2$ when the diameter is at most $\frac{3}{2}c_4a_N$, that $\tau_2$ is a stopping time. Therefore we can condition on $\F_{\tau_2}$, and apply the strong Markov property. By Markov's inequality we have
\begin{align*}
\prob(\Dd_5^c) & = \prob(\exists (k,b,s)\in\Nonetwo \times \llbracket \tau_2,\tau_2 + c_5\ell_N \rrbracket :\:  X_{k,b,s}\in (2c_4a_N,(2c_4+3c_3)a_N]) \\ &
\leq \expect\left[\expect[\#\bracing{(k,b,s)\in\Nonetwo \times \llbracket \tau_2,\tau_2 + c_5\ell_N \rrbracket :\:  X_{k,b,s}\in (2c_4a_N,(2c_4+3c_3)a_N]}\left| \mathcal{F}_{\tau_2} \right.] \right]. 
\end{align*}
Note that if $\tau_2<\infty$ then during the time interval $[\tau_2,\tau_2 + c_5\ell_N]$ there are at most $2N(c_5\ell_N+1)$ jumps; it follows that
\begin{align*}
\prob(\Dd_5^c)&\le \expect\Bigg[ \sum_{(k,b,s)\in\Nonetwo\times \llbracket \tau_2,\tau_2 + c_5\ell_N \rrbracket} \prob\left(\left. X_{k,b,s}\in (2c_4a_N, (2c_4+3c_3)a_N] \right| \mathcal F_{\tau_2} \right)\mathds{1}_{\bracing{\tau_2<\infty}}\Bigg]
\\ &
\leq 2N(c_5\ell_N+1)\left(h(2c_4a_N)^{-1} - h((2c_4+3c_3)a_N)^{-1}\right)
\numberthis\label{eq: PD5c2c3}
\end{align*}
by the strong Markov property.
Now we can use the monotonicity of $h$ and then the upper bound~\eqref{eq: hzc2} to get
\begin{align*}
h(2c_4a_N)^{-1} - h((2c_4+3c_3)a_N)^{-1} & \leq h((2c_4-3c_3)a_N)^{-1} - h((2c_4+3c_3)a_N)^{-1}
\\ &
\leq 
h(a_N)^{-1}(24\alpha\cdot 2^\alpha c_4^{-\alpha-1}c_3 + 2\eta^4)\numberthis\label{eq: hc3}
\end{align*}
for $N$ sufficiently large.
Therefore, by \eqref{eq: PD5c2c3}, \eqref{eq: hc3}, and \eqref{eq: haNasymp}, we have that for $N$ sufficiently large,
\begin{equation}\label{eq: PD5}
\prob(\Dd_5^c) \leq (1+c_5^{-1} \ell_N^{-1})c_5(1 + \eta^4) (24\alpha 2^\alpha c_4^{-\alpha-1}c_3 + 2\eta^4) < 4c_5\cdot 3\eta^4 < \frac{\eta}{1000},
\end{equation}
where in the second inequality we use \eqref{eq: c4alpha} and that $(1+c_5^{-1} \ell_N^{-1})(1 + \eta^4)<4$ for $N$ sufficiently large, and the last inequality follows by \eqref{eq: cjeta} and \eqref{eq: eta1000}. This concludes the proof of Lemma~\ref{lemma: probC}.
\end{proof}

We have seen in Lemma~\ref{lemma: probC} above that with an appropriate choice of constants, the probabilities of the events $\Cc_2$ to $\Cc_7$ and $\Dd_1$ to $\Dd_5$ which imply $\Aa_1$ and $\Aa_3$ are close to 1. We can now use this to prove Proposition~\ref{prop: A1A3}.

\begin{proof}[Proof of Proposition~\ref{prop: A1A3}]
Take $\eta \in (0,1]$. Without loss of generality, we can assume that $\eta$ is sufficiently small that it satisfies~\eqref{eq: eta}. Then choose $K$, $\gamma$, $\delta$, $\rho$, $c_1,\dots,c_6$ as in \eqref{c5}-\eqref{K} (at the beginning of Section~\ref{sect: probCD}). Note that before stating Lemma~\ref{lemma: probC} we checked that these constants also satisfy~{\eqref{eq: const1}-\eqref{eq: const5}}. Therefore by Proposition~\ref{prop: B} and Proposition~\ref{prop: C},
for $N$ sufficiently large and $t>4\ell_N$,
$$
\bigcap_{j=2}^7 \Cc_j \cap \bigcap_{i=1}^5 \Dd_i \subseteq \Aa_1 \cap \Aa_3.
$$
Therefore, for $N$ sufficiently large and $t>4\ell_N$, by a union bound,
\begin{align*}
\prob((\Aa_1 \cap \Aa_3)^c) & \leq \prob\left(\left(\bigcap_{j=2}^7 \Cc_j \cap \bigcap_{i=1}^5 \Dd_i \right)^c\right)   
\leq \sum_{j=2}^{7}\prob(\Cc_j^c) + \sum_{i=1}^{5 } \prob(\Dd_i^c)
< \eta
\end{align*}
by Lemma~\ref{lemma: probC}, which completes the proof.
\end{proof}

\section{Proof of Proposition~\ref{prop: A4}: star-shaped coalescence}\label{sect: star-shaped}

We will prove Proposition~\ref{prop: A4} in this section. So far we have proved Proposition~\ref{prop: A1A3}, which says that with high probability the common ancestor of the majority of the population at time $t$ is particle $(N,T)$, where $T$ is given by \eqref{eq: T}; in particular, $T$ is between times $t_2$ and~$t_1$. Now recall the notation introduced in \eqref{eq: Teps}-\eqref{eq: A4}. Proposition~\ref{prop: A4} says that for $\nu>0$, with high probability, every particle in the set $\Nn_{N,T}(T+\eps_N\ell_N)$ has at most $\nu N$ surviving descendants at time $t$, where we may assume that $(\eps_N)_{N\in \Nzero}$ satisfies
\begin{align}\label{eq: epsN}
\eps_N\ell_N\in\N_0\:\:\forall N\ge 1, \quad\quad \eps_N\ell_N \to \infty \text{ as }N\to\infty \quad \text{ and } \quad \eps_N \leq \frac{1}{4}\frac{\log_2 \ell_N}{\ell_N}\:\: \forall N \geq 1.
\end{align}
The first two of these assumptions on $\eps_N$ are from \eqref{eq: epsN1}. The third can be made without loss of generality, because if $\eps_N'>\eps_N$, and every particle in $\Nn_{N,T}(T+\eps_N\ell_N)$ has at most $\nu N$ surviving descendants at time $t$, then certainly every particle in $\Nn_{N,T}(T+\eps_N'\ell_N)$ has at most $\nu N$ surviving descendants at time $t$.

Fix $\eta \in (0,1]$ sufficiently small that it satisfies~\eqref{eq: eta}. Then choose $K$, $\gamma$, $\delta$, $\rho$, $c_1,\dots,c_6$ as in \eqref{c5}-\eqref{K}. Then take $N$ sufficiently large that Proposition~\ref{prop: A1A3} and Lemma~\ref{lemma: probC} hold for our chosen constants, and take $t > 4\ell_N$. Let $\nu > 0$ be fixed and let us write $\Aa_4 := \Aa_4(\nu)$ from now on.

\subsection{Strategy}\label{sect: star-shaped-strategy}
Our strategy for showing Proposition~\ref{prop: A4} is to give a lower bound on the position of the leftmost particle at time $t$ with high probability, and then bound the number of time-$t$ descendants of each particle in $\Nn_{N,T}(\Teps)$ which can reach that lower bound by time $t$. We will be able to control the number of such descendants because of Corollary~\ref{cor: path}. Assume that we know $\Xone(t)\geq \XN(T)+\hat a_{T,N}$, where $\hat a_{T,N} > N^\lambda$ for some $\lambda>0$, but $\hat a_{T,N}\ll a_N$. Then Corollary~\ref{cor: path} implies that with high probability all surviving particles at time $t$ must have an ancestor which made a jump of size greater than $r\hat a_{T,N}$ for an appropriate choice of~$r\in (0,1)$. So given a particle $i\in\Nn_{N,T}(\Teps)$, we can find an upper bound for the number of its time-$t$ descendants with high probability, by considering the number of its descendants which made a jump of size greater than $r\hat a_{T,N}$ before time $t$. Thus, we should choose $\hat a_{T,N}$ such that we have $\Xone(t)\geq \XN(T)+\hat a_{T,N}$ with high probability, and also such that we can get a good enough upper bound for each $D_i$ (see \eqref{eq: Di}) from Corollary~\ref{cor: path} to conclude Proposition~\ref{prop: A4}.

We now give a sketch argument to motivate our choice of lower bound on $\Xone(t)$. Assume that $T\in[t_2+\ceil{\delta\ell_N}, t_1-\ceil{\delta\ell_N}]$. We also assume that the record set at time $T$ is not broken by a big jump before time $t_1+\delta\ell_N$, and so almost all the descendants of particle $(N,T)$ survive between times $T$ and $T+\ell_N$. This all happens with high probability, as we saw in Section \ref{sect: probabilities}; in particular recall the event $\Cc_6$ from \eqref{eq: C6}. Set $\theta_{T,N} := (t_1-T)/\ell_N$.

Note that if a descendant of particle $(N,T)$ makes a jump of size greater than $\hat a_{T,N}$ at time $T+k$ for some $k\in[(1-\delta)\ell_N,\ell_N]$, then it can have $2^{(1+\theta_{T,N})\ell_N - k}$ descendants at time $t$, and all of these descendants are to the right of $\XN(T)+\hat a_{T,N}$. Also, there are approximately $2^k$ particles in the leading tribe descending from $(N,T)$ at time $T+k$. Therefore, we expect that jumps of size greater than $\hat a_{T,N}$, performed by the descendants of $(N,T)$ in the time interval $[T+(1-\delta)\ell_N, T+\ell_N]$, contribute to the number of particles to the right of $\XN(T)+\hat a_{T,N}$ at time $t$ by roughly
\begin{equation*}
\sum_{k\in \llbracket (1-\delta)\ell_N, \ell_N \rrbracket} 2^k\cdot 2^{(1+\theta_{T,N})\ell_N - k}\frac{1}{h(\hat a_{T,N})}\approx \delta \ell_N 2^{(1+\theta_{T,N})\ell_N }\frac{1}{h(\hat a_{T,N})}.
\end{equation*}
If we want to make sure that all the $N$ particles are to the right of $\XN(T)+\hat a_{T,N}$ at time $t$, then the above should be approximately $N$, and so $\hat a_{T,N}$ should be roughly $h^{-1}(\delta\ell_N N^{\theta_{T,N}})$.

There are several potential inaccuracies in this argument. For example, the descendants of a particle making a jump of size greater than $\hat a_{T,N}$ do not necessarily all survive until time $t$. We will use a reasoning similar to Lemma~\ref{lemma: descendants_est} to clarify this issue. Another problem might occur if a particle $(i,T+k)$ makes a jump of size greater than $\hat a_{T,N}$, and then at time $T+k+1$, its offspring does the same. In this case our sketch argument double counts the time-$t$ descendants of particle $(i,T+k)$. We will therefore make some adjustments in the rigorous proof to avoid double counting.  

In Sections~\ref{sect:stoppingtimes} to \ref{sect: withBigJump} below, we will make the sketch argument precise, then use Corollary~\ref{cor: path} to see that with high probability, particles must have at least one jump greater than a certain size~(roughly but not exactly $h^{-1}(\delta\ell_N N^{\theta_{T,N}})$) in their ancestry to survive until time $t$. Finally, for each particle~$(i,\Teps)$, we upper bound the number of particles at time $t$ which descend from particle $(i,\Teps)$ and have a jump greater than this certain size in their ancestry between times $\Teps$ and $t$. 

\subsection{Sequence of stopping times} \label{sect:stoppingtimes}
In the strategy above we suggested that $h^{-1}(\delta\ell_N N^{\theta_{T,N}})$ should be a good lower bound for $\X_1(t)-\X_N(T)$. A problem with this lower bound is that it depends on $T$, and conditioning on~$T$ would change the distribution of the process, as $T$ is not a stopping time; see the definition in~\eqref{eq: T}.

Note however, that the first, second, $\dots, n$th times after time $t_2$ at which a jump of size greater than~$\rho a_N$ breaks the record between times $t_2$ and $t_1$, are stopping times, and $T$ is equal to one of these times with high probability. Furthermore, the number of such times is at most~$K$ with high probability, by Lemma~\ref{lemma: probC} and the definition of the event $\Cc_7$. Therefore, we can define a finite set of stopping times in such a way that~$T$ is in the set with high probability. Then we can prove a similar statement to Proposition~\ref{prop: A4} for each stopping time in the finite set with the strategy described in the previous section. This will be enough to prove Proposition~\ref{prop: A4}.

Recall the definition of $\Sbf_N$ in~\eqref{eq: RN}.
Define a sequence of stopping times by setting $T_0 := t_2+\ceil{\delta \ell_N}-1$, and
\begin{equation}\label{eq: Tn}
T_n := 1 + \inf\bracing{\Sbf_N(\rho) \cap [T_{n-1}, t_1 - \ceil{\delta \ell_N}-1]}, 
\end{equation}
for $n \in \N$; let $T_n := t_1$ if the intersection above is empty.

For all $n\in \N$, we introduce some new notation which will be frequently used in the course of the proof. First we let
\begin{equation}\label{eq: Tneps}
\Tneps := T_n + \eps_N \ell_N.
\end{equation}
The set and number of time-$t$ descendants of the $i$th particle at time $\Tneps$ will be denoted by
\begin{equation}\label{eq: Nin}
\Nn_{i,n} := \Nn_{i,\Tneps}(t) \quad \text{ and } \quad D_{i,n} := |\Nn_{i,n}|.
\end{equation}
We also introduce
\begin{equation}\label{eq: thetan}
\theta_{n,N} := \frac{(t_1 - T_n)}{\ell_N}\ge 0.
\end{equation}
Take $0<\delta_1<\delta / 8$ and set
\begin{equation}\label{eq: hat_an}
\hat a_{n,N} := h^{-1}(\delta_1 N^{\theta_{n,N}}\ell_N),
\end{equation}
where $h^{-1}$, defined in \eqref{eq: hinv}, is the generalised inverse of $h$ from \eqref{eq: poly_tail}. We explained the motivation for this definition of $\hat a_{n,N}$ in Section~\ref{sect: star-shaped-strategy}. By the same argument as for \eqref{eq: haNasymp} (and since $\delta_1 N^{\theta_{n,N}}\ell_N \ge \delta_1 \ell_N$)
 we have that for $\epsilon>0$, for $N$ sufficiently (deterministically) large, for each $n\in \N$,
\begin{equation}\label{eq: hhataNasymp}
\frac{\delta_1 N^{\theta_{n,N}}\ell_N}{h(\hat a_{n,N})} \in [1-\epsilon,1+\epsilon].
\end{equation}
We note that $\hat a_{n,N}$ is roughly $N^{\theta_{n,N}/\alpha}$; in particular, if $h(x) = x^\alpha$ for $x \ge 1$ then $\hat a_{n,N} = (\delta_1 N^{\theta_{n,N}}\ell_N)^{1 / \alpha}$.

Take $0<\delta_2<\delta^2$.
Throughout Section~\ref{sect: star-shaped} we will use the term `medium jump' for jumps of size greater than $\delta_2 \hat a_{n,N}$, as the relevant space scale in this section is $\hat a_{n,N}$. We denote the set of medium jumps on a time interval~${[s_1,s_2]\subseteq [t_2,t-1]}$ by
\begin{align}\label{eq: hat_Bn1}
\Mm_{n,N}^{[s_1,s_2]} := \bracing{(k,b,s)\in\Nonetwo\times \llbracket s_1,s_2 \rrbracket :\: X_{k,b,s}> \delta_2 \hat a_{n,N}},
\end{align}
and we let
\begin{equation}\label{eq: hat_Bn2}
\Mm_{n,N} := \Mm_{n,N}^{[t_2,t-1]}.
\end{equation}

The stopping times $(T_n)_{n\in \N}$ allow us to give an upper bound on the probability of $\Aa_4^c$. 
Suppose $|B_N^{[t_2,t_1]}|\le K$ and $T\in[t_2 + \ceil{\delta \ell_N}, t_1 - \ceil{\delta \ell_N}]$.
Then $|\Sbf_N(\rho)\cap [t_2,t_1]|\le K$ by the definition of $\Sbf_N$ in~\eqref{eq: RN}, and so by the definition of $T$ in~\eqref{eq: T} and the definition of $T_n$ in~\eqref{eq: Tn}, it follows that $T=T_n$ for some $n\in [K]$.
Hence, by the definition of $\Aa_4$ in~\eqref{eq: A4} and then by a union bound,
\begin{align} \label{eq:A4cunion}
\prob(\Aa_4^c) &=\prob\bigg(\max_{i\in\Nn_{N,T}(\Teps)} D_i > \nu N\bigg) \notag
\\ 
&\leq 
\prob\bigg(\exists n\in[K]:\:T_n \leq t_1 - \ceil{\delta \ell_N}  \text{ and } \max_{i\in\Nn_{N,T_n}(\Tneps)} D_{i,n}> \nu N \bigg) \notag 
\\  &\quad 
\hspace{35mm} + \prob\big(|B_N^{[t_2,t_1]}| > K\big) + \prob\big(T\notin[t_2 + \ceil{\delta \ell_N}, t_1 - \ceil{\delta \ell_N}]\big).
\end{align}
By the definition of the event $\Cc_7$ in~\eqref{eq: C7} and by Lemma~\ref{lemma: probC}, 
$$\prob(|B_N^{[t_2,t_1]}| > K)\le \prob(\Cc_7^c)< \frac{\eta}{1000}.$$
Then by the definition of the event $\Aa_3$ in~\eqref{eq: A3} and by Proposition~\ref{prop: A1A3},
$$
\prob(T\notin[t_2 + \ceil{\delta \ell_N}, t_1 - \ceil{\delta \ell_N}]) \le \prob(\Aa_3^c)<\eta.
$$
Therefore, applying a union bound for the first term on the right-hand side of~\eqref{eq:A4cunion}, we obtain
\begin{align*}
\prob(\Aa_4^c) & 
\leq
\expect\left[\sum_{n = 1}^{K}
\mathds{1}_{\bracing{T_n \leq t_1 - \ceil{\delta \ell_N}}}
\prob\bigg(\max_{i\in\Nn_{N,T_n}(\Tneps)} D_{i,n}> \nu N\; \bigg|\; \F_{T_n}\bigg)\right] + \frac{1001}{1000}\eta .\numberthis\label{eq: DiTn}
\end{align*}
From now on we aim to show that each term of the sum inside the expectation is small. For all $n\in \N$, we let $\prob_{T_n}$ denote the law of the~$N$-BRW conditioned on~$\F_{T_n}$:
\begin{align}\label{eq: PTn}
\prob_{T_n}(\cdot) := \prob(\left.\cdot\;\right|\;\F_{T_n}) \quad \text{ and }\quad \expect_{T_n}[\cdot] := \expect[\left.\cdot\;\right|\;\F_{T_n}].
\end{align}

\subsection{Proof of Proposition~\ref{prop: A4}}
We now state the most important intermediate results in the proof of Proposition~\ref{prop: A4}, and show that they imply the result. We then prove these intermediate results in Sections~\ref{sect: noBigJump} and \ref{sect: withBigJump}. 


Our first main intermediate result says that the probability that a particle in $\Nn_{N,T_n}(\Tneps)$ has a descendant at time $t$ such that there is no medium jump on the path between the particle and the descendant is small. We prove this result in Section~\ref{sect: noBigJump}. 
\begin{lemma}\label{lemma: noBigJump}
For all $N$ sufficiently large, $t>4\ell_N$, and $n\in\N$ with $T_n < t_1$,
\begin{align*}
\probT\left( \exists i\in\Nn_{N,T_n}(\Tneps), k\in\Nn_{i,n}:\: P_{i,\Tneps}^{k,t}\cap \Mm_{n,N} = \emptyset \right) < \frac{\eta}{100K},
\end{align*}
where $T_n$, $\Tneps$ and $\probT$ are given by \eqref{eq: Tn}, \eqref{eq: Tneps} and \eqref{eq: PTn}, $\Nn_{N,T_n}(\Tneps)$ and $\Nn_{i,n}$ are defined in~\eqref{eq: N} and~\eqref{eq: Nin},~$P_{i,\Tneps}^{k,t}$ in \eqref{eq: P}, and $\Mm_{n,N}$ in \eqref{eq: hat_Bn2}.
\end{lemma}
Our second intermediate result 
says that with high probability, for each $i\in\Nn_{N,T_n}(\Tneps)$, there cannot be more than $\nu N$ time-$t$ descendants of particle $(i,\Tneps)$ if each descendant has a medium jump on their path. We prove this result in Section~\ref{sect: withBigJump}.

\begin{lemma}\label{lemma: withBigJump}
For all $N$ sufficiently large, $t>4\ell_N$, and $n\in \N$ with $T_n < t_1$,
\begin{align*}
\probT\left(\exists i\in\Nn_{N,T_n}(\Tneps):\: D_{i,n} > \nu N \text{ and }P_{i,\Tneps}^{k,t}\cap \Mm_{n,N} \neq \emptyset \;\, \forall k\in\Nn_{i,n}\right) < \frac{\eta}{100K},
\end{align*}
where $T_n$, $\Tneps$ and $\probT$ are given by \eqref{eq: Tn}, \eqref{eq: Tneps} and \eqref{eq: PTn}, $\Nn_{N,T_n}(\Tneps)$, $\Nn_{i,n}$ and $D_{i,n}$ are defined in \eqref{eq: N} and \eqref{eq: Nin},~$P_{i,\Tneps}^{k,t}$ in \eqref{eq: P}, and $\Mm_{n,N}$ in \eqref{eq: hat_Bn2}.
\end{lemma}

\begin{proof}[Proof of Proposition~\ref{prop: A4}]
Suppose $N$ is sufficiently large that Lemmas~\ref{lemma: noBigJump} and~\ref{lemma: withBigJump} hold.
Take $n\in \N$ and suppose $T_n<t_1$ (which also implies $T_n\leq t_1 - \ceil{\delta\ell_N}$ by the definition \eqref{eq: Tn} of $T_n$). Suppose a particle in $\Nn_{N,T_n}(\Tneps)$ has more than $\nu N$ surviving descendants at time $t$. Then either all the descendants have an ancestor which performed a medium jump between times $\Tneps$ and $t$, or there is at least one particle which survives without a medium jump in its ancestry. Therefore we have
\begin{align*}
&\probT\bigg(\max_{i\in\Nn_{N,T_n}(\Tneps)} D_{i,n}> \nu N\bigg) \\
&\leq\: 
\probT\left( \exists i\in\Nn_{N,T_n}(\Tneps), k\in\Nn_{i,n}:\: P_{i,\Tneps}^{k,t}\cap \Mm_{n,N} = \emptyset \right)
\\ &\quad + \probT\left(\exists i\in\Nn_{N,T_n}(\Tneps):\: D_{i,n} > \nu N \text{ and }P_{i,\Tneps}^{k,t}\cap \Mm_{n,N} \neq \emptyset \;\, \forall k\in\Nn_{i,n}\right)\\
&< \frac{\eta}{50 K} \numberthis\label{eq: DinSplit}
\end{align*}
by Lemmas~\ref{lemma: noBigJump} and~\ref{lemma: withBigJump}.
Then by~\eqref{eq: DiTn}, it follows that
$$
\prob(\Aa_4^c)< K \cdot \frac{\eta}{50 K} +\frac{1001}{1000}\eta <2\eta,
$$
which completes the proof.
\end{proof}

\subsection{Leaders must take medium jumps to survive: proof of Lemma~\ref{lemma: noBigJump}}\label{sect: noBigJump}

There are two key ideas in the proof. First we show that for a fixed $n\in \N$ with $T_n<t_1$, the whole population is to the right of position $\XN(T_n) + \hat a_{n,N}$ at time $t$, with high probability. Second, we prove that with high probability paths cannot reach position $\XN(T_n) + \hat a_{n,N}$ without having a medium jump on the path. 

\begin{lemma}\label{lemma: leftmost}
	For all $N$ sufficiently large, $t>4\ell_N$, and $n\in \N$ with $T_n < t_1$, 
	\begin{align*}
	\probT(\Xone(t) <  \XN(T_n) + \hat a_{n,N}) < \frac{\eta}{200K},
	\end{align*}
	where $T_n$ and $\hat a_{n,N}$ are given by \eqref{eq: Tn} and \eqref{eq: hat_an} respectively.
\end{lemma}

\begin{proof}
	Recall the definition of $G_x(n)$ in~\eqref{eq: G}. Let $G := G_{\XN(T_n) + \hat a_{n,N}}(t)$; then, to prove the statement of the lemma, we aim to show that for $N$ sufficiently large and $t>4\ell_N$,
	\begin{equation}\label{eq: GeqN}
	\probT(|G| < N) < \frac{\eta}{200K}.
	\end{equation}

	Recall the definition of $\delta_1>0$ in \eqref{eq: hat_an}; fix $\delta '\in (8\delta_1,\delta)$ and then take $\delta_3 \in (8\delta_1 , \delta' )$ such that $\delta_3\ell_N$ is an integer (this is possible for $N$ sufficiently large). Let $S_k := T_n + \ell_N - k$ for $k\in \llbracket 1,\delta_3\ell_N \rrbracket $. Then for each $k\in \llbracket 1,\delta_3\ell_N \rrbracket $, at time $S_k$ there are at least $2^{\ell_N-k}$ particles to the right of (or at) position $\XN(T_n)$, by Lemma~\ref{lemma: descendants_est}. These particles are either in the interval $[\XN(T_n),\XN(T_n)+\hat a_{n,N})$ or to the right of this interval. Let us denote the set of particles in $[\XN(T_n),\XN(T_n)+\hat a_{n,N})$ at time $S_k$ by $A_k$,
	i.e.~for $k\in \llbracket 1,\delta_3\ell_N \rrbracket $ let
	\begin{align*}
	A_k := \bracing{i\in[N]:\: \X_i(S_k) \in [\XN(T_n),\XN(T_n)+\hat a_{n,N})}.
	\end{align*}
	We will handle the following two cases separately:
	\begin{enumerate}[label=(\alph*), ref=\alph*]
		\item the event $\Ee := \bracing{|A_k| \geq \frac{1}{2}2^{\ell_N-k}\;\, \forall k\in \llbracket 1,\delta_3\ell_N \rrbracket }$ occurs,
		\item the event $\Ee^c = \bracing{\exists k\in \llbracket 1,\delta_3\ell_N \rrbracket :\;| G_{\XN(T_n) + \hat a_{n,N}}(S_k)| > \frac{1}{2}2^{\ell_N-k}}$ occurs.
	\end{enumerate}

	First we deal with case (a). We give a lower bound on $|G|$ using a similar argument to the proof of Lemma~\ref{lemma: descendants_est}. First note that jumps of size greater than $\hat a_{n,N}$ from particles in $A_k$ arrive to the right of position $\XN(T_n) + \hat a_{n,N}$ for all $k\in \llbracket 1,\delta_3\ell_N \rrbracket $. Thus all time-$t$ descendants of a particle that makes such a jump will be in the set $G$. For $k\in \llbracket 1,\delta_3\ell_N \rrbracket $, let $\Mm_k'$ denote the set of such jumps:
	\begin{equation*}
	\Mm_k' := \bracing{(i,b,S_k):\: X_{i,b,S_k}> \hat a_{n,N} \text{ and } i\in A_k}.
	\end{equation*}
	Suppose for all $k\in \llbracket 1,\delta_3\ell_N \rrbracket $, all particles descending from the set $\Mm_k'$ survive until time $t$. Then the total number of such descendants will be
	\begin{equation}\label{eq: BAkdesc}
	\Bigg|\bigcup_{k\in \llbracket 1,\delta_3\ell_N \rrbracket } \bigcup_{(i,b,S_k)\in \Mm_k'} \Nn_{i,S_k}^b(t) \Bigg| = \sum_{k=1}^{\delta_3\ell_N}2^{k+\theta_{n,N}\ell_N-1}
	\sum_{i\in A_k, b\in\bracing{1,2}} \mathds{1}_{\bracing{X_{i,b,S_k}> \hat a_{n,N}}}.
	\end{equation}
	The first term in the sum is the number of time-$t$ descendants of a particle at time $S_k+1=T_n+\ell_N-k+1$, and the second sum gives the number of jumps of size greater than $\hat a_{n,N}$ from particles in $A_k$.
	
	If instead there exists $k\in \llbracket 1,\delta_3\ell_N \rrbracket $ such that not every particle descending from a jump in $\Mm_k'$ survives until time $t$, then there must be $N$ particles to the right of (or at) $\XN(T_n) + \hat a_{n,N}$ at some time~$s\le t$ (and therefore at time $t$, by monotonicity). We conclude the following lower bound:
	\begin{align*}
	|G| & \geq \min\Bigg( N, \sum_{k=1}^{\delta_3\ell_N}2^{k+\theta_{n,N}\ell_N-1}
	\sum_{i\in A_k, b\in\bracing{1,2}} \mathds{1}_{\bracing{X_{i,b,S_k}> \hat a_{n,N}}}\Bigg). \numberthis\label{eq: G1}
	\end{align*}
	Let $\xi_{j,k}\sim$ Ber$(h(\hat a_{n,N})^{-1})$ be i.i.d.~random variables, by which we mean that 
	\begin{align*}
	\probT(\xi_{j,k} = 1) = \frac{1}{h(\hat a_{n,N})} = 1-\probT(\xi_{j,k} = 0) \quad \text{ for all } k,j\in\N.
	\end{align*}
	The indicator random variables in \eqref{eq: G1} all have this distribution. Thus by~\eqref{eq: G1},
	\begin{align*}
	\probT(\bracing{|G| < N}\cap\Ee) &\leq \probT\Bigg(\Bigg\{\sum_{k=1}^{\delta_3\ell_N}2^{k+\theta_{n,N}\ell_N-1}
		\sum_{i\in A_k, b\in\bracing{1,2}} \mathds{1}_{\bracing{X_{i,b,S_k}> \hat a_{n,N}}} < N\Bigg\}\cap\Ee\Bigg)
	\\ &
	\leq \probT\Bigg(\sum_{k=1}^{\delta_3\ell_N}2^{k+\theta_{n,N}\ell_N-1}
	\sum_{j=1}^{2^{\ell_N-k}} \xi_{j,k} < N\Bigg),
	\numberthis\label{eq: GXi}
	\end{align*}
	since on the event $\Ee$ there are at least~$2^{\ell_N-k}$ jumps from the set $A_k$ for each $k\in \llbracket 1,\delta_3\ell_N \rrbracket $.

	We will use the concentration inequality from \cite[Theorem 2.3(c)]{McDiarmid1998} to estimate the right-hand side of \eqref{eq: GXi}. As the inequality applies for independent random variables taking values in $[0,1]$, we consider the random variables $2^{-\delta_3\ell_N+k}\xi_{j,k}\in[0,1]$ for $k\in \llbracket 1,\delta_3\ell_N \rrbracket $ and~$j\in [2^{\ell_N-k}]$. 
	Let $\mu$ denote the expectation of the sum of these random variables over $k$ and $j$:
	\begin{align*}
	\mu & := \expectT\Bigg[\sum_{k=1}^{\delta_3\ell_N} 2^{-\delta_3\ell_N+k} \sum_{j=1}^{2^{\ell_N-k}} \xi_{j,k}\Bigg]
	= \sum_{k=1}^{\delta_3\ell_N} 2^{-\delta_3\ell_N+k} \frac{2^{\ell_N-k}}{h(\hat a_{n,N})}
	\geq \frac{\delta_3\ell_N N^{1-\delta_3}}{h(\hat a_{n,N})}
	\geq 4N^{1-\delta_3-\theta_{n,N}}
	\numberthis\label{eq: muT}
	\end{align*}
	for $N$ sufficiently large,
	where the last inequality holds because $h(\hat a_{n,N})\leq 2\delta_1 N^{\theta_{n,N}}\ell_N$ by \eqref{eq: hhataNasymp} for $N$ sufficiently large, and because we chose $\delta_3 / \delta_1 \geq 8$. Thus
	\begin{align*}
	\probT\Bigg(\sum_{k=1}^{\delta_3\ell_N}2^{k+\theta_{n,N}\ell_N-1}
	\sum_{j=1}^{2^{\ell_N-k}} \xi_{j,k} < N\Bigg)
	& \leq \probT\Bigg( \sum_{k=1}^{\delta_3\ell_N} 2^{-\delta_3\ell_N+k} \sum_{j=1}^{2^{\ell_N-k}} \xi_{j,k} < 2N^{1-\delta_3 - \theta_{n,N}}\Bigg)
	\\&
	\leq \probT\Bigg( \sum_{k=1}^{\delta_3\ell_N} 2^{-\delta_3\ell_N+k}  \sum_{j=1}^{2^{\ell_N-k}} \xi_{j,k} < \frac{1}{2}\mu\Bigg)
	\end{align*}
	for $N$ sufficiently large, where in the first inequality we multiply by $2^{1-(\delta_3+\theta_{n,N})\ell_N}$ to get terms in $[0,1]$ in the sum and notice that $2^{-\ell_N}\leq N^{-1}$, and the second inequality holds by~\eqref{eq: muT}. We now apply the concentration inequality from \cite[Theorem 2.3(c)]{McDiarmid1998} to the independent random variables $2^{-\delta_3\ell_N+k}\xi_{j,k}\in[0,1]$ on the right-hand side above, giving that
	\begin{equation}\label{eq: McDiarmid}
	\probT\Bigg(\sum_{k=1}^{\delta_3\ell_N}2^{k+\theta_{n,N}\ell_N-1}\sum_{j=1}^{2^{\ell_N-k}} \xi_{j,k} < N\Bigg) \le e^{-\mu / 8} \le e^{-\frac{1}{2}N^{\delta - \delta_3}},
	\end{equation}
	where in the second inequality we use \eqref{eq: muT} again and that~$\theta_{n,N}\leq 1 - \delta$ by~\eqref{eq: thetan} and since $T_n\ge t_2+\delta \ell_N$ by~\eqref{eq: Tn}. Now putting \eqref{eq: GXi} and \eqref{eq: McDiarmid} together, since $\delta-\delta_3>\delta - \delta ' >0$ we conclude that
	\begin{align}\label{eq: Ga}
	\probT(\bracing{|G| < N}\cap\Ee) < \frac{\eta}{200K}
	\end{align}
	for $N$ sufficiently large.
	
	In case (b), $\Ee^c$ deterministically implies that $|G| = N$. Indeed, if $\Ee^c$ occurs then it follows that there exists $k_0\in \llbracket 1,\delta_3\ell_N \rrbracket $ such that $| G_{\XN(T_n) + \hat a_{n,N}}(S_{k_0})| > \frac{1}{2}2^{\ell_N-k_0}$. Recall that $S_{k_0}=T_n+\ell_N -k_0$. Then by Lemma~\ref{lemma: descendants_est} we have
	\begin{align}\label{eq: Gb}
	|G| \geq \min\left(N,\tfrac{1}{2}2^{\ell_N-k_0}2^{k_0+\theta_{n,N}\ell_N}\right) = N
	\end{align}
	for $N$ sufficiently large,
	because $\theta_{n,N} \ge \delta$ by~\eqref{eq: thetan} and~\eqref{eq: Tn}, and since we are assuming $T_n<t_1$. Thus for $N$ sufficiently large,
	\begin{equation*}
	\probT(\bracing{|G| < N}\cap\Ee^c) = 0,
	\end{equation*}
	which together with \eqref{eq: Ga} and~\eqref{eq: GeqN} concludes the proof.
\end{proof}

Now we are ready to prove Lemma~\ref{lemma: noBigJump}. Corollary~\ref{cor: path} tells us that paths cannot move a large distance without having jumps which have size at least the order of magnitude of that large distance. So Lemma~\ref{lemma: leftmost} and Corollary~\ref{cor: path} together will show that paths without medium jumps cannot survive until time $t$ with high probability.

\begin{proof}[Proof of Lemma~\ref{lemma: noBigJump}]
	We partition the event in Lemma~\ref{lemma: noBigJump} based on the position of the leftmost particle:
	\begin{align*}
	&\probT\left( \exists i\in\Nn_{N,T_n}(\Tneps), k\in\Nn_{i,n}:\: P_{i,\Tneps}^{k,t}\cap \Mm_{n,N} = \emptyset \right) \\
	&=
	\probT\left( \bracing{\exists i\in\Nn_{N,T_n}(\Tneps), k\in\Nn_{i,n}: P_{i,\Tneps}^{k,t}\cap \Mm_{n,N} = \emptyset} \cap \bracing{\Xone(t) <  \XN(T_n) + \hat a_{n,N}}\right)
	\\ &
	\hspace{3mm}+ \probT\left(\bracing{\exists i\in\Nn_{N,T_n}(\Tneps), k\in\Nn_{i,n}: P_{i,\Tneps}^{k,t}\cap \Mm_{n,N} = \emptyset} \cap \bracing{\Xone(t) \geq  \XN(T_n) + \hat a_{n,N}}\right).\numberthis\label{eq: X1smallbig}
	\end{align*}
	This will be useful, because from Lemma~\ref{lemma: leftmost} we know that the leftmost particle at time $t$ is to the right of (or at) $\XN(T_n) + \hat a_{n,N}$ with high probability. Hence it is enough to focus on the second term on the right-hand side of~\eqref{eq: X1smallbig}, and show that with high probability, paths cannot move beyond $\XN(T_n) + \hat a_{n,N}$ without medium jumps.
	
	Assume that the event in the second term on the right-hand side of \eqref{eq: X1smallbig} occurs with $i\in\Nn_{N,T_n}(\Tneps)$ and $k\in\Nn_{i,n}$, and so we have $P_{i,\Tneps}^{k,t}\cap \Mm_{n,N} = \emptyset$ and $\X_{k}(t) \geq \Xone(t) \geq \XN(T_n) + \hat a_{n,N}$. Note that particle~$(k,t)$ is a descendant of particle $(N,T_n)$ as well.	The path between these two particles has to move distance at least $\hat a_{n,N}$. Thus one of the following must happen. Either the path between particles $(N,T_n)$ and~$(k,t)$ moves $\hat a_{n,N}$ even without medium jumps, or there must be a medium jump on this path. In the latter case the medium jump must be in the time interval $[T_n,\Tneps-1]$, because we assumed $P_{i,\Tneps}^{k,t}\cap \Mm_{n,N} = \emptyset$. This leads to the following upper bound:
	\begin{align*}
	&\probT\left(\bracing{\exists i\in\Nn_{N,T_n}(\Tneps), k\in\Nn_{i,n}:\: P_{i,\Tneps}^{k,t}\cap \Mm_{n,N} = \emptyset} \cap \bracing{\Xone(t) \geq  \XN(T_n) + \hat a_{n,N}}\right)
	\\ 
	&\leq \probT\Bigg(
	\exists k\in\Nn_{N,T_n}(t):\:
	\sum_{(i,b,s)\in P_{N,T_n}^{k,t}} X_{i,b,s}\mathds{1}_{\bracing{X_{i,b,s}\leq \delta_2 \hat a_{n,N}}} \ge \hat a_{n,N}\Bigg) 
	\\ &
	\hspace{10mm}+ \probT\left( \exists s\in \llbracket T_n,\Tneps-1 \rrbracket,\; i\in\Nn_{N,T_n}(s) \text{ and } b\in\bracing{1,2}:\: X_{i,b,s} > \delta_2\hat a_{n,N} \right)\\
	&\leq CN^{-1} +  \probT\left( \exists s\in \llbracket T_n,\Tneps-1 \rrbracket,\; i\in\Nn_{N,T_n}(s) \text{ and } b\in\bracing{1,2}:\: X_{i,b,s} > \delta_2\hat a_{n,N} \right)\numberthis\label{eq: noBigJumpX1big}
	\end{align*}	
	for $N$ sufficiently large,
	where the second inequality holds for some constant $C>0$ because of Corollary~\ref{cor: path} applied with $x_N = \hat a_{n,N}$, $r = \delta_2$ and $\lambda = \delta/ (2\alpha)$. To check the conditions of Corollary~\ref{cor: path} we first notice that we chose $\delta_2 < \delta^2$, and claim that $\delta^2 < 1\wedge \frac{\delta(1\wedge\alpha)}{96\alpha}$. Indeed, at the beginning of Section~\ref{sect: star-shaped} we chose $\delta$ together with the other constants $\eta$, $K$, $\gamma$, $\rho$, $c_1,\dots,c_6$ satisfying \eqref{c5}-\eqref{K}. Then \eqref{delta}, \eqref{sigma}, \eqref{eq: cjeta} and \eqref{eq: eta1000} (using the fact that $2^{2\alpha} > \alpha$ for $\alpha>0$) easily imply the claim. Regarding the condition that $x_N>N^\lambda$,	we have $\hat a_{n,N} > N^{\theta_{n,N} / 2\alpha} \geq N^{\delta / 2\alpha}$ for $N$ sufficiently large, where the first inequality follows by~\eqref{eq: hhataNasymp} and Lemma~\ref{lemma: Potter} by the same argument as for~\eqref{eq: aNhaN} and \eqref{eq: aN2alpha}, and the second inequality holds because $\theta_{n,N} \geq \delta$ by \eqref{eq: thetan}, \eqref{eq: Tn} and since we are assuming $T_n<t_1$.
	
	Next we use a union bound to control the second term on the right-hand side of \eqref{eq: noBigJumpX1big}, using that there are at most $2\cdot2^k$ jumps descending from particle $(N,T_n)$ at time $T_n+k$. We have
	\begin{multline*}
	\probT\left( \exists s\in \llbracket T_n,\Tneps -1 \rrbracket,\; i\in\Nn_{N,T_n}(s) \text{ and } b\in\bracing{1,2}:\: X_{i,b,s} > \delta_2\hat a_{n,N} \right)\\
	\leq \sum_{k=0}^{\eps_N\ell_N-1} \frac{2\cdot 2^k}{h(\delta_2\hat a_{n,N})} < \frac{2^{1+\eps_N\ell_N}}{h(\delta_2\hat a_{n,N})} \leq \frac{8N^{\eps_N}}{\delta_2^\alpha\delta_1N^{\theta_{n,N}}\ell_N} \leq \frac{8}{\delta_2^\alpha\delta_1} N^{\eps_N - \delta}\numberthis\label{eq: rBigJump}
	\end{multline*}
	for $N$ sufficiently large, where in the third inequality we use that~$2^{\eps_N\ell_N}\le 2N^{\eps_N}$ for $N$ sufficiently large, and that~$h(\delta_2\hat a_{n,N})\ge \delta_2^\alpha\delta_1N^{\theta_{n,N}}\ell_N / 2$ for $N$ sufficiently large because of \eqref{eq: regvar} and \eqref{eq: hhataNasymp}, and in the fourth inequality we use that $\theta_{n,N} \geq \delta$ by~\eqref{eq: thetan}, \eqref{eq: Tn} and since we are assuming $T_n<t_1$.
		
Note that we have $\eps_N<\delta /2$ for $N$ sufficiently large by our assumptions in~\eqref{eq: epsN}.
Therefore, by~\eqref{eq: X1smallbig}, Lemma~\ref{lemma: leftmost}, \eqref{eq: noBigJumpX1big}, and \eqref{eq: rBigJump} we conclude that
\begin{align*}
\probT\left( \exists i\in\Nn_{N,T_n}(\Tneps), k\in\Nn_{i,n}:\: P_{i,\Tneps}^{k,t}\cap \Mm_{n,N} = \emptyset \right) < \frac{\eta}{100K}
\end{align*}
for $N$ sufficiently large.
\end{proof}

\subsection{The number of descendants of medium jumps: proof of Lemma~\ref{lemma: withBigJump}}\label{sect: withBigJump}

\begin{proof}[Proof of Lemma~\ref{lemma: withBigJump}]
	We partition the time interval $[\Tneps,t-1]$ into two subintervals, and look at the number of medium jumps and the number of time-$t$ descendants of the medium jumps. Let
	\begin{align*}
	&I_1 := [\Tneps, t_1 + 2\eps_N\ell_N - 1] \quad \text{ and }\quad I_2 := [t_1 + 2\eps_N\ell_N,t-1]
	\end{align*}
	be the two intervals, and let $A_j^i$ denote the set of particles in $\Nn_{i,n}$ which have a medium jump in their ancestral lines which happened in the time interval $I_j$:
	\begin{align}
	A_j^i := \bracing{k\in\Nn_{i,n}:\: P_{i,\Tneps}^{k,t}\cap \Mm_{n,N}^{I_j} \neq \emptyset}, \quad i\in\Nn_{N,T_n}(\Tneps), \quad j\in \{1,2\}.
	\end{align}
	If there is a medium jump in $I_1$, then there may be many, possibly of order $N$, particles at time~$t$ descending from this medium jump. However, we will see that with high probability there are no medium jumps at all in $I_1$: particle $(N,T_n)$ does not have enough descendants by the end of $I_1$ for any to have made a medium jump. In contrast, in the second interval there are many particles to make medium jumps (although not more than~$N$ at any one time), but there is less time to produce many descendants by time $t$. Indeed, for each $ i\in\Nn_{N,T_n}(\Tneps)$ the expected number of time-$t$ descendants of $(i,\Tneps)$ whose path has a medium jump in~$I_2$ is of order $N^{1-\eps_N}$. Using a concentration result from~\cite{McDiarmid1998}, we will see that the number of descendants itself (rather than the expected number) is of order $N^{1-\eps_N}$ with high probability, and therefore for each $i$, the total contribution of $A_1^i$ and $A_2^i$ is $o(N)$ with high probability. With the above strategy in mind, we give the following upper bound on the probability in the statement of Lemma~\ref{lemma: withBigJump}, using~\eqref{eq: Nin}:
	\begin{align*}
	& \probT\left( \exists i\in\Nn_{N,T_n}(\Tneps):\: D_{i,n} > \nu N \text{ and }P_{i,\Tneps}^{k,t}\cap \Mm_{n,N} \neq \emptyset \;\, \forall k\in\Nn_{i,n}\right)
	\\ & \hspace{10mm}
	\leq \probT\left( \exists i\in\Nn_{N,T_n}(\Tneps):\: \#\bracing{k\in\Nn_{i,n}:\: P_{i,\Tneps}^{k,t}\cap \Mm_{n,N} \neq \emptyset} > \nu N\right)
	\\ & \hspace{10mm}
	= \probT(\exists i\in\Nn_{N,T_n}(\Tneps):\: |A_1^i \cup A_2^i | > \nu N)
	\\ &\hspace{10mm}
	\leq \probT(\exists i\in\Nn_{N,T_n}(\Tneps):\: A_1^i \neq \emptyset) + \probT\left(\exists i\in\Nn_{N,T_n}(\Tneps):\: |A_2^i| > CN^{1-\eps_N}\right)\numberthis\label{eq: withBigJump1}
	\end{align*}
	for $N$ sufficiently large and any constant $C$, since $\eps_N \ell_N \to \infty$ as $N\to \infty$ by our choice of $\eps_N$ in~\eqref{eq: epsN}.
	
	We let $\tilde I_1 := [T_n,t_1 + 2\eps_N\ell_N -1] \supset I_1$. It is enough to bound the first term on the right-hand side of~\eqref{eq: withBigJump1} by the probability that any of the descendants of particle $(N,T_n)$ makes a medium jump by time $t_1 + 2\eps_N\ell_N -1$:
	\begin{equation}\label{eq: A1i1}
	\probT(\exists i\in\Nn_{N,T_n}(\Tneps):\: A_1^i \neq \emptyset) \leq \probT\left(\exists (j,b,s)\in\Mm_{n,N}^{\tilde I_1}:\: (N,T_n)\lesssim (j,s)\right).
	\end{equation}
	 This probability will be very small, as the total number of descendants of $(N,T_n)$ in the time interval $\tilde I_1$ is not large enough to see jumps of order $\hat a_{n,N}$. Indeed, applying a union bound over the jumps made by descendants of $(N,T_n)$ at times $T_n+k$  shows that the right-hand side of \eqref{eq: A1i1} is at most
	\begin{equation}\label{eq: A1i2}
	\sum_{k = 0}^{(\theta_{n,N} + 2\eps_N)\ell_N-1} \frac{2\cdot 2^k}{h(\delta_2\hat a_{n,N})} \le 2\cdot 2^{(\theta_{n,N}+2\eps_N)\ell_N } \frac{2}{\delta_2^\alpha\delta_1 N^{\theta_{n,N}} \ell_N} \le \frac{8}{\delta_2^\alpha\delta_1} \ell_N^{-1/2}
	\end{equation}
	for $N$ sufficiently large, where in the first inequality we use that $h(\delta_2 \hat a_{n,N}) \ge \delta_2^\alpha\delta_1N^{\theta_{n,N}}\ell_N / 2$ for $N$ sufficiently large by \eqref{eq: regvar} and \eqref{eq: hhataNasymp}, and in the second inequality we use the assumption on $\eps_N$ in~\eqref{eq: epsN}, and that $2^{\theta_{n,N}\ell_N}\le 2N^{\theta_{n,N}}$.
	
	For the second term on the right-hand side of \eqref{eq: withBigJump1} we will give an upper bound using the concentration inequality from \cite[Theorem 2.3(b)]{McDiarmid1998}. First we bound $|A_2^i|$ for any $i\in\Nn_{N,T_n}(\Tneps)$: 
	\begin{align}\label{eq: withBigJump2}
	|A_2^i| \leq \sum_{k = (\theta_{n,N} + 2\eps_N)\ell_N\:\:}^{(1 + \theta_{n,N})\ell_N-1}  \sum_{\:\:j\in \Nn_{i,\Tneps}(T_n+k), b\in\bracing{1,2}} \mathds{1}_{\bracing{X_{j,b,T_n+k} > \delta_2\hat a_{n,N}}} |\Nn_{j,T_n+k}^b(t)|,
	\end{align}
	where we sum up the number of time-$t$ descendants of every particle descended from $(i,\Tneps)$ which made a jump of size greater than~$\delta_2 \hat a_{n,N}$ at a time $T_n+k$ in the time interval $I_2$. Now let $\xi^i_{j,k}\sim$ Ber$(h(\delta_2\hat a_{n,N})^{-1})$ be i.i.d.~random variables, by which we mean that 
	\begin{align*}
	\probT(\xi^i_{j,k} = 1) = h(\delta_2 \hat a_{n,N})^{-1} = 1 - \probT(\xi^i_{j,k} = 0) ,
	\end{align*}
	for all $i,j,k\in\N$. The indicator random variables in \eqref{eq: withBigJump2} all have this distribution. Considering that we have $|\Nn_{i,\Tneps}(T_n+k)| \leq \min(N, 2^{k - \eps_N\ell_N})$ and $|\Nn_{j,T_n+k}^b(t)| \leq 2^{(1 + \theta_{n,N})\ell_N - k -1}$ for all $k\in \llbracket (\theta_{n,N} + 2\eps_N)\ell_N, (1 + \theta_{n,N})\ell_N -1\rrbracket$, and since $\Nn_{N,T_n}(\Tneps)\leq 2^{\eps_N\ell_N}$, we obtain the following upper bound from \eqref{eq: withBigJump2}:
	\begin{align*}
	&\probT\left(\exists i\in\Nn_{N,T_n}(\Tneps):\: |A_2^i| > CN^{1-\eps_N}\right) 
	\\ &\hspace{5mm}
	\leq 
	\probT\Bigg( \exists i\in [2^{\eps_N \ell_N}]: \: \sum_{k = (\theta_{n,N} + 2\eps_N)\ell_N}^{(1 + \theta_{n,N})\ell_N}  2^{(1 + \theta_{n,N})\ell_N - k} \sum_{j = 1}^{2\min(N, 2^{k - \eps_N\ell_N})} \xi^i_{j,k} > CN^{1-\eps_N}\Bigg)
	\\ & \hspace{5mm}
	\leq
	2^{\eps_N\ell_N}\probT\Bigg(\sum_{k = (\theta_{n,N} + 2\eps_N)\ell_N}^{(1 + \theta_{n,N})\ell_N}  2^{(1 + \theta_{n,N})\ell_N - k} \sum_{j = 1}^{2\min(N, 2^{k - \eps_N\ell_N})} \xi^1_{j,k} > CN^{1-\eps_N}\: \Bigg) \numberthis\label{eq: withBigJump3}
	\end{align*}
	by a union bound.
	
	Now \cite[Theorem 2.3(b)]{McDiarmid1998} applies for independent random variables taking values in $[0,1]$, so we consider the random variables $2^{(2\eps_N + \theta_{n,N})\ell_N - k}\xi^1_{j,k}\in[0,1]$ for each $k$ and $j$ in the sum. Let $\mu$ denote the expectation of the sum of these random variables over $k$ and~$j$:
	\begin{align*}
	\mu & := \expectT\Bigg[\sum_{k = (\theta_{n,N} + 2\eps_N)\ell_N}^{(1 + \theta_{n,N})\ell_N} 2^{(2\eps_N + \theta_{n,N})\ell_N - k} \sum_{j = 1}^{2\min(N, 2^{k - \eps_N\ell_N})} \xi^1_{j,k} \Bigg] 
	\\ 
	& =\:  \expectT\Bigg[\sum_{k = (\theta_{n,N} + 2\eps_N)\ell_N}^{(1 + \eps_N)\ell_N-1} 2^{(2\eps_N + \theta_{n,N})\ell_N - k}\frac{2^{k - \eps_N\ell_N + 1}}{h(\delta_2 \hat a_{n,N})}\Bigg] 
	+ \expectT\Bigg[\sum_{k = (1 + \eps_N)\ell_N}^{(1 + \theta_{n,N})\ell_N} 2^{(2\eps_N + \theta_{n,N})\ell_N - k} \frac{2N}{h(\delta_2\hat a_{n,N})}\Bigg].\numberthis\label{eq: withBigJump4}
	\end{align*}
	Now considering that for $N$ sufficiently large,
	$\delta_2^\alpha\delta_1N^{\theta_{n,N}}\ell_N / 2 \leq h(\delta_2 \hat a_{n,N}) \leq 2\delta_2^\alpha\delta_1N^{\theta_{n,N}}\ell_N$ by \eqref{eq: regvar} and \eqref{eq: hhataNasymp}, that $N^{2\eps_N + \theta_{n,N}}\leq 2^{(2\eps_N + \theta_{n,N})\ell_N} \leq 4N^{2\eps_N + \theta_{n,N}}$, that $ \delta \leq \theta_{n,N} \leq 1-\delta$ and that $\eps_N<\delta/4$ for $N$ sufficiently large,	it can be seen that we have
	\begin{align}\label{eq: mu}
	K_1 N^{\eps_N} \leq \mu \leq K_2 N^{\eps_N},
	\end{align}
	for some constants $K_1,K_2>0$. Then, if we multiply both sides of the sum in \eqref{eq: withBigJump3} by $2^{(2\eps_N - 1)\ell_N}$ and use that $2^{(2\eps_N - 1)\ell_N} \ge N^{2\eps_N-1}/2$, we get
	\begin{multline*}
	\probT\left(\exists i\in\Nn_{N,T_n}(\Tneps):\: |A_2^i| > CN^{1-\eps_N}\right) 
	\\ 
	\leq 
	2^{\eps_N\ell_N}\probT\Bigg(\sum_{k = (\theta_{n,N} + 2\eps_N)\ell_N}^{(1 + \theta_{n,N})\ell_N} 2^{(2\eps_N+ \theta_{n,N})\ell_N - k} \sum_{j = 1}^{2\min(N, 2^{k - \eps_N\ell_N})} \xi^1_{j,k} > \frac{1}{2}CN^{\eps_N} \Bigg).
	\end{multline*}
	By \eqref{eq: mu} we have $\mu \geq K_1 N^{\eps_N}$, and
	we can choose $C>3K_2$ so that $\frac{1}{2}CN^{\eps_N} \geq \frac{3}{2}\mu$ for $N$ sufficiently large. Then by \cite[Theorem 2.3(b)]{McDiarmid1998} we have for $N$ sufficiently large,
	\begin{equation*}
	\probT\left(\exists i\in\Nn_{N,T_n}(\Tneps):\: |A_2^i| > CN^{1-\eps_N}\right) 
	\leq 2N^{\eps_N} \exp\left( - \frac{\frac{1}{4}K_1 N^{\eps_N}}{2(1 + \frac{1}{6})} \right),\numberthis\label{eq: A2i}
	\end{equation*}
	which is small if $N$ is large, by our choice of $\eps_N$ in \eqref{eq: epsN}. Then by \eqref{eq: withBigJump1}, \eqref{eq: A1i1}, \eqref{eq: A1i2} and \eqref{eq: A2i} we conclude Lemma~\ref{lemma: withBigJump}.
\end{proof}

\section{Proofs of Propositions \ref{prop: A2prime} and \ref{prop: diameter}}\label{sect: significance}

In Proposition~\ref{prop: A2prime} we need to prove that for any interval of the form $[t_2+\ceil{s_1 \ell_N},t_2+\ceil{s_2 \ell_N}]$ with $0<s_1<s_2<1$, the probability that the time of the common ancestor $T$ is in this interval is bounded away from 0 for large $N$. The main idea of the proof is that if there is a big jump in the time interval $[t_2+\ceil{s_1 \ell_N},t_2+\ceil{s_2 \ell_N}]$ which is much larger than any other jump in the time interval $[t_3,t_1]$, then that big jump will break the record, and we will have $T\in[t_2+\ceil{s_1 \ell_N},t_2+\ceil{s_2 \ell_N}]$.

More precisely, let $r>0$ be as in Proposition~\ref{prop: A2prime}. We will ask that a particle performs a jump larger than $(r+3)a_N$ at some time $s^*\in[t_2+\ceil{s_1\ell_N},t_2+\ceil{s_2\ell_N})$, and all the other jumps in the time interval $[t_3,t_1]$ are smaller than $a_N$. We will show that this happens with a probability bounded below by a positive constant (independent of $N$).

Suppose the above event occurs, and also the events $\Cc_3$ and $\Cc_4$ occur. Then we will also see that $d(\X(s^*))\leq (1+c_1)a_N$. This will imply that the particle which makes the jump larger than $(r+3)a_N$ at time $s^*$ breaks the record, and it will lead by more than roughly $(r+2)a_N$ at time $s^*+1$. As a result, the tribe of this particle will lead between times $s^*+1$ and $t_1$, because we assumed that all jumps in $[s^*+1,t_1)$ are smaller than $a_N$. Moreover, particles not in the leading tribe cannot get closer than $ra_N$ to the leading tribe by time $t_1$; therefore, we will conclude $d(\X(t_1)) \geq ra_N$ as well. 

The following lemma will be useful for proving the above statements.
\begin{lemma}\label{lemma: diam_raN}
Take $\rho,c_1>0$.
Then for $N\ge 2$ and $t>4\ell_N$, for all $s_0\in[t_4,t_1]$ and $r_0>0$,
on the event $\Cc_3\cap\Cc_4$, 
\begin{align*}
\bracing{X_{i,b,s}\leq r_0a_N\; \forall (i,b,s)\in\Nonetwo\times\bbracket{s_0,s_0+\ell_N-1}} \subseteq \bracing{d(\X(s_0+\ell_N))\leq (r_0+c_1)a_N},
\end{align*}
where the events $\Cc_3$ and $\Cc_4$ are defined in \eqref{eq: C3} and \eqref{eq: C4} respectively.
\end{lemma}
\begin{proof}
Let $\Gg_1$ denote the event on the left-hand side in the statement of the lemma:
\begin{align*}
\Gg_1 := \bracing{X_{i,b,s}\leq r_0a_N\; \forall (i,b,s)\in\Nonetwo\times\bbracket{s_0,s_0+\ell_N-1}}.
\end{align*}
Let $j\in [N]$ be arbitrary, and let $i = \zeta_{j,s_0+\ell_N}(s_0)$. Then, on the event $\Cc_3$, we have $|B_N\cap P_{i,s_0}^{j,s_0+\ell_N}|\leq 1$, and on the event $\Cc_4$, no particle moves further than $c_1 a_N$ once big jumps have been removed from its path. Thus, on the event $\Cc_3\cap\Cc_4\cap\Gg_1$,
\begin{align*}
\X_j(s_0+\ell_N) \leq \X_i(s_0) + c_1a_N + \sum_{(i',b',s')\in B_N \cap P_{i,s_0}^{j,s_0+\ell_N}} X_{i',b',s'} \leq \XN(s_0) + (r_0+c_1)a_N.
\end{align*}
But by Lemma~\ref{lemma: descendants_est}, we have $\XN(s_0) \leq  \Xone(s_0+\ell_N)$, and the result follows.
\end{proof}

\begin{proof}[Proof of Proposition~\ref{prop: A2prime}]
Recall the definition of $\Aa_2'$ from \eqref{eq: A2prime}, and consider a uniform sample of $M$ particles at time $t$ with indices $\Pp_1,\dots,\Pp_M$. Also recall the definitions of $T(\rho)$ in \eqref{eq: T} and $\Teps(\rho)$ in~\eqref{eq: Teps}. For any $\rho>0$ we have
\begin{align*}
&\bracing{T(\rho)\in[t_2+\ceil{s_1 \ell_N},t_2+\ceil{s_2 \ell_N}]}\cap\bracing{\zeta_{\Pp_{j},t}(T(\rho)) = N\: \forall j\in[M]}
\\ & \quad \qquad \qquad \qquad 
\cap
\bracing{
	\zeta_{\Pp_{j},t}(\Teps(\rho)) \neq \zeta_{\Pp_{l},t}(\Teps(\rho))\: \forall j,l\in[M],\: j\neq l}
\subseteq \Aa_2'.\numberthis\label{eq: inA2prime}
\end{align*}
For $r>0$, we define $\Aa_3'$ as a modification of the event $\Aa_3$ from \eqref{eq: A3}:
\begin{align*}
\Aa_3' = \Aa_3'(t,N,\rho,\gamma,r,s_1,s_2) & := \bracing{T(\rho)\in[t_2+\ceil{s_1 \ell_N},t_2+\ceil{s_2 \ell_N}]}
\\ & \quad \qquad
\cap\bracing{|\Nn_{N,T(\rho)}(t)| \geq N - N^{1-\gamma}}\cap\bracing{d(\X(t_1))\geq ra_N}.\numberthis\label{eq: A3prime}
\end{align*}
We also define the set of jumps in the time interval $[t_2+\ceil{s_1 \ell_N},t_2+\ceil{s_2 \ell_N})$ which are larger than $(r+3)a_N$:
\begin{align}
B_N'(t,r,s_1,s_2) := \bracing{ \begin{array}{l}
(i,b,s)\in\Nonetwo\times\bbracket{t_2+\ceil{s_1 \ell_N},t_2+\ceil{s_2 \ell_N}-1}:\\
X_{i,b,s} > (r+3)a_N
\end{array}
},
\end{align}
and the event $\Gg$, which says that there is only one jump in the set $B_N'$, and every other jump is smaller than $a_N$ during the time interval $[t_3,t_1-1]$:
\begin{align}
\Gg = \Gg(t,N,r,s_1,s_2) := \bracing{
\begin{array}{l}
|B_N'| = 1 \; \text{ and } \; X_{i,b,s}\leq a_N,
\\
\forall (i,b,s)\in(\Nonetwo\times[t_3,t_1-1])\setminus B_N'
\end{array}	}.
\end{align}

Fix $0 < s_1 < s_2 < 1$, $M\in\N$ and $r>0$. Choose $\pi_{r,s_2-s_1}>0$ such that
\begin{align}\label{eq: etaprime1}
\pi_{r,s_2-s_1} < \frac{s_2-s_1}{8(r+3)^\alpha}\cdot e^{-8},
\end{align}
and then $\eta>0$ sufficiently small that it satisfies \eqref{eq: eta} and
\begin{align}\label{eq: etaprime2}
5\eta < \frac{s_2-s_1}{8(r+3)^\alpha}\cdot e^{-8} - \pi_{r,s_2-s_1}.
\end{align}
Then choose the constants ${\gamma,\delta,\rho, c_1 ,c_2,\dots,c_6,K}$ such that they satisfy~\eqref{c5}-\eqref{K}.
Recall from Section~\ref{sect: probCD}
that this implies the properties in \eqref{eq: const1}-\eqref{eq: const5} and \eqref{eq: eta1000}-\eqref{eq: rho} also hold for $\eta$ and $\gamma,\delta,\rho, c_1 ,$ $c_2,\dots,c_6,K$. 
Let $0<\nu < \eta / M^2$. 

In the course of the proof we will use the events $\Aa_3$ and $\Aa_4$ from \eqref{eq: A3} and \eqref{eq: A4}, and we will show the following for $N$ sufficiently large and $t>4\ell_N$:
\begin{enumerate}
\item\label{step1} 
$
\prob\left((\Aa_2')^c \cup \bracing{d(\X(t_1))< ra_N}\right) \leq \prob((\Aa_3')^c) + \prob(\Aa_3^c) + \prob(\Aa_4(\nu)^c) + \eta
$
\item\label{step2}
$
\bigcap_{j=2}^7 \Cc_j \cap \bigcap_{i=1}^5 \Dd_i \cap \Gg \subseteq \Aa_3'
$
\item\label{step3}
$
\prob(\Gg) \ge \frac{s_2-s_1}{8(r+3)^\alpha}\cdot e^{-8}
$
\item\label{step4}
$
\prob\left((\Aa_2')^c \cup \bracing{d(\X(t_1))< ra_N}\right) \leq 1 - \pi_{r,s_2-s_1}.
$
\end{enumerate}
We start by proving step~\ref{step1}. Notice that with our choices of constants, the conditions of Lemma~\ref{lemma: rewrite} hold. Therefore, we know
\begin{equation}\label{eq: A4c2}
\prob(\exists j,l\in[M],\: j\neq l:\:\zeta_{\Pp_{j},t}(\Teps) = \zeta_{\Pp_{l},t}(\Teps)) \leq \prob(\Aa_3^c) + \prob(\Aa_4(\nu)^c) + \eta/2,
\end{equation}
for $N$ sufficiently large. Hence, because of \eqref{eq: inA2prime}, in order to prove step~\ref{step1} it remains to show that
\begin{align*}
&\prob\left(\bracing{T(\rho)\notin[t_2+\ceil{s_1 \ell_N},t_2+\ceil{s_2 \ell_N}]} \cup \bracing{\exists j\in[M]:\:\zeta_{\Pp_{j},t}(T) \neq N}\cup\bracing{d(\X(t_1)) < ra_N} \right)
\\ &
\leq \prob((\Aa_3')^c) + \eta/2,\numberthis\label{eq: A3cprime}
\end{align*}
for $N$ sufficiently large. This follows similarly to the proof of \eqref{eq: A3c}. Partitioning the event on the left-hand side of \eqref{eq: A3cprime} using the event $\Aa_3'$, and then conditioning on $\mathcal F_t$, we obtain
\begin{multline*}
\prob\left(\bracing{T(\rho)\notin[t_2+\ceil{s_1 \ell_N},t_2+\ceil{s_2 \ell_N}]} \cup \bracing{\exists j\in[M]:\:\zeta_{\Pp_{j},t}(T) \neq N}\cup\bracing{d(\X(t_1)) < ra_N} \right) 
\\
\leq\: 
\expect\left[ \1_{\Aa_3'} \prob(\exists j\in[M]:\:\zeta_{\Pp_{j},t}(T) \neq N\: |\: \F_t) \right]
+ \prob\left((\Aa_3')^c\right)\numberthis\label{eq: A3cprimeM}
\end{multline*}
where we use that if $\Aa_3'$ occurs then $T(\rho)\in[t_2+\ceil{s_1 \ell_N},t_2+\ceil{s_2 \ell_N}]$ and $d(\X(t_1))\geq ra_N$, and that $\Aa_3'$ is $\F_t$-measurable. Now, on the event $\Aa_3'$, at most $N^{1-\gamma}$ time-$t$ particles are not descended from $(N,T)$, and therefore a union bound on the uniformly chosen sample (which is not $\F_t$-measurable) shows that the right-hand side of~\eqref{eq: A3cprimeM} is at most $M N^{1-\gamma}/N +  \prob\left((\Aa_3')^c\right)$. This implies \eqref{eq: A3cprime} for~$N$ sufficiently large, and by \eqref{eq: A4c2} and \eqref{eq: A3cprime} we are done with step~\ref{step1}.

We next prove step~\ref{step2}. Assume the event $\bigcap_{j=2}^7 \Cc_j \cap \Gg$ occurs. Then there exists $(i^*,b^*,s^*)\in B_N'$ with $s^*\in\bbracket{t_2+\ceil{s_1 \ell_N},t_2+\ceil{s_2 \ell_N}-1}$. We notice that every jump in the time interval $[t_3,s^*-1]$ has size at most $a_N$ on the event $\Gg$. Thus, we can apply Lemma~\ref{lemma: diam_raN} with $s_0 = s^* - \ell_N > t_3$, $\rho$ and $c_1$ as chosen at the beginning of the proof, and with $r_0=1$. We then obtain
\begin{align}\label{eq: diamsstar}
d(\X(s^*)) \leq (1+c_1)a_N.
\end{align}
This means that a particle that makes a jump larger than $(r+3)a_N$ at time $s^*$ must take the lead at time $s^*+1$. Indeed, 
\begin{align}
 \X_{i^*}(s^*) + X_{i^*,b^*,s^*} > \Xone(s^*) + (r+3)a_N \geq \XN(s^*) + (r+2-c_1)a_N,
\end{align}
where in the first inequality we use that $\X_{i^*}(s^*)\geq \Xone(s^*)$ and that $(i^*,b^*,s^*)\in B_N'$, and the second inequality follows by \eqref{eq: diamsstar}. Note that our choice of constants means that $\rho < r+2-c_1 < r+3$ holds (see e.g.~\eqref{eq: eta1000} and~\eqref{eq: rho}); thus we have $B_N'\subseteq B_N$, and Lemma~\ref{lemma: breakRecordGap}(b) applies. Therefore, by Lemma~\ref{lemma: breakRecordGap}(b), we have $(i^*,s^*)\lesssim_{b^*}(N,s^*+1)$ and 
\begin{align}\label{eq: XNsstar}
\X_{i^*}(s^*) + X_{i^*,b^*,s^*} = \XN(s^*+1) >\X_{N-1}(s^*+1) + (r+2-c_1-\rho)a_N,
\end{align}
which also shows that $s^*\in \Sbf_N(\rho)$, where $\Sbf_N(\rho)$ is the set of times when the record is broken by a big jump (see \eqref{eq: RN}).

Now we prove that $s^*+1 = T(\rho)$ and $d(\X(t_1)) \geq ra_N$. Let $\hat s\in\bbracket{ s^*+1,t_1-1}$ be arbitrary (and note that $\bbracket{ s^*+1,t_1-1}$ is not empty for $N$ sufficiently large). We will see that $\hat s\notin \Sbf_N(\rho)$, and therefore $T(\rho)\notin\bbracket{s^*+2,t_1}$, i.e. $T(\rho) = s^*+1$. 

Take $k\in [N-1]$, and assume that $j\in\Nn_{k,s^*+1}(\hat s+1)$. Note that $|B_N \cap P_{k,s^*+1}^{j,\hat s+1}| \leq 1$ by the definition of the event $\Cc_3$, and that every jump in the time interval $[s^*+1,t_1-1]$ is at most of size $a_N$ by the definition of the event $\Gg$. Hence, by the definition of the event $\Cc_4$ we have
\begin{align*}
\X_j(\hat s + 1) & \leq \X_k(s^*+1) + c_1a_N + \sum_{(i,b,s)\in B_N \cap P_{k,s^*+1}^{j,\hat s+1}} X_{i,b,s} 
\\ &
\leq \X_{N-1}(s^*+1) + (c_1+1)a_N
\\ &
< \XN(s^*+1) - (r+1-2c_1-\rho)a_N
\\ &
\leq \XN(\hat s + 1) - (r+1-2c_1-\rho)a_N, \numberthis\label{eq: XjXN}
\end{align*}
where in the second inequality we also use that $k\leq N-1$, the third inequality follows by \eqref{eq: XNsstar}, and the fourth by monotonicity.

Then \eqref{eq: XjXN} has two consequences. First, it shows that $\X_j(\hat s + 1) < \XN(\hat s + 1)$ (see e.g.~\eqref{eq: eta1000} and~\eqref{eq: rho}); thus the leader at time $\hat s+1$ must descend from particle $(N,s^*+1)$; that is, $\zeta_{N,\hat s + 1}(s^*+1) = N$. Note that we also have $X_{i,b,\hat s}\leq \rho a_N$ for all $i\in\Nn_{N,s^*+1}(\hat s)$ and $b\in\bracing{1,2}$ by the definition of the event $\Cc_3$. We conclude that the record is not broken by a big jump at time $\hat s +1$, which means that $\hat s \notin \Sbf_N(\rho)$. Since $\hat s\in\bbracket{ s^*+1,t_1-1}$ was arbitrary, and $s^*\in \Sbf_N(\rho)$, we must have $T(\rho) = s^*+1$, by the definition \eqref{eq: T} of $T(\rho)$. Hence,
\begin{align}\label{eq: GT}
\bigcap_{i=2}^7 \Cc_i \cap \Gg \subseteq \bracing{T(\rho)\in[t_2+\ceil{s_1 \ell_N},t_2+\ceil{s_2 \ell_N}]}.
\end{align}
The second consequence of \eqref{eq: XjXN} is that $d(\X(\hat s+1))> ra_N$, since $2c_1+\rho<1$. Indeed, we notice that since $s^*+1>t_2$ and $\hat s+1 \leq t_1$,
the number of descendants of particle $(N,s^*+1)$ is strictly less than $N$ at time $\hat s + 1$. Thus, there exists $k\in[N-1]$ such that $\Nn_{k,s^*+1}(\hat s + 1)\neq\emptyset$, and for such a $k$ and for some $j\in\Nn_{k,s^*+1}(\hat s + 1)$ the bound in \eqref{eq: XjXN} holds, and shows that $d(\X(\hat s+1))> ra_N$. Since $\hat s\in\bbracket{ s^*+1,t_1-1}$ was arbitrary we conclude
\begin{align}\label{eq: Gd}
\bigcap_{i=2}^7 \Cc_i \cap \Gg \subseteq \bracing{d(\X(t_1))\geq ra_N}.
\end{align}

As Propositions~\ref{prop: C} and \ref{prop: B} (and the definition of $\mathcal A_3$ in~\eqref{eq: A3}) imply for $N$ sufficiently large that
\begin{align*}
\bigcap_{j=2}^7 \Cc_j \cap \bigcap_{i=1}^5 \Dd_i \cap \Gg \subseteq \bigcap_{i=1}^7 \Cc_i \cap \Gg \subseteq \Aa_3 \subseteq \bracing{|\Nn_{N,T(\rho)}(t)| \geq N - N^{1-\gamma}},
\end{align*}
step~\ref{step2} follows by \eqref{eq: GT} and \eqref{eq: Gd}.

For step~\ref{step3}, the event $\Gg$ says that out of the $4N\ell_N$ jumps occurring in the time interval $[t_3,t_1-1]$, there are $4N\ell_N-1$ jumps of size at most $a_N$, and there is one larger than $(r+3)a_N$, which can happen any time during the time interval $[t_2+\ceil{s_1 \ell_N},t_2+\ceil{s_2 \ell_N})$. Using that $\ceil{s_2 \ell_N}-1 - \ceil{s_1 \ell_N} \geq (s_2-s_1)\ell_N/2$ for large $N$, we have for $N$ sufficiently large,
\begin{align*}
\prob(\Gg) & \geq 2N\frac{(s_2-s_1)}{2}\ell_N\cdot h((r+3)a_N)^{-1}\left( 1 - h(a_N)^{-1} \right)^{4N\ell_N - 1}
\\ &
\geq \frac{(s_2-s_1)}{2}\frac{h(a_N)}{h((r+3)a_N)}\cdot\frac{2N\ell_N}{h(a_N)}\cdot e^{-2\frac{4N\ell_N}{h(a_N)}}
\\ &
\geq \frac{s_2-s_1}{8(r+3)^\alpha}\cdot e^{-8},
\end{align*}
where the second inequality holds if $N$ is sufficiently large that $1-h(a_N)^{-1}>e^{-2h(a_N)^{-1}}$, which is possible because $h(a_N)\to\infty$ as $N\to\infty$ by \eqref{eq: haNasymp}. In the third inequality we use that
$h(a_N)/h((r+3)a_N)\geq (r+3)^{-\alpha} / 2$ for $N$ large enough by \eqref{eq: regvar} and~\eqref{eq: aNhaN}, and that $1/2 \leq 2N\ell_N / h(a_N) \leq 2$ for $N$ large enough by \eqref{eq: haN}. This completes step 3.

For step~\ref{step4}, we note that we chose the constants $\eta$, $\gamma$, $\delta$, $\rho$, $c_1,c_2,\dots,c_6$, $K$ and $\nu$ in such a way that the probability bounds in Propositions~\ref{prop: A1A3} and \ref{prop: A4} and Lemma~\ref{lemma: probC} hold for $N$ sufficiently large and $t>4\ell_N$. Hence, putting steps~\ref{step1} to \ref{step3} together we conclude
\begin{align*}
\prob\left((\Aa_2')^c \cup \bracing{d(\X(t_1))< ra_N}\right) & \leq \sum_{j=2}^{7}\prob(\Cc_j^c) + \sum_{i=1}^{5}\prob(\Dd_i^c) + \prob(\Gg^c) + \prob(\Aa_3^c) + \prob(\Aa_4(\nu)^c) + \eta
\\ & 
\leq 1 - \frac{s_2-s_1}{8(r+3)^\alpha}\cdot e^{-8} + 5\eta
\\ & 
< 1 - \pi_{r,s_2-s_1},
\end{align*}
where in the last inequality we used \eqref{eq: etaprime2}. This finishes the proof of Proposition~\ref{prop: A2prime}.
\end{proof}

The proof of Proposition~\ref{prop: diameter} involves some of our previous results. We will use the statement of Proposition~\ref{prop: A2prime} about the diameter to prove that for any fixed $r>0$, $\prob\left(d(\X(n)) \geq ra_N\right)$ can be lower bounded by a positive constant. Then the statement of Proposition~\ref{prop: B} about the diameter shows that on the events $\Cc_1$ to $\Cc_7$ the diameter at time $t_1$ is greater than $c_3a_N$, so, considering Lemma~\ref{lemma: probC}, we will see that the diameter is at least of order $a_N$ at a typical time with high probability. Finally, we will conclude that the diameter is at most of order $a_N$ with high probability using Lemma~\ref{lemma: diam_raN}, and also using that jumps of size $ra_N$ are unlikely to happen in $\ell_N$ time if $r$ is very large.

\begin{proof}[Proof of Proposition~\ref{prop: diameter}]
Take ${\eta,\gamma,\delta,\rho, c_1 ,c_2,\dots,c_6,K}$ such that they satisfy \eqref{eq: eta}, \eqref{c5}-\eqref{K}, and therefore also \eqref{eq: const1}-\eqref{eq: const5} and \eqref{eq: eta1000}-\eqref{eq: rho} (and $\eta$ may be arbitrarily small). 
Let $r>0$ be arbitrary. 
Let $s_1 = 1/4$, $s_2 = 1/2$, $M=3$. Then we take $\pi_{r,s_2-s_1}>0$ and $N\in\N$ sufficiently large that the bounds in Proposition~\ref{prop: A2prime} and Lemma~\ref{lemma: probC} and the inclusions in Propositions~\ref{prop: B}~and~\ref{prop: C} and in Lemma~\ref{lemma: diam_raN} hold with the above constants and for all $t>4\ell_N$. Furthermore, we assume that $N$ is sufficiently large that 
\begin{equation}\label{eq: h1}
e^{-2h(ra_N/2)^{-1}} < 1 - h(ra_N/2)^{-1},
\end{equation}
\begin{equation}\label{eq: h2}
\frac{h(a_N)}{h(ra_N /2)} \leq 2(r/2)^{-\alpha},
\end{equation}
and
\begin{equation}\label{eq: h3}
\frac{2N\ell_N}{h(a_N)} \leq 2.
\end{equation}
We can take $N$ sufficiently large that \eqref{eq: h1}, \eqref{eq: h2} and \eqref{eq: h3} hold because of~\eqref{eq: haNasymp}, \eqref{eq: aNhaN} (i.e. $a_N\to\infty$ as $N\to\infty$), \eqref{eq: regvar} and~\eqref{eq: haN}. Having fixed $N$ with these properties, take $n>3\ell_N$. 

First we apply Proposition~\ref{prop: A2prime} in the above setting with $t=n+\ell_N$ (and $t_1=n$). The proposition implies that
\begin{equation}\label{eq: pir}
0< \pi_{r,s_2-s_1} < \prob\left(d(\X(n)) \geq ra_N\right).
\end{equation}
Now we prove that if $r$ is sufficiently small then we have
\begin{align}\label{eq: diam_rtozero}
\prob\left(d(\X(n)) < ra_N\right) < \eta.
\end{align}
Assume that $r<c_3$, where $c_3$ was specified at the beginning of this proof.

Consider the events $(\Cc_j)_{j=2}^7$ and $(\Dd_i)_{i=1}^5$ with the constants ${\gamma,\delta,\rho, c_1 ,c_2,\dots,c_6,K}$ and with $t=n+\ell_N$. By Propositions~\ref{prop: C}~and~\ref{prop: B} we have 
\begin{equation*}
\bigcap_{j=2}^7 \Cc_j \cap \bigcap_{i=1}^5 \Dd_i \subseteq \bigcap_{j=1}^7 \Cc_j \subseteq \bracing{d(\X(n)) \geq \tfrac{3}{2}c_3a_N}.
\end{equation*}
Therefore, since $r<c_3$, and then by Lemma~\ref{lemma: probC}, we have
\begin{equation*}
\prob(d(\X(n)) < ra_N) \leq \prob(d(\X(n)) < \tfrac{3}{2}c_3a_N) \leq \sum_{j=2}^7\prob(\Cc_j^c) + \sum_{i=1}^5\prob(\Dd_i^c) < \eta,
\end{equation*}
which establishes \eqref{eq: diam_rtozero}.

Next we prove that if $r$ is sufficiently large then 
\begin{align}\label{eq: diam_rtoinfty}
\prob\left(d(\X(n)) \geq ra_N\right) < \eta.
\end{align} 
Assume $r>1$. We apply Lemma~\ref{lemma: diam_raN} with $t=n+\ell_N$, $s_0 = n-\ell_N$ and $r_0 = r/2$. Note that by~\eqref{eq: eta1000} and \eqref{eq: cjeta} we have $r_0 + c_1 < r$. Then Lemma~\ref{lemma: diam_raN} implies
\begin{align*}
\prob(d(\X(n)) \geq ra_N) & \leq \prob(\exists (i,b,s)\in\Nonetwo\times\bbracket{n-\ell_N,n-1}:\: X_{i,b,s}> \tfrac{r}{2}a_N)
\\ & 
= 1 - (1-h(ra_N/2)^{-1})^{2N\ell_N}
\\ &
\leq 1 - \exp\left(-2\frac{2N\ell_N}{h(ra_N/2)}\right)
\\ &
= 1 - \exp\left(-2\frac{2N\ell_N}{h(a_N)}\frac{h(a_N)}{h(ra_N/2)}\right)
\\ & 
\leq 1 - \exp\left(-8(r/2)^{-\alpha}\right),\numberthis\label{eq: diam_exp}
\end{align*}
where in the equality we use the tail distribution \eqref{eq: poly_tail} for the $2N\ell_N$ jumps in the time interval $\bbracket{n-\ell_N,n-1}$, the second inequality holds by \eqref{eq: h1}, and in the third we use \eqref{eq: h2} and \eqref{eq: h3}. Then~\eqref{eq: diam_exp} shows that~\eqref{eq: diam_rtoinfty} holds for $r$ sufficiently large.

Since $\eta>0$ was arbitrarily small, \eqref{eq: pir} and \eqref{eq: diam_rtozero} show the existence of $p_r$ and \eqref{eq: diam_rtoinfty} proves the existence of $q_r$ as in the statement of Proposition~\ref{prop: diameter}, and therefore we have finished the proof of this result.
\end{proof}

\section{Glossary of notation}\label{sect: glossary}
Below we list the most frequently used notation of this paper. In the second column of the table we give a brief description, and in the third column we refer to the section or equation where the notation is defined or first appears.

\noindent
\begin{longtable}{ p{3cm} p{10cm} r }
\midrule[0.08em]
\textbf{Notation} & \textbf{Meaning} & \textbf{Def./Sect.} 
\\
\midrule[0.08em] 
$N$ & number of particles & Sect.~\ref{sect: NBRW} \\  
$(i,n)$ & refers to the $i$th particle from the left at time $n$ & Sect.~\ref{sect: NBRW}
\\
$\X_i(n)$ & location of the $i$th particle from the left at time $n$ & Sect.~\ref{sect: NBRW}
\\
$h$ & the function $1/h$ defines the tail of the jump distribution & \eqref{eq: poly_tail}
\\
$\alpha$ & $h$ is regularly varying with index $\alpha>0$ & \eqref{eq: regvar}, \eqref{eq: poly_tail}
\\
$\ell_N$ & time scale: $\ell_N = \ceil{\log_2 N}$ & \eqref{eq: LN}
\\
$a_N$ & space scale: $a_N = h^{-1}(2N\ell_N)$, $h(a_N) \sim 2N\ell_N$ & \eqref{eq: aN}
\\
$t$ & $t\in\N$ is an arbitrary time, we assume $t>4\ell_N$ & Sect.~\ref{sect: intro_result}
\\
$t_i$ & $t_i=t-i\ell_N$, we use $t_1,t_2,t_3,t_4$ & \eqref{eq: ti}
\\
$X_{i,b,n}$ & jump size of the $b$th offspring of particle $(i,n)$ & Sect.~\ref{sect: NBRWdef}
\\
$(i,b,n)$ & refers to the jump $X_{i,b,n}$ of the $b$th offspring of particle $(i,n)$ & Sect.~\ref{sect: notation}
\\
$d(\X(n))$ & diameter of the particle cloud at time $n$ & \eqref{eq:diameterdefn}
\\
$(i,n) \lesssim (j,n+k)$ & particle $(i,n)$ is the time-$n$ ancestor of particle $(j,n+k)$ & \eqref{eq: ancestor}
\\
$(i,n) \lesssim_b (j,n+k)$ & the $b$th offspring of particle $(i,n)$ is the time-$(n+1)$ ancestor of particle $(j,n+k)$ & Sect.~\ref{sect: notation}
\\
$\zeta_{i,n+k}(n)$ & $\zeta_{i,n+k}(n)\in[N]$ is the index of the time-$n$ ancestor of the particle $(i,n+k)$ & \eqref{eq: zeta}
\\
$P_{i_0,n}^{i_k,n+k}$ & path (sequence of jumps) between particles $(i_0,n)$ and $(i_k,n+k)$, if $(i_0,n)\lesssim (i_k,n+k)$ & \eqref{eq: P}
\\
$\Nn_{i,n}(n+k)$ & $\Nn_{i,n}(n+k)\subseteq [N]$ is the set of time-$(n+k)$ descendants of particle $(i,n)$ & \eqref{eq: N}
\\
$\Nn_{i,n}^b(n+k)$ & $\Nn_{i,n}^b(n+k)\subseteq [N]$ is the set of time-$(n+k)$ descendants of the $b$th offspring of particle $(i,n)$ & \eqref{eq: Nb}
\\
$\rho a_N$ & jumps of size greater than $\rho a_N$ are called big jumps & Sect.~\ref{sect: bigjumps}
\\
$B_N$ & set of big jumps & \eqref{eq: BN1}, \eqref{eq: BN2}
\\
$\Sbf_N$ & set of times when the record is broken by a big jump & \eqref{eq: RN}
\\
$\hat\Sbf_N$ & times when the leader is surpassed by a big jump & \eqref{eq: RNhat}
\\
$T$ & time of the common ancestor of almost every particle at time $t$ & Sect.~\ref{sect: intro_result} \\
$T=T(\rho)$ & the last time before $t_1$ when a particle breaks the record with a big jump & \eqref{eq: T}
\\
$(N,T)$ & the leader (rightmost) particle at time $T$ &  Sect.~\ref{sect: heur_pic}
\\
$Z_i(s)$ & distance between the $i$th and the rightmost particle & \eqref{eq: Z}
\\
\midrule[0.08em]
  
\end{longtable}

Next, we list the events which appear throughout our main argument. We give a brief explanation of each event and refer to the equation where the event is defined. We also include short descriptions of the main results involving these events to give a summary of the major steps of the proof of Theorem~\ref{thm}. We write ``whp'' as shorthand for ``with high probability''.

\noindent
\begin{longtable}{ p{1cm} p{12.4cm} r } 
\toprule
\textbf{Event} & \textbf{Meaning} & \textbf{Def./Sect.} 
\\ 
\midrule[0.08em]
$\Aa_1$ & Almost the whole population is close to the leftmost particle at time $t$. & \eqref{eq: A1}
\\
$\Aa_2$ & The genealogy of the population at time $t$ is given by a star-shaped coalescent; there is a common ancestor at time $T\in[t_2,t_1]$. & \eqref{eq: A2}
\\
\midrule
& \centering $\Aa_1$ and $\Aa_2$ occur whp (Theorem~\ref{thm}) & 
\\
\midrule
$\Aa_3$ & Almost every particle at time $t$ descends from the leader at time $T\in[t_2,t_1]$.  & \eqref{eq: A3}
\\
$\Aa_4$ & Shortly after time $T$ no particle has a positive proportion of the population as descendants at time $t$. & \eqref{eq: A4}
\\
\midrule[0.08em]
& \centering If $\Aa_3$ and $\Aa_4$ occur whp then $\Aa_2$ occurs whp (Lemma~\ref{lemma: rewrite}) &
\\
\midrule
& \centering The event $\Aa_4$ occurs whp (Proposition~\ref{prop: A4}) &
\\
\midrule
& \centering The event $\Aa_1\cap\Aa_3$ occurs whp (Proposition~\ref{prop: A1A3}). This is shown using the events below. &
\\
\midrule[0.08em]
$\Bb_1$ & There is a leading tribe, descended from the leader at time $T\in[t_2,t_1]$, which is a significant distance from the other particles at time $t_1$. & \eqref{eq: B1}
\\
$\Bb_2$ & Particles which are not in the leading tribe at time $t_1$ have $o(N)$ descendants in total at time $t$. & \eqref{eq: B2}
\\
\midrule
& \centering $\Bb_1\cap\Bb_2 \subseteq \Aa_3$ (Lemma~\ref{lemma: A})&
\\
\midrule
$\Cc_1$ & A particle leads by a large distance compared to the second rightmost particle at some point in $[t_2+1,t_1]$. & \eqref{eq: C1}
\\
$\Cc_2$ & Particles far from the leader stay far behind or beat the leader by a lot. & \eqref{eq: C2}
\\
$\Cc_3$ & There is at most one big jump on a path of length $\ell_N$. & \eqref{eq: C3}
\\
$\Cc_4$ & Paths without big jumps move very little on the $a_N$ space scale. & \eqref{eq: C4}
\\
$\Cc_5$ & Two big jumps cannot happen at the same time. & \eqref{eq: C5}
\\
$\Cc_6$ & No big jumps happen at times very close to $t_2$ or $t_1$. & \eqref{eq: C6}
\\
$\Cc_7$ & The number of big jumps performed in $[t_4,t]$ is bounded above by a constant independent of $N$. & \eqref{eq: C7}
\\
\midrule
& \centering $\bigcap_{j=1}^7\Cc_j \subseteq \Bb_1\cap\Bb_2\cap\Aa_1 \subseteq \Aa_1\cap\Aa_3$ (Proposition~\ref{prop: B})& 
\\
\midrule
$\Dd_1$ & Same as $\Cc_2$ with different constants. & \eqref{eq: D1}
\\
$\Dd_2$ & In every short interval on the $\ell_N$ time scale, at least one big jump larger than a certain size occurs. & \eqref{eq: D2}
\\
$\Dd_3$ & In the first half of $[t_2,t_1]$ a big jump larger than a certain size occurs.  & \eqref{eq: D3}
\\
$\Dd_4$ & Shortly before time $t_2$, only jumps smaller than a certain size occur. & \eqref{eq: D4}
\\
$\Dd_5$ & During a short time interval, jumps of size in a certain small range do not happen. & \eqref{eq: D5}
\\
\midrule
& \centering $\bigcap_{j=2}^7\Cc_j \cap \bigcap_{i=1}^5\Dd_i \subseteq \Cc_1$ (Proposition~\ref{prop: C}) & 
\\
\midrule
& \centering The events $\Cc_2-\Cc_7$, $\Dd_1-\Dd_5$ all occur whp (Lemma~\ref{lemma: probC})& 
\\
\bottomrule
\end{longtable}

\subsubsection*{Acknowledgements}

MR would like to thank the Royal Society for funding his University Research Fellowship.
ZT is supported by a scholarship from the EPSRC Centre for Doctoral Training in Statistical Applied Mathematics at Bath (SAMBa), under the project EP/L015684/1.

\bibliographystyle{plain}
\bibliography{branching_ref}

\end{document}